\documentclass[pdflatex,sn-mathphys-num]{sn-jnl}

\usepackage{graphicx}
\usepackage{amsmath,amssymb,amsfonts}
\usepackage{amsthm}
\usepackage{xcolor}
\usepackage{textcomp}
\usepackage{manyfoot}
\usepackage{booktabs}
\usepackage{dsfont}
\usepackage{cleveref}

\newcommand{\pd}{\hfill $\Box$ \newline}
\newcommand{\rN}{\mathbb{R}}
\newcommand{\cN}{\mathbb{C}}
\newcommand{\zN}{\mathbb{Z}}
\newcommand{\nN}{\mathbb{N}}

\newcommand{\me}{\mathrm{e}}
\newcommand{\mi}{\mathrm{i}}
\newcommand{\md}{\mathrm{d}}

\newcommand{\mL}{\mathcal{L}}

\newcommand{\supp}{\mathrm{supp}}

\newcommand{\Id}{\mathrm{Id}}

\newcommand{\bx}{\mathbf{x}}
\newcommand{\bX}{\mathbf{X}}
\newcommand{\by}{\mathbf{y}}
\newcommand{\bY}{\mathbf{Y}}
\newcommand{\bu}{\mathbf{u}}
\newcommand{\bk}{\mathbf{k}}
\newcommand{\bv}{\mathbf{v}}
\newcommand{\bK}{\mathbf{K}}
\newcommand{\bracketeps}{{(\varepsilon)}}
\newcommand{\bT}{\mathbf{T}}

\theoremstyle{thmstyleone}
\newtheorem{theorem}{Theorem}[section]
\newtheorem{proposition}{Proposition}[section]

\newtheorem{lemma}{Lemma}[section]

\theoremstyle{thmstylethree}
\newtheorem{definition}{Definition}[section]

\theoremstyle{thmstyletwo}
\newtheorem*{remark}{Remark}

\numberwithin{equation}{section}

\begin{document}

\title[Pseudo-magnetic Fields and Effective Dynamics]{Pseudo-magnetic Fields and Effective Dynamics in Strained Honeycomb Structures}

\author[1]{\fnm{Chengyu} \sur{Zhang}}\email{zhang-cy23@mails.tsinghua.edu.cn}

\author[2]{\fnm{Borui} \sur{Miao}}\email{mbr@mail.tsinghua.edu.cn}

\author*[2,3]{\fnm{Yi} \sur{Zhu}}\email{yizhu@tsinghua.edu.cn}

\affil[1]{\orgdiv{Qiuzhen College}, \orgname{Tsinghua University}, \orgaddress{\city{Beijing}, \postcode{100084}, \country{China}}}

\affil[2]{\orgdiv{Yau Mathematical Sciences Center}, \orgname{Tsinghua University}, \orgaddress{\city{Beijing}, \postcode{100084}, \country{China}}}

\affil[3]{\orgname{Beijing Institute of Mathematical Sciences and Applications}, \orgaddress{\city{Beijing}, \postcode{101408}, \country{China}}}

\abstract{Strain offers an effective method for generating pseudo-magnetic fields in optical and acoustic materials, thereby enabling precise manipulation of wave propagation. In this article, we investigate wave packets spectrally localized near Dirac points in strained honeycomb-structured media and rigorously justify their long-time effective dynamics. We show that the envelope dynamics is governed by a two-dimensional Dirac equation with nontrivial gauge fields and prove that the associated two-scale ansatz approximates the exact wave evolution with error $O(\varepsilon)$ in $H^s$ for $0\le t\le \rho\varepsilon^{-1}$. Two difficulties distinguish this problem from standard wave-packet justifications. First, strain deforms the principal part of the wave operator, so the residual contains second-order differential terms that are not controlled by the unperturbed wave energy. Second, for a vanishing potential, the spectrum of the strained operator is not bounded away from zero, and a direct Duhamel estimate on the low-energy spectral subspace produces an apparent secular growth. We overcome the first difficulty by evolving with the strained operator and comparing regularized spectral projections for the strained and unperturbed operators through norm-resolvent estimates and functional calculus. For the second, we isolate the leading low-energy forced response through an explicit resolvent construction. Together, these results establish a rigorous continuum-wave theory of strain-induced pseudo-magnetic Dirac dynamics for slowly deformed honeycomb media, including perturbations of the principal part and the physically relevant zero-potential regime. More broadly, the spectral strategy may be useful for other linear systems with perturbations acting at the highest differential order.}

\keywords{strained honeycomb structures, pseudo-magnetic fields, Dirac points, wave packets, effective dynamics}

\pacs[MSC Classification]{35L05, 35B27, 35P05, 35Q41}

\maketitle

\section{Introduction} \label{Sec:Intro}
Graphene has attracted growing interest across many scientific fields because of its remarkable transport properties \cite{haldane2008photonic,hsu2016tci,neto2009graphene,wallace1947band}, originating from conical band crossings, termed ``Dirac points", in its band structure. It is well understood that these conical points are protected by honeycomb symmetry, which has motivated studies of artificial graphene, namely quantum-mechanical, photonic, and phononic analogs of graphene. A rigorous mathematical analysis of honeycomb structures was initiated in \cite{fefferman2012honeycomb}, where Fefferman and Weinstein proved the generic existence of Dirac points in honeycomb media. These results were later extended to general elliptic operators \cite{lee2019elliptic} and to the sub-wavelength regime \cite{ammari2020honeycomb}. One of the most important facts is that electrons behave as two-dimensional massless Dirac fermions in honeycomb media \cite{ablowitz2009conical}, and later studies in wave-packet dynamics justify the use of massless Dirac equations; see \cite{fefferman2014wave} for electrons, \cite{lee2019elliptic} for photonics, \cite{ammari2020highfrequency,ammari2024wavepackets} for phononics in the sub-wavelength regime, \cite{arbunich2016rigorous,xie2021wave} for nonlinear problems, and \cite{watson2018wavepackets} for inhomogeneous periodic media. These significant behaviors have prompted a wide range of applications.


A particularly striking way of manipulating Dirac waves is through strain. It is well understood that a nonuniform strain applied to a honeycomb structure can induce pseudo-magnetic fields near the Dirac points \cite{guinea2010strain,levy2010pseudomagnetic,manes2013generalized,amorim2016strains}. The associated pseudo-magnetic field is not a real magnetic field: it acts with opposite signs at the two time-reversal-related Dirac points and therefore preserves time-reversal symmetry \cite{hsu2020nanoscale}. Notably, such pseudo-magnetic fields can have remarkably large field strengths, and can be designed to be uniform on sufficiently large regions, resulting in highly degenerate pseudo-Landau levels; see \cite{guglielmon2021landau}, where the relation between the induced pseudo-magnetic fields and the imposed strain is computed explicitly.  Representative honeycomb and strained honeycomb structures are shown in \Cref{fig:honeycombStructure,fig:strainedHoneycombStructure}.

\begin{figure}[h]
	\centering
	\includegraphics[width=\linewidth]{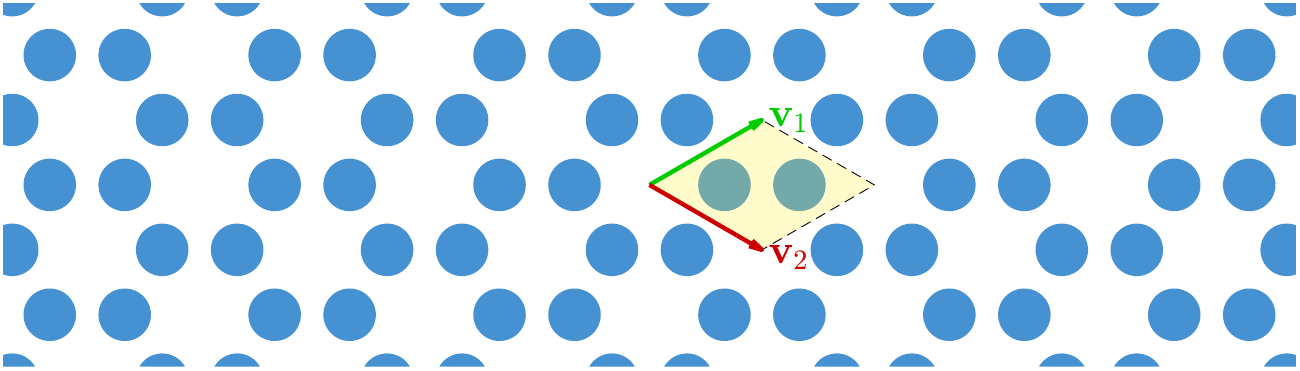}
	\caption{A honeycomb structure, with a fundamental cell of the honeycomb lattice shaded}
	\label{fig:honeycombStructure}
\end{figure}
\begin{figure}[h]
	\centering
	\includegraphics[width=\linewidth]{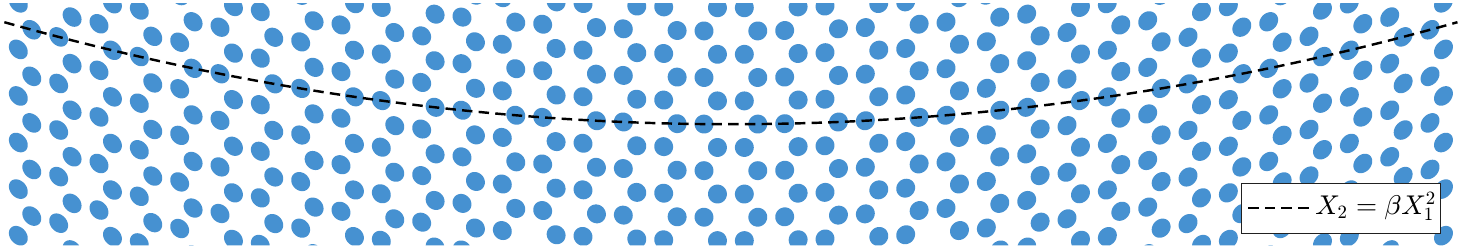}
	\caption{A strained honeycomb structure generated by the deformation $\bu(\bX) = (0,\beta X_1^2)$}
	\label{fig:strainedHoneycombStructure}
\end{figure}

The mechanism is particularly significant in photonic and acoustic artificial graphene: pseudo-magnetic fields also arise and remain strong in these settings, even though photons and phonons do not carry electric charge and do not couple directly to real magnetic fields. This significant phenomenon allows the realization of pseudo-Landau levels \cite{schomerus2013parity}, which enhances light-matter interaction through the large density of states, even in electrically neutral wave systems. Subsequent experiments have directly imaged pseudo-Landau levels and their $\sqrt{n}$ spacing \cite{barsukova2024direct}. Numerous experiments \cite{rechtsman2013pseudomagnetic,jamadi2020photoniclandau,salerno2015how,bellec2020supersymmetric} and applications \cite{umucalilar2011artificial,hafezi2011robust,fang2012realizing,fang2017generalized} of pseudo-magnetic fields in optics can also be found. Recent rigorous studies in discrete tight-binding settings justify macroscopic wave-packet models with strain-induced synthetic magnetic potentials \cite{bal2026macroscopic} and establish pseudo-magnetic localization with nearly flat Landau-level spectra in strained honeycomb lattices \cite{li2026pseudomagnetism}.

Despite these recent advances for discrete tight-binding models, a fully rigorous mathematical framework for the underlying strained continuum wave equation remains largely absent and is needed to put this physical picture on solid mathematical ground. The present article takes up this problem by developing a rigorous analysis of the underlying strained wave equation. As we proceed further, several fundamental difficulties in providing a complete justification of the predicted pseudo-magnetic effects will become apparent. We shall introduce some new tools and methods to overcome the challenges.

 The goal of this article is to give a rigorous derivation of the pseudo-magnetic effective dynamics induced by strain in honeycomb media. This formulation is motivated in particular by photonic and acoustic artificial graphene, where the relevant fields are naturally governed by wave equations rather than by Schr\"odinger dynamics. The wave setting also brings in several analytic features that are absent from the usual Schr\"odinger or lower-order perturbative frameworks, and these features will play an essential role in the proof. As in \cite{guglielmon2021landau}, introduce a deformation $\bu: \rN^2 \to \rN^2$, $\bT_\varepsilon(\by) = \by + \bu(\varepsilon \by)$  and a pair of honeycomb structured media $(A,V)$, see \Cref{DefHoneycombStructure}. We define the \textbf{strained honeycomb structure}
\begin{equation*}
	\begin{aligned}
		A^{(\varepsilon)}(\bx) &:= A(\mathbf{T}_\varepsilon^{-1}(\bx)), \\
		V^{(\varepsilon)}(\bx) &:= V(\mathbf{T}_\varepsilon^{-1}(\bx)).
	\end{aligned}
\end{equation*}
In this continuum-wave setting, with $V\ge0$ so that the associated elliptic operator is nonnegative, the photonic or acoustic fields in strained honeycomb media are modeled by the following wave equation
\begin{equation}
\label{WaveEquation}
	\frac{\partial^2 \psi}{\partial t^2} - \nabla_\bx \cdot \left( A^{(\varepsilon)}(\bx) \nabla_\bx \psi \right) + V^{(\varepsilon)}(\bx)\psi =0.
\end{equation}
 We prove that for initial data that are spectrally concentrated at a Dirac point $(\bK,E_D)$, see \Cref{DefDiracPoint}, the envelopes of the wave-packet satisfy the following Dirac equation with a gauge field:
\begin{equation}
\label{EffectiveDynamics}
	\mi\partial_T \boldsymbol{\alpha} = v\left[ (\mi \partial_1 + A_1)\sigma_1 + (-\mi\partial_2 - A_2)\sigma_2\right] \boldsymbol{\alpha} + W_{\mathrm{eff}}\sigma_0 \boldsymbol{\alpha},
\end{equation} 
where, after suitably rotating coordinates, we take $\mu\in\rN$ \cite{guglielmon2021landau}, and $A_1,A_2,W_{\mathrm{eff}}$ are given by
\begin{equation}
\label{DefPseudoElectromagneticFieldsI}
\begin{aligned}
	v &= \frac{\nu_{_F}}{2\sqrt{E_D}}, \\
	W_{\mathrm{eff}} &= -\frac{1}{2\sqrt{E_D}} \left(\xi \operatorname{Tr}(U\sigma_0) + \xi^\#\operatorname{Tr}(U\sigma_2)\right), \\
	A_1 &= - \frac{\mu}{\nu_{_F}}\mathrm{Tr}(U\sigma_3), \\
	A_2 &=  \frac{\mu}{\nu_{_F}}\mathrm{Tr}(U\sigma_1).
\end{aligned}
\end{equation}
Here, $U:= D \bu$ is its differential and the constants $\nu_{_F},\mu,\xi,\xi^\#$, depending only on the honeycomb structures $(A,V)$, are defined in \Cref{SecDPSO}.

To be more precise, we go back to the original space coordinate and consider the unitary transform
\begin{equation*}
	\begin{array}{rccc}
		\mathcal{U}_\varepsilon : & L^2(\rN^2_{\bx}) & \longrightarrow & L^2(\rN^2_{\by}), \\
		& f(\bx) & \longmapsto & J_\varepsilon(\by)^{-1/2} f(\bT_\varepsilon(\by)),
	\end{array}
\end{equation*}
where $M_\varepsilon(\by)$ is the inverse deformation gradient and $J_\varepsilon(\by)$ is its determinant:
\begin{equation*}
	M_\varepsilon(\by) := (D_\bx\bT_\varepsilon^{-1})(\bT_\varepsilon(\by)) = (D_\by\bT_\varepsilon(\by))^{-1},\qquad
	J_\varepsilon(\by):=\det M_\varepsilon(\by).
\end{equation*}
Since $D_\by\bT_\varepsilon(\by)=I+\varepsilon U(\varepsilon \by)$ is uniformly close to the identity for sufficiently small $\varepsilon>0$, we have $J_\varepsilon(\by)>0$.
Then, with $\varphi = \mathcal{U}_\varepsilon \psi$, equation \eqref{WaveEquation}, further imposed by initial data spectrally concentrated at a Dirac point, becomes
\begin{equation}
\label{DeformedGoverningEquation}
	\left \{\begin{aligned}
	\frac{\partial^2 \varphi}{\partial t^2}
	&- J_\varepsilon(\by)^{\frac{1}{2}}\!\nabla_\by\!\cdot\!
	\frac{M_\varepsilon(\by)A(\by)M_\varepsilon^{T}(\by)}{J_\varepsilon(\by)}
	\nabla_\by\!\left(J_\varepsilon(\by)^{\frac{1}{2}}\!\varphi\right)
	+ V(\by)\varphi = 0, \\
	\varphi(\by,0) &= \varepsilon J_\varepsilon(\by)^{-1/2}\!
	\left[ \alpha_{10}(\varepsilon\by)\Phi_1(\by)+\alpha_{20}(\varepsilon\by)\Phi_2(\by) \right], \\
	\frac{\partial\varphi}{\partial t}(\by,0) &= \varepsilon \mi \sqrt{E_D}J_\varepsilon(\by)^{-1/2}\!
	\left[ \alpha_{10}(\varepsilon\by)\Phi_1(\by)+\alpha_{20}(\varepsilon\by)\Phi_2(\by) \right].
\end{aligned}\right.
\end{equation}
Here $\Phi_1,\Phi_2$ are the eigenmodes associated with the Dirac point, see \Cref{DefDiracPoint}.

The approximate solution is
\begin{equation}
\label{TheAnsatz}
	\varphi_{\mathrm{eff}}(\by,t) = \varepsilon J_\varepsilon(\by)^{-1/2} \me^{\mi \sqrt{E_D}t} \left[ \alpha_1(\varepsilon \by,\varepsilon t)\Phi_1(\by) + \alpha_2(\varepsilon \by,\varepsilon t)\Phi_2(\by)\right],
\end{equation}
and $\boldsymbol{\alpha}:=(\alpha_1,\alpha_2)^T$ satisfies the prescribed effective equation \eqref{EffectiveDynamics}. As our main results, we prove in \Cref{MainTheorem1,MainTheorem2} that, for sufficiently small $\varepsilon>0$ and $V \ge 0$, the solution to \eqref{DeformedGoverningEquation} can be approximated by $\varphi_{\mathrm{eff}}$ over a large but finite time scale $t=O(\varepsilon^{-1})$ in the sense that
\begin{equation*}
	\sup_{t \in [0,\rho\varepsilon^{-1}]}\| \varphi(t) - \varphi_{\mathrm{eff}}(t)\|_{H^s} \lesssim \varepsilon.
\end{equation*}

Existing rigorous derivations of effective Dirac dynamics in honeycomb media are largely built around a fixed periodic operator. In the unperturbed Schr\"odinger, photonic, and phononic settings, the Floquet-Bloch decomposition of the background honeycomb operator remains the main object, and the effective Dirac equation cancels the leading resonant contribution associated with the two Bloch bands meeting at the Dirac point. In perturbed problems, the perturbations treated in this framework are typically lower-order and can be placed in the source term of the error equation. The leading near-Dirac contribution is cancelled by the effective dynamics. The additional source generated by the perturbation can still be controlled because the wave propagator contributes a factor $\left(\mL^{(0)}\right)^{-1/2}$, lowering the differential order by one \cite{xie2019wave}. Thus the unperturbed Floquet-Bloch decomposition and the wave energy associated with the background operator remain available. In contrast, the current work considers an elliptic operator
\[
-\nabla\cdot A^{(\varepsilon)}(x)\nabla+V^{(\varepsilon)}(x)
\]
which, after passing to strained coordinates, exhibits a second-order linear differential perturbation of the original honeycomb operator. This changes the principal part of the operator and introduces substantial challenges in obtaining rigorous error estimates. Indeed, the wave energy defined through the unperturbed operator is affected at the highest differential order and hence no longer represents the true energy of the strained system; see \Cref{SecMR}.

There is also an additional difficulty specific to the wave equation. For the usual Dirac wave-packet dynamics of a fixed periodic wave operator, the relevant spectral window is centered at a positive Dirac energy and is separated from zero, so the factor $\lambda^{-1/2}$ in the wave propagator is harmless. In some low-frequency wave-packet problems, the small Bloch frequency is tied to an additional physical parameter and can be absorbed into the corresponding scaling estimates. The present strained problem has a different low-energy obstruction: the physically natural case $V\equiv0$ places zero in the spectrum of $\mL^{(\varepsilon)}$. Although the wave packet is centered at the positive frequency $\sqrt{E_D}$, the residual source has a low-energy component with respect to $\mL^{(\varepsilon)}$. A direct Duhamel estimate therefore produces an artificial secular growth, which is insufficient to close the approximation on the long time scale considered here.

To overcome these challenges, we develop a spectral approach centered on the strained operator itself. Rather than treating the strained part as an external source and evolving the error only with the unperturbed honeycomb operator, we keep $\mL^{(\varepsilon)}$ as the generator of the wave dynamics. The relevant spectral windows of $\mL^{(\varepsilon)}$ are compared with those of $\mL^{(0)}$ through norm-resolvent estimates, regularized spectral projections, and functional calculus. This allows us to isolate the near-Dirac contribution, transfer it to the unperturbed Bloch representation where the effective Dirac equation cancels the leading resonant term, and control the remaining spectral components. The method is particularly suited to the present problem because the highest-order part of the perturbation is kept inside the principal operator rather than being moved to the source term in the error equation. We expect this strategy to be useful for other linear systems involving higher-order perturbations.

For the zero-potential case, we do not estimate the singular low-energy wave kernel directly. Instead, we introduce an auxiliary low-energy response that captures the leading forced dynamics on the low-energy spectral subspace. This shows that the secular growth produced by a direct Duhamel estimate is artificial: the oscillation at $\sqrt{E_D}$ cancels the leading low-energy behavior. This auxiliary construction closes the approximation on the long time scale $t=O(\varepsilon^{-1})$.

The rest of the article is organized as follows. \Cref{SecDPSO} serves as a preliminary, where we review the notions of Floquet-Bloch theory, the honeycomb structure and the Dirac points, and record some basic facts about the operator $\mL^\bracketeps$. In \Cref{SecMR}, we formulate the main setups and theorems and discuss the main analytical difficulties, followed by the proofs in \Cref{SecDFEE,SecCR,SecTS0}. To be more precise, in \Cref{SecDFEE} we decompose the equation satisfied by the error and handle the straightforward parts, while in \Cref{SecCR} we deal with the resonance source. These two sections together complete the proof of the first case $V\ge0,V \not \equiv
 0$. The second case $V \equiv 0$ is treated in \Cref{SecTS0}. \Cref{SecDDE} is devoted to discussing transformations of the effective Dirac equation and the derivation of Landau levels. In addition, the detailed computations of asymptotic expansions can be found in Appendices \ref{AppAEO} and \ref{AppAEEE}. Finally, we also include the Schr\"odinger equation case in Appendix \ref{AppSE}, for completeness.

The following notations and conventions are used in this article.
\begin{enumerate}
	\item We use bold letters such as $\bx,\bv,\bk,\bY,\boldsymbol{\alpha}$ et al. to represent values with multiple components. For example, $\bx=(x_1,x_2)^T$ are the coordinates in $\rN^2$, and $\bv_1,\bv_2$ are some specific vectors. $\boldsymbol{\alpha}$ is a vector-valued function with components $\alpha_1$ and $\alpha_2$. $\mathbf{0}$ is the zero element of $\rN^2$ or $\zN^2$.
	\item $C^\infty$ is the space of smooth functions. $C_b^\infty$ consists of smooth functions such that all of their derivatives are bounded. $C_c^\infty$ are smooth functions that are compactly supported. $\mathcal{S}(\rN^2)$ is the Schwartz space and $\mathcal{S}'(\rN^2)$ is the tempered distribution.
	\item For a Lebesgue measurable set $M \subset \rN^2$, we use $|M|$ to denote its Lebesgue measure.
	\item $L^2(\rN^2)$ and $H^s(\rN^2),W^{N,1}(\rN^2),s \in \nN$ are standard Sobolev spaces, and we also drop the base space $\rN^2$ for simplicity.
	\item In the rest of the article, $\nabla = (\partial_{y_1},\partial_{y_2})^T$ is the gradient operation. $\nabla \cdot$ is the divergence operation. Differentials with respect to the slow spatial variable $\bY$ will be denoted by $\nabla_\bY = (\partial_{Y_1},\partial_{Y_2})^T$. For a map $\mathbf{F} = (F_1,F_2):\rN^2 \to \rN^2$, we use $D\mathbf{F}$ to denote its differential $D\mathbf{F} = (\partial_j F_i)_{ij}$.
	\item If $S$ is a self-adjoint operator and $f$ is a Borel-measurable function on $\rN$, we use $f(S)$ to denote the Borel functional calculus.
	\item The Pauli matrices
		\begin{equation*}
			\sigma_0 = \begin{pmatrix}
				1&0 \\
				0&1
			\end{pmatrix}, \quad \sigma_1 = \begin{pmatrix}
				0&1 \\
				1&0
			\end{pmatrix}, \quad \sigma_2 = \begin{pmatrix}
				0&-\mi \\
				\mi&0
			\end{pmatrix}, \quad \sigma_3 = \begin{pmatrix}
				1&0 \\
				0&-1
			\end{pmatrix}.
		\end{equation*}
	\item The notation ``$F \lesssim_\beta G$" means that there exists a constant $C_\beta$ depending only on some parameter $\beta$ such that $F \le C_\beta G$ holds for all the other parameters appearing in the expressions. $F \approx_\beta G$ if and only if both $F \lesssim_\beta G$ and $G \lesssim_\beta F$ hold.
	\item The Einstein summation convention is used throughout. The product $A_iB_i$ should be understood as the natural product of the factors, namely multiplication when they are numbers, the dot product when they are vectors, and the Frobenius product ($A:B = \operatorname{Tr}(A^TB)$) when they are matrices.
\end{enumerate}

\section{Dirac Points and the Strained Operator}
\label{SecDPSO}
In this section, we give some preliminary facts about honeycomb structures. First, we present an overview of the Floquet-Bloch theory as the basic tool to clarify the notion of energy bands, dispersion relations, and Dirac points. We refer to \cite{reed1978iv,eastham1973spectral,kuchment2012floquet,kuchment2016overview} for several authoritative references. Next we introduce the honeycomb structures and review from \cite{fefferman2012honeycomb,fefferman2014wave,lee2019elliptic} the existence of Dirac points and the corresponding Bloch eigenfunctions. Basic facts about the strained operator $\mL^\bracketeps$ are displayed in the last subsection.

\subsection{Floquet-Bloch theory}
The Floquet-Bloch theory is used to decompose a periodic system into a direct integral of systems so that each fiber admits an orthonormal basis of Bloch eigenfunctions. Here we work only in $\rN^2$, although the whole theory applies to $\rN^n$.

Suppose a lattice $\Lambda$ is generated by two linearly independent vectors $\bv_1$ and $\bv_2$ with a fundamental cell $\Omega :=\left\{ \theta_1 \bv_1+\theta_2 \bv_2 : 0\le \theta_i < 1, \text{ } i=1,2\right\}.$ The dual (or reciprocal) lattice $\Lambda^*$ of $\Lambda$ is defined to be generated by $\bk_1,\bk_2$ satisfying $\bv_i \cdot \bk_j = 2\pi \delta_{ij}$. The (first) \textbf{Brillouin zone} $\mathbb{B}$ is, by definition, the closure of the set consisting of vectors that are closer to $\mathbf{0}$ than any other translation (by $\Lambda^*$) of themselves. The Brillouin zone is equivalent to a fundamental cell of $\Lambda^*$, up to $\Lambda^*$-translations. For any $\Lambda^*$-periodic integrand, integrals over any fundamental cell of $\Lambda^*$ can be taken over $\mathbb{B}$ instead.

We next introduce the function spaces of $\mathbf{k}$-quasi-periodic functions:
\begin{equation*}
	L^2_{\mathbf{k},\Lambda} = \left\{ f \in L^2_{loc}(\rN^2) : f(\mathbf{x}+\mathbf{v})=\me^{\mi \bk \cdot \bv}f(\mathbf{x}),\forall \mathbf{x} \in \rN^2,\forall \mathbf{v} \in \Lambda\right\}.
\end{equation*}
When there is no ambiguity, we also omit the subscript $\Lambda$. Equipped with $L^2(\Omega)$-inner products, they are made to be Hilbert spaces. In the same way, we can introduce the Sobolev spaces $H^s_{\bk,\Lambda}$.

The essential ingredient is a decomposition of $L^2(\rN^2)$ into $L_\bk^2$'s, through the so-called \textbf{Bloch transform} 
\begin{equation}
	\mathcal{B}f(\bx;\bk) = \frac{1}{|\mathbb{B}|^\frac{1}{2}} \sum_{\bv \in \Lambda} \me^{\mi \bk \cdot \bv} f(\bx - \bv),\quad \bx \in \Omega,\bk \in \mathbb{B},
\end{equation} 
which is initially defined on $\mathcal{S}(\rN^2)$, into the space $L^2(\mathbb{B};L^2_\bk)$ and then extends unitarily to the whole $L^2(\rN^2)$. The space $L^2(\mathbb{B};L^2_\bk)$ consists of $L^2$-sections of the bundle $\bigsqcup_{\bk \in \mathbb{B}} L^2_\bk \to \mathbb{B}$, whose elements can also be regarded as functions of both $\bx$ and $\bk$. If $f \in L^2(\mathbb{B};L^2_\bk)$, for fixed $\bk \in \mathbb{B}_h$ we denote by $f(\bk)$ the asserted $L^2_\bk$-vector by $f$ at $\bk$. The Bloch transform realizes $L^2(\rN^2)$ as a direct integral (see \cite{reed1978i} for the definition) of $L^2_\bk$'s, meaning that
\begin{equation}
\label{BlochInverse}
	L^2(\rN^2) = \int_{\bk \in \mathbb{B}}^\oplus L^2_\bk, \qquad f(\bx) = \frac{1}{|\mathbb{B}|^\frac{1}{2}}\int_\mathbb{B} \mathcal{B}f(\bx;\bk)\md \bk, \quad \forall f \in L^2(\rN^2).
\end{equation}

Let $H = -\nabla_\bx \cdot A(\bx) \nabla_\bx + V(\bx)$ be a second-order elliptic differential operator, with coefficients $A$ and $V$ being smooth and $\Lambda$-periodic. Conjugating by $\mathcal{B}$ we can determine what $H$ do to elements in $L^2(\mathbb{B};L^2_\bk)$:
\begin{equation*}
	\left( \mathcal{B}H\mathcal{B}^{-1} g \right)(\bk) =H \left( g(\bk)\right).
\end{equation*}
Roughly speaking, each $L^2_\bk$ is invariant under $H$, whose restrictions to these spaces are denoted by $H(\bk)$. Each $H(\bk)$ is elliptic on a bounded region with $\bk$-quasi-periodic boundary condition, resulting in compact resolvents, and thus has a discrete spectrum and admits a Hilbert-Schmidt style decomposition
\begin{equation*}
	H(\bk)\Phi_b(\bk) = E_b(\bk)\Phi_b(\bk).
\end{equation*}
Here, $\{ \Phi_b(\bk) \}_{b=1}^\infty$ is some orthonormal basis of $L^2_\bk$, called \textbf{Bloch modes}, and $E_b(\bk)$ are arranged in a non-decreasing order $E_1(\bk) \le E_2(\bk) \le \cdots \le E_b(\bk) \le \cdots$.

As $\bk$ varies in $\mathbb{B}$, each $E_b$ sweeps out a surface, so $E_b$'s are often called  \textbf{Bloch bands}. The following regularity result of Bloch bands is proved in the appendix of \cite{fefferman2014wave}:
\begin{proposition}
\label{LipschitzContinuity}
	The Bloch bands $\{ E_b \}_{b=1}^\infty $ of such an $H$ are Lipschitz continuous.
\end{proposition}

Until now, we have completely ``diagonalized" the operator $H$, and hence its spectrum is the union of the numbers appearing on the diagonal. Moreover, for $f \in L^2(\rN^2)$ each $\mathcal{B}f(\bk)$ can be expanded along the basis $\{ \Phi_b(\bk) \}_{b=1}^\infty$. Substituting this expansion back into \eqref{BlochInverse} we obtain the so-called Floquet-Bloch decomposition
\begin{equation*}
	f(\bx) = \frac{1}{|\mathbb{B}|^{\frac{1}{2}}} \sum_{b=1}^\infty \int_\mathbb{B} \langle \Phi_b(\bk),\mathcal{B}f(\bk)\rangle_{L^2_\bk} \Phi_b(\bx;\bk)\md \bk.
\end{equation*}
Bringing in the definition of the Bloch transform, we find that
\begin{equation}
\label{AbbreviationOfInnerProductAgainstPhi}
	\langle \Phi_b(\bk),f\rangle_{L^2(\rN^2)} := \int_{\rN^2} \overline{\Phi_b(\bx;\bk)}f(\bx)\md \bx = |\mathbb{B}_h|^{\frac{1}{2}} \langle \Phi_b(\bk),\mathcal{B}f(\bk)\rangle_{L^2_\bk},
\end{equation}
and hence obtain the most commonly seen form of Floquet-Bloch decomposition \eqref{FBDecomposition2}.

We summarize the Floquet-Bloch theory as follows.
\begin{proposition}
	The spectrum of $H$ is the union of all $E_b(\bk)$'s: $\sigma(H) = \bigcup_{\bk \in \mathbb{B}} \sigma(H(\bk)) = \bigcup_{\bk \in \mathbb{B}}\bigcup_{b=1}^\infty \{E_b(\bk)\}$. Moreover, the set of functions
		$\{ \Phi_b(\bk) \}_{\bk \in \mathbb{B},b \in \zN^+}$ is complete in $L^2(\rN^2)$, in the sense that each $f \in L^2(\rN^2)$ can be written in the form
	\begin{equation}
	\label{FBDecomposition2}
			f = \frac{1}{|\mathbb{B}|} \sum_{b=1}^\infty \int_\mathbb{B} \langle \Phi_b(\bk),f\rangle_{L^2(\rN^2)} \Phi_b(\bk)\md \bk,
	\end{equation}
	where the series converges in the $L^2(\rN^2)$ sense.
\end{proposition}

\subsection{Honeycomb Structure}
The honeycomb lattice $\Lambda_h$ is a triangular lattice generated by $\bv_1 \ne 0,\bv_2=-R^{-1}\bv_1$, with $\bk_1,\bk_2$ denoting their dual basis, and where $R$ is the $2\pi/3$ clockwise rotation matrix
\begin{equation*}
	R = \begin{pmatrix}
		-\frac{1}{2} & \frac{\sqrt{3}}{2} \\[7pt]
		-\frac{\sqrt{3}}{2} & -\frac{1}{2}
	\end{pmatrix}.
\end{equation*}
In this case, the Brillouin zone $\mathbb{B}_h$ is a regular hexagon in $\rN^2$, whose vertices will play an important role. The six vertices of $\mathbb{B}_h$ are split into two $C_3$-orbits represented by $\bK=\frac{1}{3}(\bk_1-\bk_2)$ and $\bK'=-\bK$. The vertices in the same orbit differ only by a $\Lambda^*$-translation, and moreover, the discussion of $\bK$ and $\bK'$ is analogous, so we will focus on $\bK$ for simplicity.

To define a honeycomb structured material, we first introduce some operators acting on function spaces. They are:
\begin{enumerate}
	\item \textbf{parity symmetry}: $\mathcal{P}[f](\by) = f(-\by),\forall \by \in \rN^2$.
	\item \textbf{complex conjugation}: $\mathcal{C}[f](\by) = \overline{f(\by)},\forall \by \in \rN^2$.
	\item \textbf{translation}: $\tau_\bv [f](\by) = f(\by-\bv),\bv \in \Lambda_h,\forall \by \in \rN^2$.
	\item and \textbf{rotation operator}: $\mathcal{R}[f](\by) = f(R^{-1}\by),\forall \by \in \rN^2.$
\end{enumerate}
With this, we can define the honeycomb structured media in a brief way.
\begin{definition}
\label{DefHoneycombStructure}
	A pair of honeycomb structured media is a dual of material weight $A(\by)$ (a $2 \times 2$-Hermitian-matrix-valued function that is uniformly elliptic) and a potential $V(\by)$ (a real-valued function), satisfying: \begin{enumerate}
		\item $A$ is $\mathcal{PC}$-symmetric, and $V$ is $\mathcal{P}$-symmetric, meaning that $\mathcal{PC}[A]=A$, $\mathcal{P}[V]=V$.
		\item $A$ satisfies $\mathcal{R}[A]=R^*AR$ and $V$ is $\mathcal{R}$-symmetric.
		\item both $A$ and $V$ are periodic w.r.t. $\Lambda_h$.
	\end{enumerate}
\end{definition}

Now let $(A,V)$ be a pair of honeycomb structured media, introduce a second-order elliptic operator $\mL = -\nabla \cdot A(\by) \nabla + V(\by)$. With the above properties, one can verify the following symmetries about $\mL$:
\begin{equation*}
	[\mathcal{PC,L}] = [\mathcal{R,L}] = [\tau_\bv,\mL] = 0,\quad \forall \bv \in \Lambda_h,
\end{equation*}
which, together with some mild non-degenerate condition, shall lead to the generic occurrence of a Dirac point.

\subsection{Dirac points}
Dirac points are momentum-energy pairs where two Bloch bands touch conically. Here we give a specific definition of Dirac points, taken from \cite{fefferman2012honeycomb,lee2019elliptic}. Put $\tau = \me^{\mi \frac{2\pi}{3}}$, which is a generator of eigenvalues of the operator $\mathcal{R}$.
\begin{definition}
\label{DefDiracPoint}
 Suppose $\bK^*$ is a vertex of the Brillouin zone $\mathbb{B}_h$. A pair $(\mathbf{K}^*,E_D) \in \mathbb{B} \times \rN$ of quasi-momentum and energy is called a Dirac point (of $\mL$, defined above), if there exists $b_*$ such that the two Floquet-Bloch band functions
	\begin{equation*}
		\bk \mapsto E_{b_*}(\bk) \text{  and  } \bk \mapsto E_{b_*+1}(\bk)
	\end{equation*}
	satisfy the following property. 
	\begin{enumerate}
		\item $E_D$ is an eigenvalue of multiplicity $2$ of $\mL$ acting on $L^2_{\bK^*}$.
		\item The eigenspace is generated by normalized $\Phi_1$ and $\Phi_2$, satisfying $\mathcal{R}\Phi_1 = \tau [\Phi_1], \Phi_2 = \mathcal{PC}[\Phi_1]$.
		\item The following non-degenerate condition holds, provided that a suitable phase is chosen:
			\begin{equation}
			\label{FermiVelocity}
				\nu_{_F} := \frac{1}{2} \overline{\langle \Phi_1,\mathcal{A}\Phi_2\rangle_{L^2_{\bK^*}}} \cdot \begin{pmatrix}
					1 \\
					\mi
				\end{pmatrix} > 0,
			\end{equation}
			where
			\begin{equation}
			\label{DefmathcalA}
				\mathcal{A}f(\by) := \frac{1}{\mi}\Big[ A(\by)(\nabla f)(\by) + \nabla \cdot (A(\by)f(\by)) \Big].
			\end{equation}
	
		\item There exists a $k_0>0$ and two Lipschitz continuous functions $e_1,e_2$ defined for $|\boldsymbol{\kappa}|<k_0$, such that
		\begin{equation*}
			|e_i(\boldsymbol{\kappa})|\le C|\boldsymbol{\kappa}|, \quad i=1,2 \text{ for some } C>0.
		\end{equation*}
		\begin{equation*}
		\begin{aligned}
			E_{b_*}(\mathbf{K}^*+\boldsymbol{\kappa}) = E_D - \nu_{_F}|\boldsymbol{\kappa}|(1+e_1(\boldsymbol{\kappa})), \\
			E_{b_*+1}(\mathbf{K}^*+\boldsymbol{\kappa}) = E_D + \nu_{_F}|\boldsymbol{\kappa}|(1+e_2(\boldsymbol{\kappa})).
		\end{aligned}
		\end{equation*}
		 for all $|\boldsymbol{\kappa}|<k_0$.
		 \end{enumerate}
		 
		 If $b_*$ is as described in this definition, we often tag $b_* =-$ and $b_*+1=+$.
\end{definition}
The importance of \cite{fefferman2012honeycomb} and \cite{lee2019elliptic} lies in that they proved that the conditions in the above definition are enjoyed by a generic honeycomb structure, and hence guarantee the existence of Dirac points. We do not repeat the results here. 

We list two more conclusions. The first one is due to the Lipschitz continuity of the Bloch bands (\Cref{LipschitzContinuity}) and the second is proved in \cite{fefferman2014wave}.
\begin{proposition}
\label{EnergyLevelNearED}
	Let $(A,V)$ be a pair of honeycomb structured media and $(\bK,E_D)$ is a Dirac point as in \Cref{DefDiracPoint}. Then, by suitably shrinking $k_0>0$, we can find a $\delta>0$ and a $B \in \nN$, such that the following statements hold.
	\begin{enumerate}
		\item For $E_b(\bk)$ to lie in $(E_D-2\delta,E_D+2\delta)$, we have
	\begin{equation*}
		\text{either } b = \pm, \text{ or } \left( |\bk-\bK| \ge k_0 \text{ and } b \le B \right).
	\end{equation*}
		\item For $0 < |\boldsymbol{\kappa}| < k_0$, one can choose the Bloch modes on $b= \pm$ possessing the following expansions
	\begin{equation}
	\label{PhikNearK}
    \Phi_\pm(\by;\bK+\boldsymbol{\kappa}) = \frac{\me^{\mi \boldsymbol{\kappa}\cdot \by}}{\sqrt{2}}\left[ \frac{\kappa_1+\mi\kappa_2}{|\boldsymbol{\kappa}|} \Phi_1(\by) \pm \Phi_2(\by) + \Phi_{R,\pm}(\by;\boldsymbol{\kappa}) \right],
	\end{equation}
	where $\boldsymbol{\kappa}=(\kappa_1,\kappa_2)$ and $\left \| \Phi_{R,\pm}(\boldsymbol{\kappa})\right\|_{H^2_\bK} \le C |\boldsymbol{\kappa}|$ for some $C$ independent of $\boldsymbol{\kappa}$.
	\end{enumerate}
\end{proposition}

In the rest of this article, we always assume that $(A,V)$ is a pair of honeycomb structured media described by \Cref{DefHoneycombStructure}, and $(\bK,E_D)$ is a Dirac point of the operator $\mL:= -\nabla \cdot(A(\by) \nabla) + V(\by)$ described as in \Cref{DefDiracPoint}.

\subsection{Bifurcation Matrices}
We compute in advance some expressions for later use. The following ones are proved in \cite{lee2019elliptic}.
\begin{proposition}
\label{MathcalAMatrix}
	Let $\Phi_1$ and $\Phi_2$ be as in \Cref{DefDiracPoint}.
	\begin{enumerate}
		\item $\langle \Phi_i,\Phi_j \rangle_{L^2_\bK}=\delta_{ij},$ for $i,j=1,2,$.
		\item For the operator $\mathcal{A}$ defined by \eqref{DefmathcalA}, we have
			\begin{equation}
			\begin{gathered}
	\langle \Phi_1,\mathcal{A}\Phi_1 \rangle_{L^2_\bK} = \mathbf{0},\quad 
	\langle \Phi_1,\mathcal{A}\Phi_2 \rangle_{L^2_\bK} = \nu_{_F} \begin{pmatrix}
		1 \\
		\mi
	\end{pmatrix}, \\
	\langle \Phi_2,\mathcal{A}\Phi_1 \rangle_{L^2_\bK} = \nu_{_F} \begin{pmatrix}
		1 \\
		-\mi
	\end{pmatrix}, \quad 
	\langle \Phi_2,\mathcal{A}\Phi_2 \rangle_{L^2_\bK} = \mathbf{0},
			\end{gathered}
			\end{equation}
			where $\nu_{_F}$ is defined in \eqref{FermiVelocity}.
	\end{enumerate}
\end{proposition}

In this article, we need to introduce one more operation and compute the similar expressions for it. Define $\mathfrak{A}f$ to be a $2 \times 2$-matrix valued function whose entries are given by
\begin{equation}
\label{SymmetricExpressionFrakA}
	(\mathfrak{A}f)_{ij}(\by): = \partial_{y_{_l}}(a_{li}\partial_{y_j} f)(\by)+\partial_{y_j}(a_{il}\partial_{y_{_l}} f)(\by). 
\end{equation}
Equivalently, using tensor notation, we have
\begin{equation}
\label{CoordinateFreeExpressionFrakA}
	\mathfrak{A}f= \nabla\cdot A \otimes \nabla f + A^T\text{Hess} f + D(A\nabla f).
\end{equation}
see Appendix \ref{AppAEO}. In view of the expressions, it is obvious that, given any constant real matrix $M$, the operation sending $f \in L^2_\bK$ to $\operatorname{Tr}(M\mathfrak{A}f) \in L^2_\bK$ is self-adjoint. We group the identities in the next proposition.
\begin{proposition}
\label{MathfrakAMatrix}
	Let $\Phi_1$ and $\Phi_2$ be as in \Cref{DefDiracPoint}. Then
	\begin{equation}
	\begin{gathered}
		\langle \Phi_1,\mathfrak{A}\Phi_1 \rangle_{L^2_\bK} = \langle \Phi_2,\mathfrak{A}\Phi_2 \rangle_{L^2_\bK} = \xi \sigma_0 + \xi^\#\sigma_2, \\
		\langle \Phi_1,\mathfrak{A}\Phi_2 \rangle_{L^2_\bK} = \overline{\langle \Phi_2,\mathfrak{A}\Phi_1 \rangle_{L^2_\bK}} = \mu(\sigma_3 - \mi \sigma_1).	
	\end{gathered}
	\end{equation}
	Here, 
	\begin{equation}
    \label{DefOfMu}
		\mu= \frac{1}{4} \langle \Phi_1,\mathfrak{A}\Phi_2 \rangle : (\sigma_3+\mi\sigma_1), \quad \xi = \frac{1}{2} \mathrm{Tr}(\langle \Phi_1,\mathfrak{A}\Phi_1 \rangle) \le 0,
	\end{equation}
	and if $A(\bx)=a(\bx)\Id$ is isotropic, $\xi^\#=0$.
\end{proposition}

\begin{proof} From a direct computation, we can verify that
\begin{equation*}
	\mathcal{R}[\mathfrak{A}f]=R^{-1}(\mathfrak{A}\mathcal{R}[f])R, \quad \mathcal{PC}[\mathfrak{A}f] = \mathfrak{A}\mathcal{PC}[f],
\end{equation*}
for all $f \in H^2_\bK$. Return to the inner product (omitting the subscript)
\begin{equation*}
	\begin{aligned}
		\langle \Phi_k,\mathfrak{A}\Phi_l \rangle &= \langle \mathcal{R}[\Phi_k],\mathcal{R}[\mathfrak{A}\Phi_l] \rangle \\
		&= R^{-1} \langle  \mathcal{R}[\Phi_k],\mathfrak{A}\mathcal{R}[\Phi_l] \rangle R \\
		&= \overline{\tau_k}\tau_l R^{-1} \langle  \Phi_k,\mathfrak{A}\Phi_l \rangle R.
	\end{aligned}
\end{equation*}
This identity indicates that $\langle \Phi_k,\mathfrak{A}\Phi_l \rangle$ is a $\tau_k\overline{\tau_l}$-eigenvector of the mapping $M \mapsto R^{-1}MR$.
\begin{lemma}
	The linear transformation $\mathfrak{R}:M \mapsto R^{-1}MR$ has: 
	\begin{enumerate}
		\item eigenvalue $\tau$, with the corresponding eigenspace spanned by $\sigma_3+\mi \sigma_1$.
		\item eigenvalue $\bar{\tau}$, with the corresponding eigenspace spanned by $\sigma_3 -\mi \sigma_1$.
		\item eigenvalue $1$, with the corresponding eigenspace spanned by $\sigma_0$ and $\sigma_2$.
	\end{enumerate}
\end{lemma}
So we can conclude on this lemma that
\begin{enumerate}
	\item when $k=1,n=2$, $\tau_1\overline{\tau_2} = \me^{\frac{4}{3}\pi \mi}$, hence
$\langle \Phi_1,\mathfrak{A}\Phi_2 \rangle  = \mu(\sigma_3 - \mi \sigma_1)$.
	\item when $k=2,n=1$, from self-adjointness
	$\langle \Phi_2,\mathfrak{A}\Phi_1 \rangle = \bar{\mu} (\sigma_3 +\mi \sigma_1)$.
	\item when $k=n=1$, we have $\langle \Phi_1,\mathfrak{A}\Phi_1 \rangle = \xi \sigma_0 + \xi^\# \sigma_2$ for some purely imaginary $\xi^\#$.
	\item when $k=n=2$, by $\mathcal{PC}$-symmetry $\langle \Phi_2,\mathfrak{A}\Phi_2 \rangle = \langle \mathcal{PC}[\Phi_1],\mathfrak{A}\mathcal{PC}[\Phi_1] \rangle = \langle \Phi_1,\mathfrak{A}\Phi_1 \rangle$ .
\end{enumerate}
For isotropic $A(\bx)$, $\mathfrak{A}f$ is always symmetric, forcing $\xi^\#$ to be $0$.
\end{proof}

\subsection{The strained operator}
To end this section, we prepare some basic facts and asymptotic expansion properties of the perturbed operator $\mL^\bracketeps$.
    
 Define $\bT_\varepsilon :\rN^2 \to \rN^2$ by
\begin{equation*}
	\bT_\varepsilon (\by)=\by+\bu(\varepsilon \by),
\end{equation*}
for some deformation $\bu:\rN^2 \to \rN^2$ such that $U:=D_\bY\bu \in C_b^\infty$. In the rest of this article we omit the subscript of $\bT_\varepsilon$. The operator of main concern is
\begin{equation}
\begin{gathered}
	\widetilde{\mL^\bracketeps}:= -J_\varepsilon(\by)\nabla \cdot \frac{M_\varepsilon(\by)A(\by)M_\varepsilon^T(\by)}{J_\varepsilon(\by)} \nabla + V(\by), \\
	M_\varepsilon(\by) := (D_\bx\bT^{-1})(\bT(\by)) = (D_\by\bT(\by))^{-1},\qquad J_\varepsilon(\by):=\det M_\varepsilon(\by).
\end{gathered}
\end{equation}
The effect of change of coordinates also changes the Lebesgue measure in $\rN^2$, making $\widetilde{\mL^\bracketeps}$ no longer self-adjoint. Introduce the weighted $L^2$ space $L^2_\varepsilon (\rN^2)$ whose inner product is defined by
\begin{equation*}
	\langle u,v \rangle_{L^2_\varepsilon}: = \int_{\rN^2} \overline{u(\by)}v(\by)J_\varepsilon(\by)^{-1} \md \by.
\end{equation*}
Then, it is easy to verify that this $\widetilde{\mL^\bracketeps}$ is self-adjoint on this weighted space. But due to the nice behavior of the weight $J_\varepsilon(\by)$, for small $\varepsilon$, $L^2(\rN^2) = L^2_\varepsilon(\rN^2)$ as sets with equivalent norms.

To recover self-adjointness in regular $L^2(\rN^2)$, we further introduce the unitary transform
\begin{equation*}
	\begin{array}{rccc}
		\mathcal{J}_\varepsilon : & L^2_\varepsilon(\rN^2_\by) & \longrightarrow & L^2(\rN^2_\by), \\
		& u(\by) & \longmapsto & J_\varepsilon(\by)^{-1/2}u(\by),
	\end{array}
\end{equation*}
then
\begin{equation*}
	\mL^\bracketeps : = \mathcal{J}_\varepsilon \widetilde{\mL^\bracketeps} \mathcal{J}^{-1}_\varepsilon
\end{equation*}
is now a self-adjoint operator on $L^2(\rN^2)$. Together with the change of coordinate $\bT_\varepsilon$, they give the unitary transform $\mathcal{U}_\varepsilon : L^2(\rN^2) \to L^2(\rN^2)$ introduced in \Cref{Sec:Intro}.

We first collect the basic facts in the following proposition, whose proof is postponed to Appendix \ref{AppAEO}.
\begin{proposition}
\label{AsymptoticExpansionOfTheOperator}
	For sufficiently small $\varepsilon>0$, $\mL^{(\varepsilon)}$ is a self-adjoint operator on $L^2(\rN^2)$ with domain $H^2(\rN^2)$. Moreover, the following asymptotic expansion holds for all $f \in H^2(\rN^2)$:
\begin{equation}
\label{AsymptotcExpansionOfLeps}
	\mL^{(\varepsilon)}f (\by) = \mL^{(0)}f + \varepsilon \operatorname{Tr}\left(U(\varepsilon \by)\mathfrak{A}f(\by)\right) + \varepsilon^2 \mL_r f(\by),
\end{equation}
where, $\mathfrak{A}$ is defined by \eqref{SymmetricExpressionFrakA}, and $\mL_r$ is a second order differential operator whose coefficients are all in $C_b^\infty$, with derivative bounds independent of $\varepsilon$. In the rest of the article, we put $\mL_R f: = \operatorname{Tr}\left(U(\varepsilon \by)\mathfrak{A}f(\by)\right) + \varepsilon \mL_r f(\by)$ so that 
\begin{equation}
	\mL^\bracketeps = \mL^{(0)} + \varepsilon \mL_R.
\end{equation}
\end{proposition}
From this we see that $\mL^\bracketeps$ is a second-order perturbation of $\mL^{(0)}$, which causes the main technical difficulty in the proof. See the remarks in \Cref{SecMR}.

To estimate the Sobolev norms we need to recall the elliptic regularity, which says that the values $\| f \|_{H^2}$ and $\left\| \mL^{(0)}f \right\|_{L^2}+\|f\|_{L^2}$ are comparable. This observation should also extend to $\mL^{(\varepsilon)}$. Here we prove a stronger version, showing that this equivalence is uniform in $\varepsilon$. 

\begin{proposition}
\label{UniformEllipticRegularity} 
	For sufficiently small $\varepsilon>0$, the following inequality holds for some $C_N>0$ independent of $\varepsilon$:
	\begin{equation}
		C_N^{-1}\| f \|_{H^{2N}} \le \left\| \left( \mL^{(\varepsilon)}\right)^N f \right\|_{L^2} + \| f\|_{L^2} \le C_N\| f \|_{H^{2N}},\quad \forall f \in H^{2N}(\rN^2).
	\end{equation}
\end{proposition}
\begin{proof}
	We prove the case $N=1$. By elliptic regularity for $\mL^{(0)}$ the above inequality holds for some $c>0$ and $\varepsilon =0$. Using \Cref{AsymptoticExpansionOfTheOperator} we find that
	\begin{equation*}
	\begin{aligned}
		\left\| \mL^{(0)}f \right\|_{L^2} + \|f\|_{L^2} - \tilde{C}\varepsilon \|f \|_{H^2} &\le \left\| \mL^{(\varepsilon)}f  \right\|_{L^2} +\|f\|_{L^2} \\
		 &= \left\| \mL^{(0)}f + \varepsilon \operatorname{Tr}\left(U(\varepsilon \cdot)\mathfrak{A}f\right) + \varepsilon^2 \mathcal{L}_rf \right\|_{L^2} +\|f\|_{L^2}\\
		 &\le \left\| \mL^{(0)}f \right\|_{L^2} +\|f\|_{L^2} + \tilde{C}\varepsilon \|f \|_{H^2}.
	\end{aligned}
	\end{equation*}
	Hence, the proposition follows easily, provided that  $\varepsilon >0$ is small enough. Iterating this procedure, we generalize this result to any Sobolev norm of even order.
\end{proof}

\section{Main Results}
\label{SecMR}
Our main object is the following wave equation 
\begin{equation}
\label{ConcernedWaveEquation}
\left \{\begin{aligned}
	\frac{\partial^2 \varphi}{\partial t^2} &+\mL^\bracketeps\varphi = 0, \\
	\varphi(\by,0) &= \varepsilon J_\varepsilon(\by)^{-1/2} \left[ \alpha_{10}(\varepsilon\by)\Phi_1(\by)+\alpha_{20}(\varepsilon\by)\Phi_2(\by) \right], \\
	\frac{\partial\varphi}{\partial t}(\by,0) &= \varepsilon \mi \sqrt{E_D} J_\varepsilon(\by)^{-1/2} \left[ \alpha_{10}(\varepsilon\by)\Phi_1(\by)+\alpha_{20}(\varepsilon\by)\Phi_2(\by) \right].
\end{aligned}\right.
\end{equation}
The candidate is of the form
\begin{equation}
\label{TheAnsatzMR}
	\varphi_{\mathrm{eff}}(\by,t) = \varepsilon \me^{\mi \sqrt{E_D}t} J_\varepsilon(\by)^{-1/2} \left[ \alpha_1(\varepsilon \by,\varepsilon t)\Phi_1(\by) + \alpha_2(\varepsilon \by,\varepsilon t)\Phi_2(\by)\right],
\end{equation}
where $\boldsymbol{\alpha}(\bY,T)=(\alpha_1(\bY,T),\alpha_2(\bY,T))$ is the solution to the following Dirac equations
\begin{equation}
\label{DiracEquation}
	\left\{\begin{aligned}
		2\mi \sqrt{E_D}\partial_T\alpha_1 &= \nu_{_F}(\mi \partial_{Y_1}\alpha_2 - \partial_{Y_2}\alpha_2) \\&\quad - \mu\left(\operatorname{Tr}(U\sigma_3)-\mi\operatorname{Tr}(U\sigma_1) \right)\alpha_2 - \left( \xi \operatorname{Tr}(U\sigma_0) + \xi^\# \operatorname{Tr}(U\sigma_2) \right)\alpha_1, \\
		2\mi \sqrt{E_D}\partial_T\alpha_2 &= \nu_{_F}(\mi \partial_{Y_1}\alpha_1 + \partial_{Y_2}\alpha_1) \\&\quad - \bar{\mu}\left(\operatorname{Tr}(U\sigma_3)+\mi\operatorname{Tr}(U\sigma_1) \right)\alpha_1 - \left( \xi \operatorname{Tr}(U\sigma_0) + \xi^\# \operatorname{Tr}(U\sigma_2) \right)\alpha_2, \\
		\alpha_1(\bY,0)&=\alpha_{10}(\bY), \quad\alpha_2(\bY,0)=\alpha_{20}(\bY).
	\end{aligned}\right.
\end{equation}

We first recall the well-posedness and regularity on these equations, which should be a simple corollary of the classic results in first-order hyperbolic equations theory \cite{kato1975cauchy}. Introduce the weighted Sobolev spaces, to control the $W^{N,1}$ norms that will appear later, for $s \in \nN,m \ge 0$:
\begin{equation*}
\begin{gathered}
	H^s_m (\rN^2):= \{ f \in \mathcal{S}'(\rN^2): \langle \bY \rangle^m \partial_\bY^{\boldsymbol{\gamma}}f(\bY) \in L^2(\rN^2),\forall |\boldsymbol{\gamma}| \le s\}, \\
	\|f\|_{H^s_m}:= \sum_{|\boldsymbol{\gamma}|\le s} \left\|\langle \bY \rangle^m\partial_\bY^{\boldsymbol{\gamma}}f(\bY)\right\|_{L^2}.	
\end{gathered}
\end{equation*}

\begin{proposition}
\label{RegularityOfTheDiracEquation}
	Let $E_D >0,\nu_{_F}>0$, $\xi \le 0$,$\xi^\#$ be purely imaginary and $U \in C_b^\infty$. Then \eqref{DiracEquation} is a first-order hyperbolic system with coefficients independent of time. Hence, if $s \in \nN, s > 3, m\ge 0$, for each initial value $\boldsymbol{\alpha}(0) = (\alpha_{10},\alpha_{20})^T \in H^s_m$, there exists a unique solution
	\begin{equation*}
		\boldsymbol{\alpha} \in C^0\left( [0,\rho],H^s_m \right) \cap C^1\left( [0,\rho],H^{s-1}_m \right) \cap C^2\left( [0,\rho],H^{s-2}_m \right),
	\end{equation*}
	for all $\rho>0$. Moreover, there exists $C_1,C_2>0$ (depending on $\rho$) such that
	\begin{equation*}
		\| \boldsymbol{\alpha}(T)\|_{H^s_m} + \| \partial_T\boldsymbol{\alpha}(T)\|_{H^{s-1}_m}+ \| \partial^2_T\boldsymbol{\alpha}(T)\|_{H^{s-2}_m} \le C_1 \me^{C_2T}\|\boldsymbol{\alpha}(0)\|_{H^s_m},
	\end{equation*}
	for all $T \in [0,\rho]$.
\end{proposition}
For a further physical discussion of the Dirac equation \eqref{DiracEquation}, see \Cref{SecDDE}.

Our main results are as follows.
\begin{theorem}
\label{MainTheorem1}
 Suppose $(A,V)$ is a pair of honeycomb media, $(\bK,E_D)$ is a Dirac point, $V \ge 0$ but $V \not\equiv 0$, and $\bu:\rN^2 \to \rN^2$ is a deformation such that $U=D_\bY\bu \in C_b^\infty$. For $N \ge 1,\nu >0$, assume that $\boldsymbol{\alpha}(0)=(\alpha_{10},\alpha_{20}) \in H^{2N+2}\cap H^4_{1+\nu}$. Define $\varphi_{\mathrm{eff}}$ through \eqref{TheAnsatzMR} with $\boldsymbol{\alpha} = (\alpha_1,\alpha_2)$ satisfying the Dirac equation \eqref{DiracEquation}. Then for any $\rho>0$, the wave equation \eqref{ConcernedWaveEquation} admits a unique solution $\varphi \in C([0,\rho \varepsilon^{-1}];H^{2N}(\rN^2))$ satisfying
	\begin{equation}
	\label{MainErrorBound}
		\sup_{t \in [0,\rho \varepsilon^{-1}]}\| \varphi(t) - \varphi_{\mathrm{eff}}(t) \|_{H^{2N}} \le C_{N,\nu,\rho} \varepsilon \left(\| \boldsymbol{\alpha}(0)\|_{H^{2N+2}} + \|\boldsymbol{\alpha}(0)\|_{H^4_{1+\nu}} \right),
	\end{equation} 
	for some constant $C_{N,\nu,\rho}$ independent of sufficiently small $\varepsilon>0$.
	
	In particular, if, moreover, $\boldsymbol{\alpha}(0) \in \mathcal{S}$, then for any $\rho>0,s \in \rN$, there exists a constant $C_{s,\rho,\boldsymbol{\alpha}_0}$ such that
	\begin{equation}
		\sup_{t \in [0,\rho \varepsilon^{-1}]}\| \varphi(t) - \varphi_{\mathrm{eff}}(t) \|_{H^{s}}  \le C_{s,\rho,\boldsymbol{\alpha}_0}\varepsilon,
	\end{equation}
	for all sufficiently small $\varepsilon>0$.
\end{theorem}

The case $V\equiv0$, which causes a singularity at $0 \in \sigma \left( \mL^\bracketeps \right)$, should be treated separately. To keep the deduction clean, we no longer track the derivatives bounds and only display a shorter version here. Explicit bounds in terms of Sobolev norms of \(\alpha\), analogous to those in \Cref{MainTheorem1}, can be obtained by the same argument. We omit them in order to keep the statement focused on the low-energy mechanism.
\begin{theorem}
\label{MainTheorem2}
 Suppose $(A,V)$ is a pair of honeycomb media, $(\bK,E_D)$ is a Dirac point, $V \equiv 0$, and $\bu:\rN^2 \to \rN^2$ is a deformation such that $U=D_\bY\bu \in C_b^\infty$. Assume that $\boldsymbol{\alpha}(0)=(\alpha_{10},\alpha_{20}) \in \mathcal{S}$. Define $\varphi_{\mathrm{eff}}$ through \eqref{TheAnsatzMR} with $\boldsymbol{\alpha} = (\alpha_1,\alpha_2)$ satisfying the Dirac equation \eqref{DiracEquation}. Then for any $\rho>0$ and $s \ge 0$, the wave equation \eqref{ConcernedWaveEquation} admits a unique solution $\varphi \in C^\infty([0,\rho\varepsilon^{-1}] \times \rN^2)$ satisfying
	\begin{equation}\label{MainErrorBoundV0}
		\sup_{t \in [0,\rho \varepsilon^{-1}]}\| \varphi(t) - \varphi_{\mathrm{eff}}(t) \|_{H^s} \le C_{s,\rho,\boldsymbol{\alpha}_0} \varepsilon,
	\end{equation} 
	for some constant $C_{s,\rho,\boldsymbol{\alpha}_0}$ independent of sufficiently small $\varepsilon>0$.
\end{theorem}

The well-posedness of the wave equation \eqref{ConcernedWaveEquation} is a simple application of the classical theory; see \cite{evans2022partial}. To show that the ansatz \eqref{TheAnsatzMR}, provided that $\boldsymbol{\alpha}$ satisfies the Dirac equation \eqref{DiracEquation}, approximates nicely the solution $\varphi$, we put $\eta = \varphi - \varphi_{\mathrm{eff}}$ and derive (Appendix \ref{AppAEEE}) the following equation for $\eta$
\begin{equation}
\label{ErrorEquation}
\left\{
\begin{aligned}
	\frac{\partial^2 \eta}{\partial t^2} + \mL^{(\varepsilon)}\eta &= \me^{\mi \sqrt{E_D}t}\left[ F_1(\varepsilon t)+F_2(\varepsilon t)\right], \\
	\eta(\by,0) & = 0, \\
	\frac{\partial \eta}{\partial t}(\by,0) &= F_0(\by),
\end{aligned}\right.
\end{equation}
where
\begin{equation}
	F_0(\by) = -\varepsilon^2 J_\varepsilon(\by)^{-1/2}\partial_T \alpha_k(\varepsilon \by,0)\Phi_k(\by).
\end{equation}
\begin{equation}
\label{ExpressionOfF1}
\begin{aligned}
	F_1(\by,\varepsilon t) &= \varepsilon^2 \Big(-2\mi \sqrt{E_D} \partial_T \alpha_k(\varepsilon \by,\varepsilon t) \Phi_k(\by) \\& \qquad + \nabla_\bY \alpha_k(\varepsilon \by,\varepsilon t) \cdot \mi \mathcal{A}\Phi_k(\by) - \alpha_k(\varepsilon \by,\varepsilon t) \operatorname{Tr}(U(\varepsilon \by)\mathfrak{A}\Phi_k(\by)) \Big),
\end{aligned}
\end{equation}
and $F_2(\varepsilon t)$ satisfies the bounds 
\begin{equation}\label{F2Bound}
	\|F_2(\varepsilon t)\|_{H^s} \lesssim \varepsilon^2\left(\|\boldsymbol{\alpha}(\varepsilon t)\|_{H^{s+2}} + \| \partial^2_T\boldsymbol{\alpha}(\varepsilon t)\|_{H^s}\right),
\end{equation}
for some constant independent of $\varepsilon>0$ and $t\ge 0$. Above and below, the Einstein summation convention is always adopted. The estimate \eqref{MainErrorBound} is a combination of \eqref{G0HsBound}, \eqref{G2HsBound} and \Cref{G1H2NBound}, together with regularity \Cref{RegularityOfTheDiracEquation}. The estimate \eqref{MainErrorBoundV0} requires a more delicate investigation into the near zero part of the system, and is done in \Cref{SecTS0}.

\begin{remark}
	In \cite{xie2019wave}, where the perturbation is assumed to be of first-order, the authors treat the perturbation part of $\mL^\bracketeps\eta$ as an unknown external source term, that is, 
	\begin{equation*}
		\partial^2_t\eta(t) + \mL^{(0)} \eta(t) = \me^{\mi \sqrt{E_D}t}\left[ F_1(\varepsilon t)+F_2(\varepsilon t)\right] + \varepsilon \mL_R\eta(t).
	\end{equation*}
The wave propagator contributes a factor $\left(\mL^{(0)}\right)^{-1/2}$ to the RHS and thus helps control the unknown source term. This can be seen by defining the standard wave energy
	\begin{equation*}
		E(t):= \|\partial_t \eta(t)\|_{L^2}^2 + \langle \eta(t),\mL^{(0)}\eta(t)\rangle_{L^2(\rN^2)}.
	\end{equation*}
		Differentiating with respect to $t$, we obtain
	\begin{equation*}
	\begin{aligned}
		\partial_t E(t) &= 2 \Re\langle \partial_t^2 \eta(t),\partial_t \eta(t) \rangle +  2\Re\langle \partial_t \eta(t),\mL^{(0)}\eta(t)\rangle_{L^2(\rN^2)} \\
						&\lesssim \varepsilon \Re\langle \partial_t\eta(t),\mL_R\eta(t)\rangle + \varepsilon \|\partial_t \eta(t)\|_{L^2}^2 \\
						&\lesssim \varepsilon \|\partial_t \eta(t)\|_{L^2}^2 + \varepsilon\| \eta \|_{H^1}^2 \approx \varepsilon E(t).
	\end{aligned}
	\end{equation*}
Hence, the growth of the energy can be controlled, via Gr\"onwall's inequality. In the present case, the growth of this energy cannot be bounded as above. This shows that one cannot control the dynamics via the energy associated with the unperturbed operator alone. The analysis must engage directly with the perturbed operator.

The method used in this article is only available in dealing with linear perturbations. For the treatments in the presence of non-linear perturbations, see \cite{arbunich2016rigorous,xie2021wave}.
\end{remark}

\begin{remark}
	The case $V \equiv 0$ presents another technical difficulty, since the spectrum of $\mL^\bracketeps$ now contains $0$, a singular point for the wave propagator. A direct propagator estimate breaks down in this case, as the singularity produces an extra linear growth in time. We address this difficulty by showing that this singular growth is in fact artificial for the relevant source terms. More precisely, after projecting onto the low-energy spectral subspace, the oscillatory factor $\me^{\mi\sqrt{E_D}t}$ cancels the leading low-energy behavior, allowing us to recover the desired bound by isolating the leading low-energy forced response through an explicit resolvent construction. See \Cref{SubsecProofOfBoundOnLowEnergySourceTerm} for details.
\end{remark}

\section{Decompositions and Far-Energy Estimate}
\label{SecDFEE}
We now begin the proof. In this section, we first discuss the case $V \ge 0$ but $V \not\equiv 0$. This assumption bounds the spectra of all $\mL^{(\varepsilon)}$'s by below, away from $0$, as follows.
\begin{lemma}\label{No0Spectrum}
	For sufficiently small $\varepsilon > 0$ and periodic $V \in C^\infty$, if $V \ge 0$ and $V \not\equiv 0$, there exists $c>0$ independent of $\varepsilon$ such that
	\begin{equation*}
		\sigma \left( \mL^\bracketeps \right) \subset [c,+\infty).
	\end{equation*}
\end{lemma}
This follows from the standard lower-bound theory for periodic elliptic operators, see \cite[Sections~4--5]{kuchment2016overview}, together with the uniform equivalence, for sufficiently small $\varepsilon$, of the quadratic forms of $\mL^\bracketeps$ and $\mL^{(0)}$.

Solve the error equation by means of functional calculus:
\begin{equation}
\begin{aligned}
		\eta(t) &= \frac{\me^{\mi \sqrt{\mL^{(\varepsilon)}}t}- \me^{-\mi \sqrt{\mL^{(\varepsilon)}}t}}{2\mi \sqrt{\mL^{(\varepsilon)}} } F_0\\
		 &\quad + \int_0^t \frac{\me^{\mi \sqrt{\mL^{(\varepsilon)}}(t-s)}- \me^{-\mi \sqrt{\mL^{(\varepsilon)}}(t-s)}}{2\mi \sqrt{\mL^{(\varepsilon)}}} \me^{\mi \sqrt{E_D} s}[F_1(\varepsilon s)+F_2(\varepsilon s)] \md s.	
\end{aligned}
\end{equation}
We further introduce
\begin{equation*}
	\begin{aligned}
		G_0(t) &:= \frac{\me^{\mi \sqrt{\mL^{(\varepsilon)}}t}- \me^{-\mi \sqrt{\mL^{(\varepsilon)}}t}}{2\mi \sqrt{\mL^{(\varepsilon)}} } F_0, \\
		G_1(t) &:= \int_0^t \frac{\me^{\mi \sqrt{\mL^{(\varepsilon)}}(t-s)}- \me^{-\mi \sqrt{\mL^{(\varepsilon)}}(t-s)}}{2\mi \sqrt{\mL^{(\varepsilon)}}} \me^{\mi \sqrt{E_D} s}F_1(\varepsilon s) \md s, \\
		G_2(t) &:= \int_0^t \frac{\me^{\mi \sqrt{\mL^{(\varepsilon)}}(t-s)}- \me^{-\mi \sqrt{\mL^{(\varepsilon)}}(t-s)}}{2\mi \sqrt{\mL^{(\varepsilon)}}} \me^{\mi \sqrt{E_D} s}F_2(\varepsilon s) \md s,
	\end{aligned}
\end{equation*}
so
\begin{equation*}
	\eta(t) = G_0(t) + G_1(t) + G_2(t).
\end{equation*}
These $G_i$'s are to be dealt with separately. The estimates for $G_0$ and $G_2$ are straightforward and is given right below, while those for $G_1$ require a lengthier analysis over two sections.

\subsection{\texorpdfstring{Estimate of $G_0$ and $G_2$}{Estimate of G0 and G2}}

In this case, the two terms are estimated by the uniform bound of the propagators, guaranteed by \Cref{No0Spectrum}. Starting with $L^2$ bound, we have
\begin{equation*}
	\|G_0(t)\|_{L^2} \lesssim \| F_0 \|_{L^2} \lesssim \varepsilon \| \partial_T \boldsymbol{\alpha}(0)\|_{L^2} ,
\end{equation*}
\begin{equation*}
	\|G_2(t)\|_{L^2} \lesssim \int_0^t \| F_2(\varepsilon s)\|_{L^2} \md s \lesssim \varepsilon^2 t \sup_{s \in [0,t]}\left(\|\boldsymbol{\alpha}(\varepsilon s)\|_{H^{2}} + \| \partial^2_T\boldsymbol{\alpha}(\varepsilon s)\|_{L^2}\right).
\end{equation*}
To deal with $H^{2N}$ norms, we utilize \Cref{UniformEllipticRegularity} to write
\begin{equation*}
\begin{aligned}
	\|G_0(t) \|_{H^{2N}} &\lesssim_N \left\| \left( \mL^{(\varepsilon)} \right)^N G_0(t) \right\|_{L^2} + \|G_0(t)\|_{L^2} \\
						&= \left\| \left( \mL^{(\varepsilon)} \right)^N \frac{\me^{\mi \sqrt{\mL^{(\varepsilon)}}t}- \me^{-\mi \sqrt{\mL^{(\varepsilon)}}t}}{2\mi \sqrt{\mL^{(\varepsilon)}} } F_0\right\|_{L^2} + \|G_0(t)\|_{L^2}.
\end{aligned}
\end{equation*}
Recall from spectral theory that functions of an operator commute with each other, so we have
\begin{equation}
\label{G0HsBound}
\begin{aligned}
	\|G_0(t) \|_{H^{2N}} &\lesssim_N \left\| \frac{\me^{\mi \sqrt{\mL^{(\varepsilon)}}t}- \me^{-\mi \sqrt{\mL^{(\varepsilon)}}t}}{2\mi \sqrt{\mL^{(\varepsilon)}} } \left( \mL^{(\varepsilon)} \right)^NF_0\right\|_{L^2} + \|G_0(t)\|_{L^2} \\
						&\lesssim_N \left \|\left( \mL^{(\varepsilon)} \right)^N F_0 \right\|_{L^2}+ \|G_0(t)\|_{L^2} \\
						&\lesssim_N \|F_0\|_{H^{2N}} + \varepsilon \| \partial_T \boldsymbol{\alpha}(0)\|_{L^2} \lesssim \varepsilon\| \partial_T \boldsymbol{\alpha}(0)\|_{H^{2N}}.
\end{aligned}
\end{equation}
And a similar technique applied to $G_2$ yields
\begin{equation}
\label{G2HsBound}
	\| G_2(t) \|_{H^{2N}} \lesssim_N  \varepsilon^2 t \sup_{s \in [0,t]}\left(\|\boldsymbol{\alpha}(\varepsilon s)\|_{H^{2N+2}} + \| \partial^2_T\boldsymbol{\alpha}(\varepsilon s)\|_{H^{2N}}\right).
\end{equation}

\subsection{\texorpdfstring{Decompositions of $G_1$}{Decompositions of G1}}
Recall that
\begin{equation*}
	G_1(t) = \int_0^t \frac{\me^{\mi \sqrt{\mL^{(\varepsilon)}}(t-s)}- \me^{-\mi \sqrt{\mL^{(\varepsilon)}}(t-s)}}{2\mi \sqrt{\mL^{(\varepsilon)}}} \me^{\mi \sqrt{E_D} s}F_1(\varepsilon s) \md s.
\end{equation*}
Our strategy is to separate the resonance part, which is $F_1$ projected onto a neighborhood of $E_D$, from the rest of the spectrum. To do so, we need to introduce a partition of unity for $\sigma\left( \mL^\bracketeps \right)$. For a systematic introduction to spectral theory and functional calculus, which is used a lot here, we refer to \cite{reed1978i,helffer2013spectral}. 

Let $\chi_D$ be a smooth bump function equal to $1$ in a $\delta$-vicinity of $E_D$, and supported in $(E_D-2\delta,E_D+2\delta)$. Here, $\delta$ can be chosen to be very small, and we will shrink it when needed, without relabeling. Then by \Cref{EnergyLevelNearED} we have
\begin{equation*}
	E_b(\bk) \in \supp\chi_D \Longrightarrow b \in \{-,+\} \text{ or } (|\bk-\bK|\ge k_0 \text{ and }b \le B). 
\end{equation*}
Put $\chi_L = \mathds{1}_{(-\infty,E_D)}(1-\chi_D)$ and $\chi_H=\mathds{1}_{(E_D,+\infty)}(1-\chi_D)$, both of which are supported away from $E_D$. The identity is then divided into $3$ parts
\begin{equation*}
	1= \chi_L + \chi_D +\chi_H,
\end{equation*}
through which we define the functional calculus
\begin{equation}
	P_L^{(\varepsilon)}:= \chi_L \left(\mL^{(\varepsilon)} \right),\quad P_D^{(\varepsilon)}:= \chi_D \left(\mL^{(\varepsilon)} \right),\quad P_H^{(\varepsilon)}:= \chi_H \left(\mL^{(\varepsilon)} \right).
\end{equation}
They are some kinds of smoothed versions of spectral projections, and we refer to them as regularized spectral projections, although they are not genuine projectors.
\begin{remark}
	The natural devices for decomposition with respect to spectra should be spectral projections. However, the classical spectral projections fail to stay stable under small perturbations and invalidate \Cref{ConvergenceOfSpectralProjections} below. To overcome this, we replace the indicator functions with some smooth cut-offs in the definition of spectral projections.
\end{remark}
We decompose it into $3$ parts using the regularized spectral projections:
\begin{equation}
	G_1 = G_L + G_D +G_H,\quad G_J = P_J^{(\varepsilon)}G_1,\quad J=L,D,H.
\end{equation}
Near the energy level where the resonance occurs, $P_L^{(\varepsilon)}F_1$ and $P_H^{(\varepsilon)}F_1$ are approximately $0$. These two terms are therefore straightforward to estimate. For the resonance source term $P_D^{(\varepsilon)}F_1$, we will use the Dirac equation to eliminate it.

In this section, we complete the estimate for $G_L$ and $G_H$ and leave the one for $G_D$ to the next section. First, collect the results here.
\begin{proposition}
\label{G1H2NBound}
 Let $\boldsymbol{\alpha}$ satisfy the Dirac equation. Then for any $N \ge 3,\nu >0$, there exists a constant $C_{N,\nu}$ depending only on $N,\nu$ such that
	\begin{equation}
	\begin{aligned}
		\| G_1(t) \|_{H^{2N}}^2 \! &\le \! C_{N,\nu} \!\left( \varepsilon^2 \!\! + \! \varepsilon^4t^2 \right)\! \sup_{s \in [0,t]} \! \Big(\! \|\boldsymbol{\alpha}(\varepsilon s) \|_{H^{2N \!-\!1}}^2 \!\!+\! \|\partial_T\boldsymbol{\alpha}(\varepsilon s)\|_{H^{2N\!-\!1}}^2\!\! +\! \|\partial^2_T\boldsymbol{\alpha}(\varepsilon s)\|_{H^{2N\!-\!2}}^2\! \Big)	 \\
			&\quad + C_{N,\nu} \varepsilon^4 t^2 \sup_{s\in[0,t]}\left(\| \boldsymbol{\alpha}(\varepsilon s) \|_{H^{4}_{1+\nu}}^2 + \| \partial_T\boldsymbol{\alpha}(\varepsilon s) \|_{H^{3}_{1+\nu}}^2 \right),
	\end{aligned}
	\end{equation}
	for all sufficiently small $\varepsilon>0$. Whereas, for $N \le 2,\nu>0$, we have  
	\begin{equation}
		\| G_1(t) \|_{H^{2N}}^2 \le C_\nu \left( \varepsilon^2 +\varepsilon^4t^2 \right) \sup_{s \in [0,t]} \left( \|\boldsymbol{\alpha}(\varepsilon s) \|_{H^{4}_{1+\nu}}^2 + \|\partial_T\boldsymbol{\alpha}(\varepsilon s)\|_{H^{3}_{1+\nu}} ^2\right),
	\end{equation}
	for all sufficiently small $\varepsilon>0$ and some constant $C_\nu$ depending only on $\nu$.
\end{proposition}
This proposition, together with \eqref{G0HsBound} and \eqref{G2HsBound} bounds $\| \eta\|_{H^{2N}}$ by norms of $\boldsymbol{\alpha}$, which are in term bounded by norms of $\boldsymbol{\alpha}(0)$ through \Cref{RegularityOfTheDiracEquation}. Consequently, the proof of \Cref{MainTheorem1} shall be completed upon this.

\subsection{Far-energy estimate}
Let us take $G_H$ for example. We write for simplicity that $S = \sqrt{\mL^\bracketeps}$, then the expression of $G_H$ takes the form
\begin{equation*}
	G_H(t) = \sum_{\sigma=\pm} \frac{\me^{\mi\sigma S t}}{2\mi}\int_0^t \me^{\mi \left(\sqrt{E_D} \Id- \sigma S\right)s} S^{-1} P_H^\bracketeps F_1(\varepsilon s)\md s.
\end{equation*}
The Stone's theorem for a group of unitary operators allows us to differentiate that
\begin{equation*}
	\me^{\mi \left(\sqrt{E_D} \Id- \sigma S\right)s} S^{-1} = -\mi S^{-1}\left( \sqrt{E_D}\Id - \sigma S  \right)^{-1} \partial_s \left(\me^{\mi \left(\sqrt{E_D} \Id- \sigma S\right)s} \right).
\end{equation*}
Going back to $G_H$ and integrating by parts, we have
\begin{equation}
\label{GHIntegratedByParts}
	\begin{aligned}
		-G_H(t) &= \sum_{\sigma=\pm} \frac{\me^{\mi\sigma S t}}{2} \int_0^t \! S^{-1} \!\left( \! \sqrt{E_D}\Id \! -\! \sigma S  \right)^{-1}\!P_H^\bracketeps F_1(\varepsilon s) \md \left(\me^{\mi \left(\sqrt{E_D} \Id- \sigma S\right)s} \right) \\
			&= \! S^{-1} \!\left( \! \sqrt{E_D}\Id \! -\! \sigma S  \right)^{-1}\! \sum_{\sigma=\pm} \! \frac{\me^{\mi\sigma S t}}{2} \!\left( P_H^\bracketeps F_1(\varepsilon t) \me^{\mi \left(\sqrt{E_D} \Id- \sigma S\right)t} - P_H^\bracketeps F_1(0)   \right) \\
			&\quad + \! \varepsilon S^{-1} \!\left( \! \sqrt{E_D}\Id \! -\! \sigma S  \right)^{-1}\! \sum_{\sigma=\pm} \! \frac{\me^{\mi\sigma S t}}{2} \! \int_0^t \me^{\mi \left(\sqrt{E_D} \Id- \sigma S\right)s} P_H^\bracketeps \partial_T F_1(\varepsilon s) \md s.
	\end{aligned}
\end{equation}
The validity of all these computations can be justified by the spectral decomposition theorem; see \cite{reed1978i}.

In the subspace $\mathrm{Ran}\left(P_H^\bracketeps \right)$, it holds that
\begin{equation}
\label{SInverseBounds}
	 \left\| S^{-1}\left( \sqrt{E_D}\Id + S  \right)^{-1}\right\| \lesssim \frac{1}{E_D+\delta}, \quad \left\| S^{-1}\left( \sqrt{E_D}\Id - S  \right)^{-1}\right\| \lesssim \frac{1}{\delta E_D}.
\end{equation}
Both are controlled by a constant independent of $\varepsilon$ and $t$, because our $E_D$ and $\delta$ are fixed and related only to the original honeycomb structure. Hence, taking $L^2$-norms on both sides of \eqref{GHIntegratedByParts}, and using the unitarity of operators of the form $\me^{\mi L t}$, we have
\begin{equation*}
	\begin{aligned}
		\| G_H(t) \|_{L^2}^2 &\lesssim \left( \frac{1}{E_D + \delta} + \frac{1}{\delta E_D} \right)^2 \left( \left\| P_H^\bracketeps F_1(\varepsilon t) \right\|_{L^2}^2 + \left\| P_H^\bracketeps F_1(0) \right\|_{L^2}^2\right)  \\
							&\qquad + \left( \frac{1}{E_D + \delta} + \frac{1}{\delta E_D} \right)^2\varepsilon^2 t \int_0^t \left\| P_H^\bracketeps \partial_T F_1(\varepsilon s) \right\|_{L^2}^2 \md s \\
							& \lesssim (1 + \varepsilon^2 t^2) \sup_{s \in [0,t]} \left( \left\|  F_1(\varepsilon s) \right\|_{L^2}^2 + \left\| \partial_T F_1(\varepsilon s) \right\|_{L^2}^2\right).
	\end{aligned} 
\end{equation*}

A similar result holds for $G_L$ with $L$ in place of $H$. Using the lower spectral bound from \Cref{No0Spectrum} and \eqref{ExpressionOfF1}, we get the following bounds with only a minor change and bounding operators like \eqref{SInverseBounds}.
	\begin{align}
		\label{GLBound}\| G_L(t) \|_{L^2}^2 &\lesssim \left( \varepsilon^2 + \varepsilon^4 t^2\right) \sup_{s \in [0,t]}\left( \|\boldsymbol{\alpha}(\varepsilon s)\|_{H^{1}}^2 + \|\partial_T\boldsymbol{\alpha}(\varepsilon s)\|_{H^1}^2 +\|\partial^2_T\boldsymbol{\alpha}(\varepsilon s)\|_{L^2}^2\right), \\
		\label{GHBound}\| G_H(t) \|_{L^2}^2 &\lesssim \left( \varepsilon^2 + \varepsilon^4 t^2\right) \sup_{s \in [0,t]}\left( \|\boldsymbol{\alpha}(\varepsilon s)\|_{H^{1}}^2 + \|\partial_T\boldsymbol{\alpha}(\varepsilon s)\|_{H^1}^2 + \|\partial^2_T\boldsymbol{\alpha}(\varepsilon s)\|_{L^2}^2 \right).
	\end{align}

\subsection{Near-energy system}
For
\begin{equation*}
	G_D(t) = \int_0^t \frac{\me^{\mi \sqrt{\mL^{(\varepsilon)}}(t-s)}- \me^{-\mi \sqrt{\mL^{(\varepsilon)}}(t-s)}}{2\mi \sqrt{\mL^{(\varepsilon)}}} \me^{\mi \sqrt{E_D} s}P_D^{(\varepsilon)} F_1(\varepsilon s) \md s,
\end{equation*}
the resonance causes a secular term of order $O(\varepsilon t)$ in the error, so we have to eliminate the source. The problem is, that only from the perspective of the original operator $\mL^{(0)}$ can we bound the resonance source in a satisfactory manner. That is to say, we must pull $P_D^\bracketeps F_1$ back to $\sigma\left( \mL^{(0)}\right)$. Here goes our essential observation.
\begin{proposition}
\label{ConvergenceOfSpectralProjections}
	For sufficiently small $\varepsilon>0$, we have
	\begin{equation}
		\left\| P_D^{(\varepsilon)} - P_D^{(0)} \right\|_{L^2 \to L^2} \le C \varepsilon,
	\end{equation}
	for some constant $C$ independent of $\varepsilon$.
\end{proposition}
Using \Cref{ConvergenceOfSpectralProjections} we have
\begin{equation*}
	P_D^{(\varepsilon)}F_1(\varepsilon s) = P_D^{(0)}F_1(\varepsilon s) + F_3(\varepsilon s),
\end{equation*}
for some $\|F_3(\varepsilon s)\|_{L^2} \lesssim \varepsilon \|F_1(\varepsilon s)\|_{L^2} \lesssim \varepsilon^2 \left( \| \boldsymbol{\alpha}(\varepsilon s)\|_{H^1} + \| \partial_T \boldsymbol{\alpha}(\varepsilon s)\|_{L^2}\right) ,\forall s \ge 0$. So, only controlling the term $P_D^{(0)}F_1(\varepsilon s)$ is left to us.
~\\

\noindent \textit{Proof of \Cref{ConvergenceOfSpectralProjections}.} This is a result of the convergence $\mL^{(\varepsilon)} \to \mL^{(0)}$ in the norm-resolvent sense \cite{reed1978i}, which is summarized in the following lemma. 
\begin{lemma}
\label{ResolventControl}
	For sufficiently small $\varepsilon>0$, $\mL^{(\varepsilon)}$ converges to $\mL^{(0)}$ in norm-resolvent sense. Moreover, the following bound holds for all $z \in \cN\backslash \rN$:
	\begin{equation}
	\label{ResolventBound}
		\left \| (z - \mL^{(\varepsilon)})^{-1} -  (z - \mL^{(0)})^{-1} \right \|_{L^2 \to L^2} \le C \frac{1}{|\Im{z}|}\left( \frac{|z|+1 }{|\Im{z}|}+1 \right) \varepsilon ,
	\end{equation}
	where $C>0$ is a constant independent of $z$ and $\varepsilon$.
\end{lemma}
\noindent Another key ingredient is the following \textbf{Helffer-Sj\"ostrand formula}
\begin{equation}
\label{HSFormula}
	P_D^{(\varepsilon)} = \frac{1}{\pi} \int_{\cN} \bar{\partial}\tilde{\chi}_D(z)\left(z - \mL^{(\varepsilon)}\right)^{-1} \md z,
\end{equation}
where $\bar{\partial}$ is the Cauchy-Riemann differential operator and $\tilde{\chi}_D$ is a so-called \textbf{almost-analytic continuation} of $\chi_D$, satisfying the following properties
\begin{equation}
\label{DefOfAlmostAnalyticContinuation}
\begin{gathered}
	\tilde{\chi}_D \in C_c^\infty(\cN),\quad \tilde{\chi}_D|_{\rN}=\chi_D, \\
	\left| \bar{\partial}\tilde{\chi}_D(z) \right| \lesssim_N |\Im{z}|^N,\quad \forall N \ge 0.	
\end{gathered}
\end{equation}
For the existence of an almost analytic extension and a proof of this formula, we refer to \cite{dimassi1999spectral,zworski2012semiclassical}.

To prove the proposition, we first employ \eqref{HSFormula} to write
	\begin{equation*}
		P_D^{(\varepsilon)} - P_D^{(0)} = \frac{1}{\pi} \int_\cN \bar{\partial}\tilde{\chi}_D(z)\left[\left(z - \mL^{(\varepsilon)}\right)^{-1} - \left(z - \mL^{(0)}\right)^{-1} \right] \md z.
	\end{equation*}
	Bound both sides by operator norms and substitute \eqref{ResolventBound}. Then, the right hand side is left with a $\varepsilon$ multiplied by a remaining integral, which is absolutely convergent in view of \eqref{DefOfAlmostAnalyticContinuation}.\pd
	
\noindent \textit{Proof of \Cref{ResolventControl}.} To estimate $\left(z - \mL^{(\varepsilon)} \right)^{-1} -  \left(z - \mL^{(0)}\right)^{-1}$, we need to compare the solution to
\begin{equation*}
	zu - \mL^{(0)}u = F
\end{equation*}
with the one to
\begin{equation}
\label{vResolvent}
	zv - \mL^{(\varepsilon)}v = F
\end{equation}
for $z \in \cN \backslash \rN$.
Taking a difference of the two equations, we have
\begin{equation*}
	(z - \mL^{(0)})(v-u) = \varepsilon \mL_R v,
\end{equation*}
where $\varepsilon\mL_R = \mL^{(\varepsilon)}- \mL^{(0)}$ and hence by estimate of resolvents for self-adjoint operator
\begin{equation}
	\| v- u \|_{L^2} \le \frac{\varepsilon}{|\Im{z}|}\| \mL_R v \|_{L^2} \lesssim \frac{\varepsilon}{|\Im{z}|}\| v \|_{H^2} .
\end{equation}
Then we bound $\| v\|_{H^2}$ by $\|F\|_{L^2}$ through \eqref{vResolvent}. By the uniform elliptic regularity (\Cref{UniformEllipticRegularity}) :
\begin{equation*}
	\| v \|_{H^2} \lesssim \left \| \mL^{(\varepsilon)}v \right\|_{L^2} + \|v\|_{L^2} \le (|z|+1)\|v\|_{L^2} + \|F\|_{L^2},
\end{equation*}
so it suffices to bound $\|v\|_{L^2}$. But this follows again from the estimate of resolvents for self-adjoint operators:
\begin{equation*}
	\| v \|_{L^2} \le \frac{1}{|\Im{z}|}\|F\|_{L^2}.
\end{equation*}

A simpler way to bound $\|v\|_{H^2}$ is symbolic calculus. Indeed, we can show that $\left(z-\mL^{(\varepsilon)}\right)^{-1}$ are pseudo-differential operators of order $-2$, with a uniform upper bound provided that $z$ belongs to some compact subset of $\varrho(\mL^{(\varepsilon)})$. See \cite{zworski2012semiclassical}. \pd

The estimate of $P_D^{(0)}F_1$ will take over the next section. We first state the result that will be obtained there.
\begin{proposition}
\label{DiracPartBound}
If $\boldsymbol{\alpha}$ satisfies the Dirac equation, then for any $\nu>0$, there exists a constant $C_\nu>0$ making the followings hold for all sufficiently small $\varepsilon > 0$ and $s \ge 0$:
	\begin{equation}
		\left \| P_D^{(0)}F_1(\varepsilon s) \right \|_{L^2(\rN^2)}^2 \le C_\nu \varepsilon^4 \left( \| \boldsymbol{\alpha}(\varepsilon s) \|_{H^{4}_{1+\nu}(\rN^2)}^2 + \| \partial_T\boldsymbol{\alpha}(\varepsilon s) \|_{H^{3}_{1+\nu}(\rN^2)}^2 \right).
	\end{equation}
	Hence, back to the expression of $G_D$, we have
	\begin{equation}
	\label{GDBound}
		\left \| G_D(t) \right \|_{L^2(\rN^2)}^2 \le C_\nu \varepsilon^4t^2 \sup_{s \in [0,t]} \left( \| \boldsymbol{\alpha}(\varepsilon s) \|_{H^{4}_{1+\nu}(\rN^2)}^2 + \| \partial_T\boldsymbol{\alpha}(\varepsilon s) \|_{H^{3}_{1+\nu}(\rN^2)}^2 \right).
	\end{equation}
\end{proposition}

\subsection{Sobolev norms estimate for the high energy part}
Combining \eqref{GLBound}, \eqref{GHBound} and \eqref{GDBound} we have obtained the desired estimate for $G_1$, in $L^2$-norm. In this subsection, we generalize it to $H^s$-norm estimate. This has already been done for $G_0$ and $G_2$ in \eqref{G0HsBound} and \eqref{G2HsBound}, so we only need to focus on the estimate of $\| G_1 \|_{H^s}$.

In view of \Cref{UniformEllipticRegularity}, to bound $\| G_1(t) \|_{H^{2N}}$, it suffices to do with $\| \left(\mL^{(\varepsilon)}\right)^NG_1(t)\|_{L^2}$ (and with $\|G_1(t)\|_{L^2}$, which is already done), where $2N$ is an even number greater than $s$. Again, this need to be decomposed into the prescribed three parts
\begin{equation*}
	\left\| \left(\mL^{(\varepsilon)}\right)^NG_1(t)\right\|_{L^2}^2 \lesssim \sum_{J=L,D,H} \left\| \left(\mL^{(\varepsilon)}\right)^NG_J(t)\right\|_{L^2}^2.
\end{equation*}
We emphasize that, for low- and middle-parts, the differential operators localized to finite parts of its spectrum are, in fact, $L^2 \to L^2$ bounded. Hence, for $J=L$ and $J=D$ we have
\begin{equation}
\label{LowAndMiddle2NG1}
	\left\| \left(\mL^{(\varepsilon)}\right)^NG_L(t)\right\|_{L^2}^2 \lesssim \|G_L(t)\|_{L^2}^2, \quad  \left\| \left(\mL^{(\varepsilon)}\right)^NG_D(t)\right\|_{L^2}^2 \lesssim \|G_D(t)\|_{L^2}^2,
\end{equation}
and they had been appropriately controlled through \eqref{GLBound} and \eqref{GDBound}.

It remains to control
\begin{equation*}
\begin{aligned}
		-  \left(  \mL^{(\varepsilon)}  \right)^N  G_H(t)  &=  \left(  \mL^{(\varepsilon)} \right)^N  S^{-1} \left(  \sqrt{E_D}\Id  - \sigma S   \right)^{-1} \Bigg[\sum_{\sigma=\pm} \frac{\me^{\mi\sigma S t}}{2} \\
		 	&\quad \times \Bigg(\me^{\mi \left(\sqrt{E_D} \Id- \sigma S\right)t} P_H^\bracketeps F_1(\varepsilon t) - P_H^\bracketeps F_1(0) \\ 
		 	&\qquad + \varepsilon \int_0^t \me^{\mi \left(\sqrt{E_D} \Id- \sigma S\right)s} P_H^\bracketeps \partial_T F_1(\varepsilon s) \md s\Bigg) \Bigg],
	\end{aligned}	
\end{equation*}
where the expression \eqref{GHIntegratedByParts} is used. This is done by directly taking $L^2$-norms on both sides. Notice that in $\mathrm{Ran}\left( P_H^\bracketeps \right)$, by symbolic calculus,
\begin{equation*}
	\left \| \me^{\mi\sigma S t} \left(\mL^{(\varepsilon)}\right)^NS^{-1}\left( \sqrt{E_D}\Id - \sigma S  \right)^{-1} \me^{\mi \left(\sqrt{E_D} \Id- \sigma S\right)s} \right\|_{H^{2N-2} \to L^2} \lesssim 1,
\end{equation*}
which yields
\begin{equation*}
	\left \| \left(\mL^{(\varepsilon)}\right)^N G_H(t) \right\|^2_{L^2} \lesssim (1 + \varepsilon^2 t^2) \sup_{s \in [0,t]} \left( \left\|  F_1(\varepsilon s) \right\|_{H^{2N-2}}^2 + \left\| \partial_T F_1(\varepsilon s) \right\|_{H^{2N-2}}^2\right).
\end{equation*}

Finally, back to \eqref{ExpressionOfF1} we have
\begin{equation}
\label{ControlOfHighEnergySpectralIntegral}
\begin{aligned}
	\left \| \left( \mL^{(\varepsilon)}\right)^N G_H(t) \right\|^2_{L^2} &\lesssim  \left( \varepsilon^2 +\varepsilon^4t^2 \right) \sup_{s \in [0,t]} \Big( \|\boldsymbol{\alpha}(\varepsilon s) \|_{H^{2N-1}}^2 \\  
		& \qquad + \|\partial_T\boldsymbol{\alpha}(\varepsilon s)\|_{H^{2N-1}} ^2 + \|\partial^2_T\boldsymbol{\alpha}(\varepsilon s)\|_{H^{2N-2}} ^2 \Big).
\end{aligned}
\end{equation}
Combining \eqref{GLBound}, \eqref{GDBound}, \eqref{LowAndMiddle2NG1} and \eqref{ControlOfHighEnergySpectralIntegral}, we conclude the results of \Cref{G1H2NBound}.

\section{Control of the Resonance}
\label{SecCR}
It remains to prove \Cref{DiracPartBound}. The problem is now reduced to the domain of the unperturbed operator. Since the spectral structure of $\mL^{(0)}$ is well-known, we have the following decomposition for $P_D^{(0)}F_1(\varepsilon s)$.

\begin{proposition}
\label{FunctionalCalculusOfFBDecomposition}
	For sufficiently small $\varepsilon>0$:
	\begin{equation}
	\label{FBDecompositionOfPDF1}
		P_D^{(0)}F_1(\by,\varepsilon s) = \frac{1}{|\mathbb{B}_h|} \sum_{b \ge 1}^{\infty} \int_{\mathbb{B}_h} \chi_D(E_b(\bk))\Phi_b(\by;\bk)\langle \Phi_b(\bk),F_1(\varepsilon s)\rangle_{L^2(\rN^2)} \md \bk,
	\end{equation}
	in sense of $L^2$-convergence.
\end{proposition}

\begin{proof}
	Denote by $\mathcal{B}:L^2(\rN^2) \to L^2(\mathbb{B}_h;L^2_\bk)$ the Bloch transform. Each element in the latter space is given by
	\begin{equation*}
		v(\bx;\bk) = \sum_{b \ge 1}v_b(\bk)\Phi_b(\bx;\bk).
	\end{equation*}
	The operator $\widehat{\mL^{(0)}}:= \mathcal{B}\mL^{(0)}\mathcal{B}^{-1}$ acting on $L^2(\mathbb{B}_h;L^2_\bk)$ has the effect
	\begin{equation*}
		\left( \widehat{\mL^{(0)}}v \right)(\bx,\bk) = \sum_{b \ge 1}E_b(\bk)v_b(\bk)\Phi_b(\bx;\bk).
	\end{equation*}
	
	By unitary equivalence, $\widehat{P_D^{(0)}}:=\mathcal{B}P_D^{(0)}\mathcal{B}^{-1}$ is the spectral projection onto the same part of the spectrum of $\widehat{\mL^{(0)}}$. Hence
	\begin{equation*}
		\left( \widehat{P_D^{(0)}}v \right)(\bx,\bk) = \sum_{b \ge 1}\chi_D \left( E_b(\bk)\right) v_b(\bk)\Phi_b(\bx;\bk).
	\end{equation*}
	Pulling this formula back to $L^2(\rN^2)$ yields the desired result.
\end{proof}

The estimate of this Floquet-Bloch decomposition is analogous to the corresponding parts of \cite{fefferman2014wave} and \cite{xie2021wave}, but we still make a complete proof.

\subsection{Decomposition with respective to quasi-momenta}
Recall from \Cref{EnergyLevelNearED} that for $E_b(\bk)$ to lie in the support of $\chi_D$, either $b = b_*,b_*+1$, or $|\bk - \bK| \ge k_0$. And there exists a $B \in \nN$ such that $E_b(\bk) \ge E_D+\delta$ whenever $b > B$. Hence, the summation \eqref{FBDecompositionOfPDF1} is divided into $3$ parts:
\begin{equation}
\begin{aligned}
	P_D^{(0)}F_1(\by,\varepsilon s) &= \sum_{b = b_*}^{b_*+1} \int_{|\bk-\bK|<\varepsilon} \chi_D(E_b(\bk))\Phi_b(\by;\bk)\langle \Phi_b(\bk),F_1(\varepsilon s)\rangle_{L^2(\rN^2)} \md \bk \\
			&\quad + \sum_{b = b_*}^{b_*+1} \int_{\varepsilon \le |\bk-\bK|<k_0} \chi_D(E_b(\bk))\Phi_b(\by;\bk)\langle \Phi_b(\bk),F_1(\varepsilon s)\rangle_{L^2(\rN^2)} \md \bk \\
			&\quad + \sum_{b \ge 1}^{B} \int_{|\bk-\bK|\ge k_0} \chi_D(E_b(\bk))\Phi_b(\by;\bk)\langle \Phi_b(\bk),F_1(\varepsilon s)\rangle_{L^2(\rN^2)} \md \bk \\
			&:= H_1(\varepsilon s) + H_2(\varepsilon s) + H_3(\varepsilon s).
\end{aligned}
\end{equation}
We re-tag $b \in \{ b_*,b_*+1 \}$ by $ b \in \{ +,-\}$ in the first two terms. As was repeated, the first term will be eliminated by the Dirac equation. The third term can be made arbitrarily small since we remove the quasi-momenta where the initial data are concentrated. The results of this section are concluded as follows.
\begin{proposition}
\label{PD0F1Bound}
	 If $\boldsymbol{\alpha}$ satisfies the Dirac equation, then for any $N \ge 3$, there exists a constant $C_N>0$ making the followings hold for all sufficiently small $\varepsilon > 0$ and $s \ge 0$:
	\begin{align}
		\label{H1Bound} \|H_1(\varepsilon s)\|_{L^2(\rN^2)}^2 &\le C_N \varepsilon^4\|\boldsymbol{\Gamma}(\varepsilon s)\|_{W^{N,1}(\mathbb{R}^2)}^2, \\
		\label{H2Bound} \|H_2(\varepsilon s)\|_{L^2(\rN^2)}^2 &\le C_N \varepsilon^4\|\boldsymbol{\Gamma}(\varepsilon s)\|_{W^{N,1}(\mathbb{R}^2)}^2, \\
		\label{H3Bound} \|H_3(\varepsilon s)\|_{L^2(\rN^2)}^2 &\le C_N \varepsilon^{2N}\|\boldsymbol{\Gamma}(\varepsilon s)\|_{W^{N,1}(\mathbb{R}^2)}^2.
	\end{align}
	Hence in view of \eqref{DefOfGamma} and with $N=3$, we can bound that
	\begin{equation}
	\begin{aligned}
		\left \| P_D^{(0)}F_1(\varepsilon s) \right \|_{L^2(\rN^2)}^2 &\le C\varepsilon^4 \left( \| \boldsymbol{\alpha}(\varepsilon s) \|_{W^{4,1}(\rN^2)}^2 + \| \partial_T\boldsymbol{\alpha}(\varepsilon s) \|_{W^{3,1}(\rN^2)}^2 \right) \\
				&\le C_\nu \varepsilon^4 \left( \| \boldsymbol{\alpha}(\varepsilon s) \|_{H^4_{1+\nu}(\rN^2)}^2 + \| \partial_T\boldsymbol{\alpha}(\varepsilon s) \|_{H^3_{1+\nu}(\rN^2)}^2 \right).
	\end{aligned}
	\end{equation}
\end{proposition}

Take the following conventions
\begin{equation}
\label{DefOfGamma}
\begin{alignedat}{2}
  \Gamma_{1}(\bY,T)&=-2\mi\sqrt{E_D}\,\partial_T\alpha_1(\bY,T),\quad&
  \Gamma_{2}(\bY,T)&=-2\mi\sqrt{E_D}\,\partial_T\alpha_2(\bY,T),\\
  \Gamma_{3}(\bY,T)&=\nabla_\bY\alpha_1(\bY,T),\quad&
  \Gamma_{4}(\bY,T)&=\nabla_\bY\alpha_2(\bY,T),\\
  \Gamma_{5}(\bY,T)&=-\alpha_1(\bY,T)U^T(\bY),\quad&
  \Gamma_{6}(\bY,T)&=-\alpha_2(\bY,T)U^T(\bY),
\end{alignedat}
\end{equation}
and
\begin{equation}
\label{Psi's}
\begin{alignedat}{2}
  \Psi_{1}&=\Phi_1,\quad&
  \Psi_{2}&=\Phi_2,\\
  \Psi_{3}&=\mi \mathcal{A}\Phi_1,\quad&
  \Psi_{4}&=\mi \mathcal{A}\Phi_2,\\
  \Psi_{5}&=\mathfrak{A}\Phi_1,\quad&
  \Psi_{6}&=\mathfrak{A}\Phi_2.
\end{alignedat}
\end{equation}
Following Einstein summation convention, we have
\begin{equation*}
	F_1(\by,\varepsilon t) = \varepsilon^2 \Gamma_{r}(\varepsilon\by,\varepsilon s)\Psi_{r}(\by),
\end{equation*}
where we slightly abuse the notation when, for example, $\Gamma_{5}\Psi_{5} = \Gamma_{5}:\Psi_{5}$. The notation $\boldsymbol{\Gamma}$ collects all $\Gamma_r$ aligned as a vector. We list the several frequently used results here. They are proved in \cite{fefferman2014wave}.
\begin{proposition}
\label{FarMomentaDecay}
Let $\Gamma(\bY) \in W^{N,1}(\mathbb{R}^2),N\ge 3$ and $\Psi(\by) \in C^\infty(\mathbb{R}^2) \cap L^2_{\mathbf{K}}(\mathbb{R}^2 / \Lambda)$. Then, the following statements hold
\begin{equation}
\label{ExpansionOfInnerproductAgainstPeriodicFunctions}
\begin{aligned}
\langle &\Phi_\pm(\cdot; \mathbf{k}), \Gamma(\varepsilon \cdot)\Psi(\cdot) \rangle_{L^2(\rN^2)} \\
&= \! \varepsilon^{-2} \frac{1}{|\Omega_h|} \int_{\Omega_h} \!\! \Phi_\pm(\by; \mathbf{k})\! \sum_{\mathbf{m} \in \mathbb{Z}^2}\! e^{\mi(m_1 \mathbf{k}_1 + m_2 \mathbf{k}_2 + \mathbf{k} - \mathbf{K}) \cdot \by} \widehat{\Gamma}\left( \frac{m_1 \mathbf{k}_1 + m_2 \mathbf{k}_2 + \mathbf{k} - \mathbf{K}}{\varepsilon} \right) \Psi(\by) \md \by,
\end{aligned}	
\end{equation}
and
\begin{align}
\label{RapidDecay}
|\widehat{\Gamma}(\boldsymbol{\kappa})| \leq C_N \frac{1}{|\boldsymbol{\kappa}|^N} \|\Gamma\|_{W^{N,1}(\mathbb{R}^2)}, \quad \forall \, \boldsymbol{\kappa} \neq 0,
\end{align}
where $C_N$ is a universal constant depending only on $N \in \nN,N \ge 3$.

Furthermore, for any $\mathbf{k} \in \mathbb{B}_h$ and $\varepsilon > 0$ sufficiently small,
\begin{align}
\label{DistantByk0} |m_1 \mathbf{k}_1 + m_2 \mathbf{k}_2 + \mathbf{k} - \mathbf{K}| &\geq C(1 + |\mathbf{m}|), \quad \forall \, \mathbf{m} \in \mathbb{Z}^2, \ |\mathbf{k} - \mathbf{K}| \ge k_0, \\
\label{MomentaInk0} |m_1 \mathbf{k}_1 + m_2 \mathbf{k}_2 + \mathbf{k} - \mathbf{K}| &\geq C|\mathbf{m}|, \quad \forall \, \mathbf{m} \neq (0,0), \ |\mathbf{k} - \mathbf{K}| < k_0.
\end{align}
for some uniform $C>0$ independent of $\varepsilon$.
\end{proposition}

\subsection{Far-momenta estimate}
\textit{Proof of \eqref{H3Bound}.} We begin with the easier one. Use the expansion \eqref{ExpansionOfInnerproductAgainstPeriodicFunctions} and estimate \eqref{DistantByk0}:
\begin{equation*}
\begin{aligned}
	\big| \langle \Phi_b &(\bk),F_1 (\varepsilon s)\rangle_{L^2(\rN^2)} \big| \\ 
			&\le \frac{1}{|\Omega_h|} \int_{\Omega_h} \left| \overline{\Phi_b(\by;\bk)} \right| \sum_{\mathbf{m} \in \zN^2} \left| \widehat{\Gamma_{r}}\left( \frac{m_1\bk_1+m_2\bk_2+\bk - \bK}{\varepsilon},\varepsilon s\right) \Psi_{r}(\by) \right| \md \by \\
			&\lesssim_N  \sum_{\mathbf{m} \in \zN^2} \frac{\varepsilon^N}{(1+|\mathbf{m}|)^N} \|\boldsymbol{\Gamma}(\varepsilon s)\|_{W^{N,1}(\mathbb{R}^2)} \\
			&\lesssim_N \varepsilon^{N}\|\boldsymbol{\Gamma}(\varepsilon s)\|_{W^{N,1}(\mathbb{R}^2)},\quad \text{when } |\bk-\bK|\ge k_0, b \le B,
\end{aligned}
\end{equation*}
provided that $N \ge 3$ so that the series converge. Hence,
\begin{equation}
	\begin{aligned}
		\|H_3(\varepsilon s)\|_{L^2(\rN^2)}^2 &= \sum_{b = 1} ^B\int_{|\bk-\bK|\ge k_0} \left| \chi_D(E_b(\bk)) \langle \Phi_b(\bk),F_1(\varepsilon s)\rangle_{L^2(\rN^2)} \right|^2 \md \bk \\
		&\lesssim_N \sum_{b =1}^B \int_{\mathbb{B}_h} \varepsilon^{2N}\|\boldsymbol{\Gamma}(\varepsilon s)\|_{W^{N,1}(\mathbb{R}^2)}^2 \md\bk \\
		&\lesssim_N \varepsilon^{2N}\|\boldsymbol{\Gamma}(\varepsilon s)\|_{W^{N,1}(\mathbb{R}^2)}^2 ,
	\end{aligned}
\end{equation}
for every $N\ge 3$.

\subsection{Near-momenta estimate}
\textit{Proof of \eqref{H1Bound} and \eqref{H2Bound}.} We deal with $H_1$ and $H_2$ here. First estimate $\langle \Phi_b(\bk),F_1(\varepsilon s) \rangle_{L^2(\rN^2)}$ when $b \in \{ -,+\}$. We expand using \eqref{ExpansionOfInnerproductAgainstPeriodicFunctions} that
\begin{equation*}
	\begin{aligned}
		\langle \Phi_\pm(\bk),F_1 (\varepsilon s)\rangle_{L^2(\rN^2)} &= \frac{\varepsilon^{-2}}{|\Omega_h|} \int_{\Omega_h} \overline{\Phi_\pm(\by;\bk)} \sum_{\mathbf{m} \in \zN^2} \me^{\mi(m_1\bk_1+m_2\bk_2+\bk - \bK)\by} \\
		& \quad \times \varepsilon^2 \widehat{\Gamma_{r}}\left( \frac{m_1\bk_1+m_2\bk_2+\bk - \bK}{\varepsilon},\varepsilon s\right) \Psi_{r}(\by)\md \by \\
		& = I_{1,\pm}+I_{2,\pm},
	\end{aligned}
\end{equation*}
where we put
\begin{equation*}
\begin{aligned}
	I_{1,\pm} &:=	\frac{1}{|\Omega_h|} \int_{\Omega_h} \overline{\Phi_\pm(\by;\bk)} \me^{\mi(\bk - \bK)\by}\widehat{\Gamma_{r}}\left( \frac{\bk - \bK}{\varepsilon},\varepsilon s\right) \Psi_{r}(\by)\md \by \\
	I_{2,\pm} &:= \langle \Phi_\pm(\bk),F_1 (\varepsilon s)\rangle_{L^2(\rN^2)} - I_{1,\pm} = \sum_{\mathbf{m} \ne (0,0)}\cdots.
\end{aligned} 
\end{equation*}
The $I_{2,\pm}$ term is easy. When $|\bk-\bK| < k_0 $, combining \eqref{RapidDecay} and \eqref{MomentaInk0} to yield
\begin{equation*}
\begin{aligned}
	\left| \widehat{\Gamma_{r}}\left( \frac{m_1\bk_1+m_2\bk_2+\bk - \bK}{\varepsilon},\varepsilon s\right) \right| &\le C_N \frac{\varepsilon^N}{|m_1\bk_1+m_2\bk_2+\bk - \bK|^N} \|\Gamma_r(\varepsilon s)\|_{W^{N,1}(\mathbb{R}^2)} \\
						 &\le \frac{C_N}{C}\frac{\varepsilon^{N}}{|\mathbf{m}|^N}\|\Gamma_r(\varepsilon s)\|_{W^{N,1}(\mathbb{R}^2)}.
\end{aligned}
\end{equation*}
Hence, choosing $N\ge3$, it is easy to bound
\begin{equation}
\label{I2}
\begin{aligned}
	\Big|&\mathds{1}_{\{|\bk-\bK| < k_0\}} I_{2,\pm}\Big| \\
				& \, \le \sum_{\mathbf{m} \ne (0,0)} \frac{1}{|\Omega_h|} \int_{\Omega_h} \left|\overline{\Phi_\pm(\by;\bk)}  \right| \left|\widehat{\Gamma_{r}}\left( \frac{m_1\bk_1+m_2\bk_2+\bk - \bK}{\varepsilon},\varepsilon s\right) \right| \left| \Psi_{r}(\by) \right|\md \by 	\\
				& \, \lesssim_N \sum_{\mathbf{m} \ne (0,0)} \frac{\varepsilon^{N}}{|\mathbf{m}|^N}\|\boldsymbol{\Gamma}(\varepsilon s)\|_{W^{N,1}(\mathbb{R}^2)} \\
				& \, \lesssim_N \varepsilon^{N}\|\boldsymbol{\Gamma}(\varepsilon s)\|_{W^{N,1}(\mathbb{R}^2)},
\end{aligned}
\end{equation}
with a constant uniform in $\varepsilon>0$ and $\bk \in \mathbb{B}_h$.

Now we handle $I_{1,\pm}$. For $|\bk-\bK| < k_0$, we substitute \eqref{PhikNearK} into the expression of $I_{1,\pm}$:
\begin{equation}
\label{FullI1}
	\begin{aligned}
		\sqrt{2} \mathds{1}_{\{|\bk-\bK| < k_0\}} I_{1,\pm} &= \widehat{\Gamma_{r}}\left(\frac{\boldsymbol{\kappa}}{\varepsilon},\varepsilon s\right)\left \langle \frac{\kappa_1 + \mi \kappa_2}{|\boldsymbol{\kappa}|}\Phi_1 \pm \Phi_2,\Psi_{r} \right\rangle_{L^2_\bK(\rN^2/\Lambda_h)}  \\
			&\quad + \widehat{\Gamma_{r}}\left(\frac{\boldsymbol{\kappa}}{\varepsilon},\varepsilon s\right) \langle \Phi_{R,\pm}(\boldsymbol{\kappa}),\Psi_{r} \rangle_{L^2_\bK(\rN^2/\Lambda_h)}\\
			&= \widehat{\Gamma_{r}}\left(\frac{\boldsymbol{\kappa}}{\varepsilon},\varepsilon s\right)\frac{\kappa_1 - \mi \kappa_2}{|\boldsymbol{\kappa}|}\left \langle\Phi_1,\Psi_{r} \right\rangle_{L^2_\bK(\rN^2/\Lambda_h)} \\
			&\quad \pm \widehat{\Gamma_{r}}\left(\frac{\boldsymbol{\kappa}}{\varepsilon},\varepsilon s\right)\left \langle\Phi_2,\Psi_{r}\right\rangle_{L^2_\bK(\rN^2/\Lambda_h)} \\
			&\quad + \widehat{\Gamma_{r}}\left(\frac{\boldsymbol{\kappa}}{\varepsilon},\varepsilon s\right) \langle \Phi_{R,\pm}(\boldsymbol{\kappa}),\Psi_{r} \rangle_{L^2_\bK(\rN^2/\Lambda_h)}.
	\end{aligned}
\end{equation}
Consider first the case $|\boldsymbol{\kappa}|= |\bk - \bK|<\varepsilon$, then the last term
\begin{equation*}
	\left| \widehat{\Gamma_{r}}\left(\frac{\boldsymbol{\kappa}}{\varepsilon},\varepsilon s\right) \langle \Phi_{R,\pm}(\boldsymbol{\kappa}),\Psi_{r} \rangle_{L^2_\bK(\rN^2/\Lambda_h)} \right| \lesssim \varepsilon\|\boldsymbol{\Gamma}(\varepsilon s) \|_{L^1(\rN^2)},
\end{equation*}
with a constant independent from $\varepsilon>0$ and $|\boldsymbol{\kappa}|<\varepsilon$. The other inner products appearing in \eqref{FullI1} are computed in \Cref{MathcalAMatrix}, \eqref{MathfrakAMatrix} and are listed below:
\begin{equation*}
\begin{alignedat}{3}
	\langle \Phi_1,\Psi_{1} \rangle = 1, \quad &\langle \Phi_1,\Psi_{3} \rangle = \mathbf{0}, & \quad \langle \Phi_1,\Psi_{5} \rangle =\xi \sigma_0 + \xi^\# \sigma_2, \\
	\langle \Phi_1,\Psi_{2} \rangle = 0, \quad &\langle \Phi_1,\Psi_{4} \rangle = \nu_{_F} (\mi, -1)^T, & \quad \langle \Phi_1,\Psi_{6} \rangle =\mu (\sigma_3-\mi \sigma_1), \\
	\langle \Phi_2,\Psi_{1} \rangle = 0, \quad &\langle \Phi_2,\Psi_{3} \rangle = \nu_{_F} (\mi, 1)^T, & \quad
		\langle \Phi_2,\Psi_{5} \rangle =\bar{\mu}(\sigma_3 + \mi \sigma_1), \\
	\langle \Phi_2,\Psi_{2} \rangle = 1, \quad &\langle \Phi_2,\Psi_{4} \rangle = \mathbf{0}, & \quad \langle \Phi_2,\Psi_{6} \rangle =\xi \sigma_0 + \xi^\# \sigma_2.
\end{alignedat}
\end{equation*}
Hence
\begin{equation*}
	\sqrt{2}\mathds{1}_{\{|\bk-\bK|<k_0\}} I_{1,\pm}\! = \! \frac{\kappa_1 - \mi \kappa_2}{|\boldsymbol{\kappa}|}D_1 \! \pm \! D_2 \! +  \!O_{L^{\infty}(\{|\bk-\bK|<k_0\})}\left(\varepsilon \| \boldsymbol{\Gamma}(\varepsilon s)\|_{L_1(\rN^2)} \right),
\end{equation*}
where
\begin{equation*}
\begin{aligned}
	D_1 &=  -2\mi\sqrt{E_D}\,\widehat{\partial_T\alpha_1}\left(\frac{\boldsymbol{\kappa}}{\varepsilon},\varepsilon s\right) + \nu_{_F} \left( \mi \widehat{\partial_{Y_1}\alpha_2}\left(\frac{\boldsymbol{\kappa}}{\varepsilon},\varepsilon s\right) - \widehat{\partial_{Y_2}\alpha_2}\left(\frac{\boldsymbol{\kappa}}{\varepsilon},\varepsilon s\right) \right) \\
	&\qquad - \mathcal{F}\big[ \mu\left(\operatorname{Tr}(U\sigma_3)-\mi\operatorname{Tr}(U\sigma_1) \right)\alpha_2 + \left( \xi \operatorname{Tr}(U\sigma_0) + \xi^\# \operatorname{Tr}(U\sigma_2) \right)\alpha_1\big]\left(\frac{\boldsymbol{\kappa}}{\varepsilon},\varepsilon s\right)  \\
	D_2 &=	-2\mi\sqrt{E_D}\,\widehat{\partial_T\alpha_2}\left(\frac{\boldsymbol{\kappa}}{\varepsilon},\varepsilon s\right) + \nu_{_F} \left( \mi \widehat{\partial_{Y_1}\alpha_1}\left(\frac{\boldsymbol{\kappa}}{\varepsilon},\varepsilon s\right) + \widehat{\partial_{Y_2}\alpha_1}\left(\frac{\boldsymbol{\kappa}}{\varepsilon},\varepsilon s\right) \right) \\
				&\qquad - \mathcal{F}\big[ \bar{\mu}\left(\operatorname{Tr}(U\sigma_3)+\mi\operatorname{Tr}(U\sigma_1) \right)\alpha_1 + \left( \xi \operatorname{Tr}(U\sigma_0) + \xi^\# \operatorname{Tr}(U\sigma_2) \right)\alpha_2\big]\left(\frac{\boldsymbol{\kappa}}{\varepsilon},\varepsilon s\right).
\end{aligned}
\end{equation*}
These quantities are precisely the Fourier transforms of \eqref{DiracEquation} and thus vanish. To conclude, we have that
\begin{equation}
\label{I1Closed}
	\left| \mathds{1}_{\{|\bk-\bK|<\varepsilon\}} I_{1,\pm}  \right|\lesssim \varepsilon\| \boldsymbol{\Gamma}(\varepsilon s)\|_{L_1(\rN^2)}.
\end{equation}
Next, for $\varepsilon \le |\boldsymbol{\kappa}|=|\bk-\bK| <k_0$, we still have
\begin{equation}
\label{I1Mid}
\begin{aligned}
	\left| \sqrt{2} \mathds{1}_{\{\varepsilon\le |\bk-\bK| < k_0\}} I_{1,\pm} \right| &
	= \left |\widehat{\Gamma_{r}}\left(\frac{\boldsymbol{\kappa}}{\varepsilon},\varepsilon s\right) \langle \Phi_{r,\pm}(\boldsymbol{\kappa}),\Psi_{r} \rangle_{L^2_\bK(\rN^2/\Lambda_h)} \right| \\
	&\lesssim |\boldsymbol{\kappa}|\left |\widehat{\boldsymbol{\Gamma}}\left(\frac{\boldsymbol{\kappa}}{\varepsilon},\varepsilon s\right) \right| \\
	&\lesssim_N \left\| \boldsymbol{\Gamma} (\varepsilon s) \right\|_{W^{N,1}(\rN^2)} \frac{\varepsilon^N}{|\boldsymbol{\kappa}|^{N-1}}.
\end{aligned}
\end{equation}

Combining \eqref{I1Closed} and \eqref{I2} and choosing $N \ge 3$ we have
\begin{equation}
	\begin{aligned}
		\|H_1(\varepsilon s)\|_{L^2(\rN^2)}^2 &= \sum_{b = +,-} \int_{\mathbb{B}_h} \left| \mathds{1}_{\{|\bk-\bK| < \varepsilon\}} \chi_D(E_b(\bk)) \langle \Phi_b(\bk),F_1(\varepsilon s)\rangle_{L^2(\rN^2)} \right|^2 \md \bk \\
			  & \le\sum_{b = +,-} \int_{\mathbb{B}_h} \left| \mathds{1}_{\{|\bk-\bK| < \varepsilon\}} (I_{1,b}+I_{2,b}) \right|^2 \md \bk \\
			  & \lesssim_N \int_{|\boldsymbol{\kappa}| < \varepsilon} \varepsilon^2\left\| \boldsymbol{\Gamma} (\varepsilon s) \right\|_{W^{N,1}(\rN^2)}^2 \md \boldsymbol{\kappa} \\
			  & \lesssim_N \varepsilon^4 \left\| \boldsymbol{\Gamma} (\varepsilon s) \right\|_{W^{N,1}(\rN^2)}^2.
	\end{aligned}
\end{equation}
Also combining \eqref{I1Mid} and \eqref{I2} we have
\begin{equation}
	\begin{aligned}
		\|H_2(\varepsilon s)\|_{L^2(\rN^2)}^2 &= \sum_{b = +,-} \int_{\mathbb{B}_h} \left| \mathds{1}_{\{\varepsilon \le |\bk-\bK| <k_0\}} \chi_D(E_b(\bk)) \langle \Phi_b(\bk),F_1(\varepsilon s)\rangle_{L^2(\rN^2)} \right|^2 \md \bk \\
			 & \le \sum_{b = +,-} \int_{\mathbb{B}_h} \left| \mathds{1}_{\{\varepsilon \le |\bk-\bK| <k_0\}} (I_{1,b}+I_{2,b}) \right|^2 \md \bk \\
			 & \lesssim_N \int_{\varepsilon \le |\boldsymbol{\kappa}| < k_0} \left(\varepsilon^{2N}+ \frac{\varepsilon^{2N}}{|\boldsymbol{\kappa}|^{2N-2}} \right) \|\boldsymbol{\Gamma}(\varepsilon s)\|_{W^{N,1}(\mathbb{R}^2)}^2 \md \boldsymbol{\kappa} \\
			 &\lesssim_N \varepsilon^{2N} \|\boldsymbol{\Gamma}(\varepsilon s)\|_{W^{N,1}(\mathbb{R}^2)}^2 + \varepsilon^{2N}\|\boldsymbol{\Gamma}(\varepsilon s)\|_{W^{N,1}(\mathbb{R}^2)}^2\int_\varepsilon^{k_0}  \frac{1}{r^{2N-3}} \md r \\
			 &\lesssim_N \varepsilon^4 \|\boldsymbol{\Gamma}(\varepsilon s)\|_{W^{N,1}(\mathbb{R}^2)}^2,
	\end{aligned}
\end{equation}
provided we choose $N\ge 3$ (because only when $N \ge 3$ can \eqref{I2} be valid). The proof of \Cref{PD0F1Bound}, and hence \Cref{DiracPartBound} is thus complete.

\section{The Singularity at 0}
\label{SecTS0}

In this section we deal with the case $V \equiv 0$ and prove \Cref{MainTheorem2}. In \Cref{SecDFEE}, a basic building block of the theorem is the uniform-in-time boundedness of the wave kernel, which fails in this case, where the treatment will produce an extra linear growth in $t$. We develop another approach for this case. 

According to the notations introduced in \Cref{SecDPSO}, the Bloch mode corresponding to the lowest band should be denoted by $\Phi_1(\by;\bk)$, which unfortunately conflicts with the $E_D$-eigenfunction described in \Cref{DefDiracPoint}. To avoid this clash of notations, throughout and only in this section, \textbf{we use} $\Phi_g(\by;\bk)$ (``$g$" refers to ``ground") \textbf{to indicate the lowest band Bloch mode}. In case of $V \equiv 0$, we have that
\begin{equation*}
	E_b(\bk)=0 \Longleftrightarrow b=1 \text{ and } \bk=0, \Phi_g(\bx;\mathbf{0}) \equiv c.
\end{equation*}
Hence, we can choose a $E_0 < E_D-2\delta$ such that
\begin{equation}
	E_b(\bk) < E_0 \Longrightarrow |\bk-\bK|\ge k_0 \text{ and } b=1.
\end{equation}
Also, define the wave kernel $k(\lambda,t) = \frac{\sin(\sqrt{\lambda}t)}{\sqrt{\lambda}},\lambda>0$ and $k(0,t)=t$.

We still denote the error by $\eta = \varphi - \varphi_{\mathrm{eff}}$, where $\varphi$ is the exact solution to \eqref{ConcernedWaveEquation} and $\varphi_{\mathrm{eff}}$ is the approximate solution given by \eqref{TheAnsatzMR}. The envelop $\boldsymbol{\alpha}$ is now solved by \eqref{DiracEquation} subjected to a Schwartz initial datum $\boldsymbol{\alpha}_0 \in \mathcal{S}$, which is not essential for the theorem itself, but  only simplifies the proof and notation. The equation satisfied by $\eta$ is still \eqref{ErrorEquation}.

\subsection{Separation of the low and high parts}
 Our strategy here is again to decompose the whole system. To do so, we again define our spectral projections, but now the classical ones:
\begin{equation*}
	\Pi_0 := \mathds{1}_{[0,E_0)}\left(\mL^{(0)} \right),\quad \Pi^\bracketeps:= \mathds{1}_{[0,E_0/2)}\left( \mL^\bracketeps \right).
\end{equation*}
These projections help us separate the system into completely decoupled parts, a feature not shared by the ``regularized" ones.

We can then decompose that
\begin{equation*}
	\eta(t) = \Pi^\bracketeps \eta(t) + \left(\Id - \Pi^\bracketeps \right)\eta(t).
\end{equation*}
For the high energy part $\left(\Id - \Pi^\bracketeps \right)\eta(t)$, an exactly the same argument as \Cref{SecDFEE,SecCR} applies and gives the estimate
\begin{equation}\label{BoundOnHighEnergyError}
	\sup_{t \in [0,\rho \varepsilon^{-1}]}\left\| \left(\Id - \Pi^\bracketeps \right)\eta(t) \right\|_{H^{s}} \lesssim_{s,\rho,\boldsymbol{\alpha}_0}  \varepsilon.
\end{equation}

The estimate of the low energy part is more delicate and takes over the main body of this section. Using the spectral cut-off, we can focus on estimating the $L^2$-norm $\left \| \Pi^\bracketeps \eta \right\|_{L^2}$. The error equation is
\begin{equation}\label{LowEnergyErrorEquation}
	\left\{
\begin{aligned}
		\frac{\partial^2 \Pi^\bracketeps\eta}{\partial t^2} + \mL^{(\varepsilon)}\Pi^\bracketeps\eta &= \me^{\mi \sqrt{E_D}t}\Pi^\bracketeps F(\varepsilon t), \\
	\Pi^\bracketeps\eta(\by,0) & = 0, \\
	\frac{\partial \Pi^\bracketeps\eta}{\partial t}(\by,0) &= \Pi^\bracketeps F_0(\by),
\end{aligned}\right.
\end{equation}
The reason of grouping $F=F_1+F_2$ is that they are now at the same order in $\varepsilon$:
\begin{proposition}\label{eps2BoundOnLowEnergyF}
	For every $\rho>0$ and sufficiently small $\varepsilon >0$, there exists a $C_{\rho,\boldsymbol{\alpha}_0}>0$ such that
	\begin{equation}
		\sup_{0\le t\le \rho\varepsilon^{-1}}\left\| \Pi^\bracketeps \partial_T^k F_j(\varepsilon t) \right\|_{L^2} \le C_{\rho,\boldsymbol{\alpha}_0} \varepsilon^2, \quad j,k=0,1,2.
	\end{equation}
\end{proposition}
The proof is a combination of the following two lemmas.

\begin{lemma}
	There exists a constant $C>0$, such that for arbitrary $f \in L^2$ and small enough $\varepsilon>0$, we have
	\begin{equation}
		\| \Pi^\bracketeps f \|_{L^2} \le \| \Pi_0 f\|_{L^2} + C\varepsilon \|f\|_{L^2}.
	\end{equation}
\end{lemma}
\begin{proof}
	We use a smooth switch function to interpolate between $\mathds{1}_{[0,E_0)}$ and $\mathds{1}_{[0,E_0/2)}$. Let $\chi$ be a smooth bump function that equals to $1$ on $[0,E_0/2]$ and is supported in $(-1,E_0)$. Then since both $\mathds{1}_{[0,E_0)} - \chi$ and $\chi - \mathds{1}_{[0,E_0/2)}$ are non-negative (on the spectrum of $\mL^{(\varepsilon)}$), we see that for all $f \in L^2$,
	\begin{equation*}
		\left \langle f, \left[\Pi_0 - \chi\left(\mL^{(0)}\right)\right]f \right \rangle_{L^2} \ge 0 \Longrightarrow \left\|\Pi_0 f \right\|_{L^2} \ge \left\|\chi\left(\mL^{(0)}\right) f \right\|_{L^2}.
	\end{equation*}
	Similarly, 
	\begin{equation*}
		\left\|\Pi^{(\varepsilon)} f \right\|_{L^2} \le \left\|\chi\left(\mL^{(\varepsilon)}\right) f \right\|_{L^2}.
	\end{equation*}
Copying the proof of \Cref{ConvergenceOfSpectralProjections}, we bound the difference between $\left\|\chi\left(\mL^{(0)}\right) f \right\|_{L^2}$ and $\left\|\chi\left(\mL^{(\varepsilon)}\right) f \right\|_{L^2}$ by $C\varepsilon\|f\|_{L^2}$, which is the desired result.
\end{proof}

\begin{lemma}
	For sufficiently small $\varepsilon >0$ and each $n \in \nN$, there exists a $C$ depending only on $\boldsymbol{\alpha}_0,\rho$ and $n$, such that
	\begin{equation*}
		\left\| \Pi_0 \partial_T^k F_j(\varepsilon t) \right\|_{L^2} \le C\varepsilon^n, \quad j,k=0,1,2, \quad  \forall t \le \rho \varepsilon^{-1}.
	\end{equation*}
\end{lemma}
\begin{proof}
	A Floquet-Bloch decomposition holds
\begin{equation*}
	\Pi_0 \partial_T^kF_j(\by,\varepsilon t) = \frac{1}{|\mathbb{B}_h|} \int_{|\bk-\bK|\ge k_0} \mathds{1}_{[0,E_0)}(E_1(\bk)) \left\langle \Phi_g(\bk),\left[ \partial_T^kF_j(\varepsilon s)\right] \right\rangle_{L^2(\rN^2)}\Phi_g(\by;\bk) \md \bk,
\end{equation*}
whose proof is just an imitation of the proof of \Cref{FunctionalCalculusOfFBDecomposition}. Also imitating \Cref{FarMomentaDecay}, we can show that uniformly in $|\bk-\bK|\ge k_0$, 
\begin{equation*}
	\left|\left\langle \Phi_g(\bk),\left[ \partial_T^kF_j(\varepsilon s)\right] \right\rangle_{L^2(\rN^2)} \right| \sim O\left( \varepsilon^\infty \right),
\end{equation*}
which yields the conclusion of the lemma.
\end{proof}

\subsection{Estimating the low energy part}
We solve \eqref{LowEnergyErrorEquation} by means of functional calculus and Duhamel's principle
\begin{equation}\label{LowEnergyInitialVelocityBound}
	\Pi^\bracketeps\eta(t)
	= k(\mL^{(\varepsilon)},t)\Pi^\bracketeps F_0
	+ \int_0^t
	k(\mL^{(\varepsilon)},t-s)
	\me^{\mi\sqrt{E_D}s}\Pi^\bracketeps F(\varepsilon s)\,\md s.
\end{equation}
Here the function $\lambda\mapsto \sin(t\sqrt{\lambda})/\sqrt{\lambda}$ is understood by its continuous extension at $\lambda=0$. The first term is handled easily
\begin{equation*}
	\left\|k(\mL^{(\varepsilon)},t)\Pi^\bracketeps F_0 \right\| \lesssim \varepsilon^2 t,
\end{equation*}
for some constant $C$ depending only on $\boldsymbol{\alpha}_0$. Here and below in this paragraph, the norm is the $L^2$-norm.

In the second term, we cannot directly bound $k(\mL^{(\varepsilon)},t-s)$ by $t-s$, for this will produce a quadratic growth in $t$. The role of \Cref{BoundOnLowEnergySourceTerm} is to replace the direct estimate of the low-energy wave kernel. If one bounds
\[
k(\mL^{(\varepsilon)},t-s)=\frac{\sin((t-s)\sqrt{\mL^{(\varepsilon)}})}{\sqrt{\mL^{(\varepsilon)}}}
\]
directly by \(t-s\), the integral produces a quadratic growth in time, which is too large on the scale \(t=O(\varepsilon^{-1})\). This growth is not intrinsic. To see the underlying cancellation, suppose for a moment that the source is independent of time and restrict the spectral variable to a fixed value \(\lambda\). Then the scalar integral
\[
\int_0^t \frac{\sin((t-s)\sqrt{\lambda})}{\sqrt{\lambda}} e^{\mi\sqrt{E_D}s}\,\md s
=
\frac{
e^{i\sqrt{E_D}t}
-\cos(t\sqrt{\lambda})
-\mi\sqrt{E_D}\dfrac{\sin(t\sqrt{\lambda})}{\sqrt{\lambda}}
}{\lambda-E_D}.
\]
can be computed explicitly. The result has no singularity at \(\lambda=0\), and the growth in $t$ is linear. The apparent singularity of the wave kernel is cancelled by its interaction with the oscillatory factor \(e^{i\sqrt{E_D}s}\). Thus the correct estimate should not be obtained by taking absolute values inside the time integral, but rather by exploiting this cancellation between the wave propagator and the time oscillation. Instead of carrying out this scalar cancellation directly in the spectral representation, we introduce an auxiliary low-energy response, defined in the proof below. This response captures the leading low-energy behavior after the cancellation, and the remaining terms arise only from the slow time dependence of the source. The resulting bound is displayed in the following proposition.

\begin{proposition}\label{BoundOnLowEnergySourceTerm}
	For sufficiently small $\varepsilon > 0$ and $t \le \rho \varepsilon^{-1}$ , there exists a constant $C$, depending only on $\rho$, such that
	\begin{equation*}
		\begin{aligned}
			\Bigg \| \int_0^t
	&k(\mL^{(\varepsilon)},t-s)
	\me^{\mi\sqrt{E_D}s}\Pi^\bracketeps F(\varepsilon s)\,\md s \Bigg \| \\
	&\le C \left( t\sup_{T\le \rho}\left\| \Pi^\bracketeps F(T) \right\| + \varepsilon t^2 \sup_{T\le \rho}\left\| \Pi^\bracketeps \partial_T F(T) \right\| + \varepsilon^2t^2 \sup_{T\le \rho}\left\| \Pi^\bracketeps \partial^2_T F(T) \right\| \right).
		\end{aligned}
	\end{equation*}
\end{proposition}
Combining \Cref{eps2BoundOnLowEnergyF}, \Cref{BoundOnLowEnergySourceTerm}, and \eqref{LowEnergyInitialVelocityBound}, we obtain the bound on the low energy error
\begin{equation*}
	\sup_{t \in [0,\rho\varepsilon^{-1}]} \left\|\Pi^\bracketeps\eta(t)\right\| \lesssim_{\boldsymbol{\alpha}_0} \sup_{t \in [0,\rho\varepsilon^{-1}]}\left(\varepsilon^2 t + \varepsilon^3 t^2 + \varepsilon^4t^2 \right) \lesssim_{\boldsymbol{\alpha}_0,\rho} \varepsilon, 
\end{equation*}
Together with the energy estimate \eqref{BoundOnHighEnergyError}, we complete the proof of \Cref{MainTheorem2}.

\subsection{Proof of \Cref{BoundOnLowEnergySourceTerm}} \label{SubsecProofOfBoundOnLowEnergySourceTerm}
The key is to introduce the auxiliary low-energy response $w$. Define
\begin{equation*}
	w(t) := \me^{\mi \sqrt{E_D} t} \left( \mL^\bracketeps - E_D \right)^{-1}\Pi^\bracketeps F(\varepsilon t),
\end{equation*}
where $\left( \mL^\bracketeps - E_D \right)^{-1}$ is considered as a map $\operatorname{Ran}\Pi^\bracketeps \to \operatorname{Ran}\Pi^\bracketeps$, and is bounded uniformly in $\varepsilon$. For each fixed value of the slow time $T$, replacing $F(\varepsilon t)$ by $F(T)$ in this expression gives the time-harmonic particular solution associated with the projected forcing $\me^{\mi\sqrt{E_D}t}\Pi^\bracketeps F(T)$. The substitution $T=\varepsilon t$ introduces only the slow-time derivative terms displayed below. Indeed, we compute that
\begin{equation*}
\begin{aligned}
	\left( \partial_t^2 + \mL^\bracketeps \right) w(t) &= \me^{\mi \sqrt{E_D} t} \Pi^\bracketeps F(\varepsilon t) \\
			& \quad + 2\mi \sqrt{E_D} \varepsilon \me^{\mi \sqrt{E_D}t} \left( \mL^\bracketeps - E_D \right)^{-1}\Pi^\bracketeps \partial_T F(\varepsilon t) \\
			& \quad + \varepsilon^2 \me^{\mi \sqrt{E_D}t} \left( \mL^\bracketeps - E_D \right)^{-1}\Pi^\bracketeps \partial_T^2 F(\varepsilon t),
\end{aligned}
\end{equation*}
and hence
\begin{equation*}
	\me^{\mi \sqrt{E_D} t} \Pi^\bracketeps F(\varepsilon t) =\left( \partial_t^2 + \mL^\bracketeps \right) w(t) + K_1(t).
\end{equation*}
where $K_1(t)$ satisfies
\begin{equation*}
	\| K_1(t) \| \lesssim \varepsilon \left\| \Pi^\bracketeps \partial_T F(T) \right\| + \varepsilon^2 \left\| \Pi^\bracketeps \partial_T^2 F(T) \right\| \text{ for } t \le \rho \varepsilon^{-1}.
\end{equation*}

Applying the wave propagator and integrating on $[0,t]$, we have
\begin{equation*}
	\int_0^t k(\mL^{(\varepsilon)},t-s) \me^{\mi\sqrt{E_D}s}\Pi^\bracketeps F(\varepsilon s)\,\md s = \int_0^t k(\mL^{(\varepsilon)},t-s) \left( \partial_s^2 + \mL^\bracketeps \right) w(s) \md s+ I_1(t),
\end{equation*}
with
\begin{equation}\label{I1Bound}
	\| I_1(t) \| \lesssim \varepsilon t^2 \sup_{T\le \rho}\left\| \Pi^\bracketeps \partial_T F(T) \right\| + \varepsilon^2 t^2 \sup_{T\le \rho}\left\| \Pi^\bracketeps \partial_T^2 F(T) \right\| \quad t \le \rho \varepsilon^{-1}.
\end{equation}

The integral on the right hand side is the source part of the solution to
\begin{equation*}
	\left\{ \begin{aligned}
		\left( \partial_s^2 + \mL^\bracketeps \right) w(s) &= \left( \partial_s^2 + \mL^\bracketeps \right) w(s) \\
		w(0) &=w(0) \\
		\partial_s w(0) &= \partial_s w(0),
	\end{aligned}\right.
\end{equation*}
so the Duhamel principle gives
\begin{equation*}
	w(t) = \int_0^t k(\mL^{(\varepsilon)},t-s) \left( \partial_s^2 + \mL^\bracketeps \right) w(s) \md s + \cos\left(t\sqrt{\mL^\bracketeps}\right)w(0) + \frac{\sin\left(t\sqrt{\mL^\bracketeps}\right)}{\sqrt{\mL^\bracketeps}}\partial_sw(0).
\end{equation*}
Consequently,
\begin{equation*}
	\begin{aligned}
		\int_0^t &k(\mL^{(\varepsilon)},t-s) \left( \partial_s^2 + \mL^\bracketeps \right) w(s) \md s \\
		&= \me^{\mi \sqrt{E_D} t} \left( \mL^\bracketeps - E_D \right)^{-1}\Pi^\bracketeps F(\varepsilon t) \\
		&\quad - \cos \left( t \sqrt{\mL^\bracketeps} \right) \left( \mL^\bracketeps - E_D \right)^{-1}\Pi^\bracketeps F(0) \\
		&\quad - \frac{\sin\left(t\sqrt{\mL^\bracketeps}\right)}{\sqrt{\mL^\bracketeps}} \left( \mi \sqrt{E_D} \left( \mL^\bracketeps - E_D \right)^{-1}\Pi^\bracketeps F(0) + \varepsilon \left( \mL^\bracketeps - E_D \right)^{-1}\Pi^\bracketeps \partial_T F(0) \right).
	\end{aligned}
\end{equation*}
Finally, directly take $L^2$-norm the both sides of the above equation, and together with \eqref{I1Bound}, we obtain
\begin{equation*}
		\Bigg\| \int_0^t k(\mL^{(\varepsilon)},t-s) \left( \partial_s^2 + \mL^\bracketeps \right) w(s) \md s \Bigg\| 
						\lesssim t\sup_{T\le \rho} \left\| \Pi^\bracketeps F(T) \right \| + \varepsilon t \sup_{T\le \rho} \left\| \Pi^\bracketeps \partial_T F(T) \right \|,
\end{equation*}
for $t \le \rho \varepsilon^{-1}$. This proves \Cref{BoundOnLowEnergySourceTerm}.

\section{Discussions on the Dirac Equation}
\label{SecDDE}
In this section, we discuss some physical intuitions about the Dirac equation \eqref{DiracEquation}, with the help of numerical simulations. Recall that the effective dynamics of the wave-packet is
\begin{equation*}
	\left\{\begin{aligned}
		2\mi \sqrt{E_D}\partial_T\alpha_1 &= \nu_{_F}(\mi \partial_{Y_1}\alpha_2 - \partial_{Y_2}\alpha_2) \\&\quad - \mu\left(\operatorname{Tr}(U\sigma_3)-\mi\operatorname{Tr}(U\sigma_1) \right)\alpha_2 - \left( \xi \operatorname{Tr}(U\sigma_0) + \xi^\# \operatorname{Tr}(U\sigma_2) \right)\alpha_1, \\
		2\mi \sqrt{E_D}\partial_T\alpha_2 &= \nu_{_F}(\mi \partial_{Y_1}\alpha_1 + \partial_{Y_2}\alpha_1) \\&\quad - \bar{\mu}\left(\operatorname{Tr}(U\sigma_3)+\mi\operatorname{Tr}(U\sigma_1) \right)\alpha_1 - \left( \xi \operatorname{Tr}(U\sigma_0) + \xi^\# \operatorname{Tr}(U\sigma_2) \right)\alpha_2, \\
		\alpha_1(\bY,0)&=\alpha_{10}(\bY), \quad\alpha_2(\bY,0)=\alpha_{20}(\bY),
	\end{aligned}\right.
\end{equation*}
which is a generalized form of the Schr\"odinger equation with non-trivial gauge fields. When $\bu \equiv C$, this reduces to the massless Dirac equation \cite{fefferman2014wave}.

Recall from \eqref{DefOfMu} that $\mu$ is generally complex, depending on the phases of $\Phi_1$ and $\Phi_2$. We remark that utilizing another degree of freedom allows one to make both $\nu_{_F}$ and $\mu$ real simultaneously; see \cite{guglielmon2021landau}. When $\mu \in \rN$, the above equation simplifies to
\begin{equation}
\label{SimplifiedDirac}
	\mi\partial_T \boldsymbol{\alpha} = v\left[ (\mi \partial_{Y_1} + A_1)\sigma_1 + (-\mi\partial_{Y_2} - A_2)\sigma_2\right] \boldsymbol{\alpha} + W_{\mathrm{eff}}\sigma_0 \boldsymbol{\alpha},
\end{equation}
where
\begin{equation}
\label{DefPseudoElectromagneticFields}
\begin{aligned}
	v &= \frac{\nu_{_F}}{2\sqrt{E_D}}, \\
	W_{\mathrm{eff}} &= -\frac{1}{2\sqrt{E_D}} \left(\xi \operatorname{Tr}(U\sigma_0) + \xi^\#\operatorname{Tr}(U\sigma_2)\right), \\
	A_1 &= - \frac{\mu}{\nu_{_F}}\mathrm{Tr}(U\sigma_3), \\
	A_2 &=  \frac{\mu}{\nu_{_F}}\mathrm{Tr}(U\sigma_1).
\end{aligned}
\end{equation}
The Hamiltonian
\begin{equation}
\label{OriginalHamiltonian}
	\mathcal{H}:=v\left[ (\mi \partial_{Y_1} + A_1)\sigma_1 + (-\mi\partial_{Y_2} - A_2)\sigma_2\right] + W_{\mathrm{eff}}\sigma_0
\end{equation}
is similar to what is found in \cite{guglielmon2021landau}, apart from a different gauge and a different choice of $U$. The Hamiltonian \eqref{OriginalHamiltonian} is unitarily equivalent, through conjugation by $\sigma_2$, to the following standard form of the Dirac operator
\begin{equation}
\label{StandardHamiltonian}
\begin{aligned}
	\mathcal{D} &= v\left[ (-\mi \partial_{Y_1} - A_1)\sigma_1 + (-\mi\partial_{Y_2} - A_2)\sigma_2\right] + W_{\mathrm{eff}}\sigma_0 \\
				&= v\left[ (\frac{1}{\mi}\nabla_\bY - \mathbf{A}_{\mathrm{eff}})\cdot \boldsymbol{\sigma}\right] + W_{\mathrm{eff}}\sigma_0,
\end{aligned}
\end{equation}
with $\mathbf{A}_{\mathrm{eff}} = (A_1,A_2)$, $\boldsymbol{\sigma}=(\sigma_1,\sigma_2)$. This $\mathcal{D}$ models the behavior of a two-component massless fermion under an external magnetic field (associated with the gauge potential $\mathbf{A}_{\mathrm{eff}}$) and a scalar potential $W_{\mathrm{eff}}$. Here $W_{\mathrm{eff}}$ acts as an electrostatic (or chemical-potential) shift rather than a mass term, since a genuine mass would require a $\sigma_3$-type coupling. 

\begin{remark}
	Honeycomb-based structures provide model platforms for topological materials. Under perturbations, different topological indices are encoded in the Dirac eigenmodes. When the two sides of an interface are subjected to topologically opposite perturbations, this index difference can guarantee a topologically protected edge state localized near the prescribed edge \cite{fefferman2016edge}. This is the principle of the so-called ``bulk-edge correspondence", discussed in numerous literatures such as \cite{drouot2021bulkedge}. As linearized systems, the Dirac equations also serve as convenient models for rigorous analysis of bulk-edge correspondence \cite{bal2019continuous}, hence are of great interest. 
	
	In \cite{xie2019wave}, the authors obtain a Dirac equation with varying mass term of the form $\kappa(\bX)\sigma_3$. Non-zero mass terms can generally open gaps in the spectrum and carry topological phases \cite{bal2019topological}, which help us predict topologically protected edge states. In contrast, the system obtained here is massless. Nevertheless, it is still possible for our model to admit spectral gaps, due to the presence of the gauge potential. For example, the Landau quantization, see below.	
\end{remark}

\subsection{Strain-induced ``Landau levels"}
Furthermore, assume also the material weight $A$ is isotropic so that $\xi^\#=0$, and choose a special $\bu$ to make $\operatorname{Tr}(U\sigma_0) = \nabla \cdot \bu = 0$. These shall result in the most interested case $W_{\mathrm{eff}}=0$. Then we square \eqref{StandardHamiltonian} to find
\begin{equation*}
	\mathcal{D}^2 = v^2\left( \left( \frac{1}{\mi}\nabla_\bY - \mathbf{A}_{\mathrm{eff}} \right)^2\sigma_0 + (\partial_{Y_2}A_1 - \partial_{Y_1}A_2) \sigma_3 \right).
\end{equation*}

Formally putting $\bu(\bY) = (0,\beta Y_1^2)$ and assuming $B_0:=2\mu\beta/\nu_{_F}\ne0$, we can compute by \eqref{DefPseudoElectromagneticFields} that $A_1 = 0, A_2 = B_0 Y_1$, and hence
\begin{equation}
	\mathcal{D}^2 = v^2\left[-\partial_{Y_1}^2\sigma_0 + \left(\frac{1}{\mi} \partial_{Y_2} - B_0Y_1 \right)^2 \sigma_0  -B_0\sigma_3 \right] =v^2 \left( \mathcal{H}_\text{mag}\sigma_0 - B_0\sigma_3 \right),
\end{equation}
where $\mathcal{H}_\text{mag} = -\partial_{Y_1}^2+ \left(\frac{1}{\mi} \partial_{Y_2} - B_0Y_1 \right)^2 $ is a Schr\"odinger operator with constant magnetic field, admitting the so-called Landau levels \cite{landau1977quantum}: $E_n = (2n+1)|B_0|,n=0,1,\dots$, with the corresponding modes normalized along the $Y_1$ direction
 \begin{equation}
\begin{gathered}
\psi_{n,k}(Y_1,Y_2) = \frac{|B_0|^{1/4}}{\pi^{1/4}\sqrt{2^n n!}}e^{\mathrm i k Y_2}\,H_n(\xi)\,e^{-\xi^2/2}, \\
\xi=\sqrt{|B_0|}\!\left(Y_1-\frac{k}{B_0}\right), \quad H_n \text{ is the Hermite polynomials}.
\end{gathered}
\end{equation}

Consequently, $\mathcal{D}$ admits infinitely degenerate levels $E_{n,\pm}=\pm v\sqrt{2n |B_0|}$ for $n=1,2,\dots$ and $E_0 =0$, corresponding to the eigenmodes 
\begin{equation}
\begin{aligned}
	B_0 > 0:\quad \Psi_{n,k,\pm} = \frac{1}{\sqrt{2}} \begin{pmatrix}
		\mi\psi_{n,k} \\
		\pm\psi_{n-1,k}
	\end{pmatrix}, \quad n=1,2,\dots \quad \Psi_{0,k} = \begin{pmatrix}
		\psi_{0,k} \\
		0
	\end{pmatrix}, \\
	B_0 < 0:\quad \Psi_{n,k,\pm} = \frac{1}{\sqrt{2}} \begin{pmatrix}
		\pm\psi_{n-1,k}\\
		\mi\psi_{n,k}
	\end{pmatrix}, \quad n=1,2,\dots \quad \Psi_{0,k} = \begin{pmatrix}
		0 \\
		\psi_{0,k}
	\end{pmatrix}.
\end{aligned}
\end{equation}
\Cref{fig:LandauLevels,fig:LandauEigenmodes} illustrate some of the levels and some eigenmodes. For a fixed sign and $n \in \nN$, any $L^2$-superposition of $\Psi_{n,k,\pm}$ over $k \in \rN$ is an $E_{n,\pm}$-eigenfunction of $\mathcal{D}$. Hence, the density of state at these energy levels are very high.
\begin{figure}[t]
		\centering
 	\includegraphics[width=0.5\linewidth]{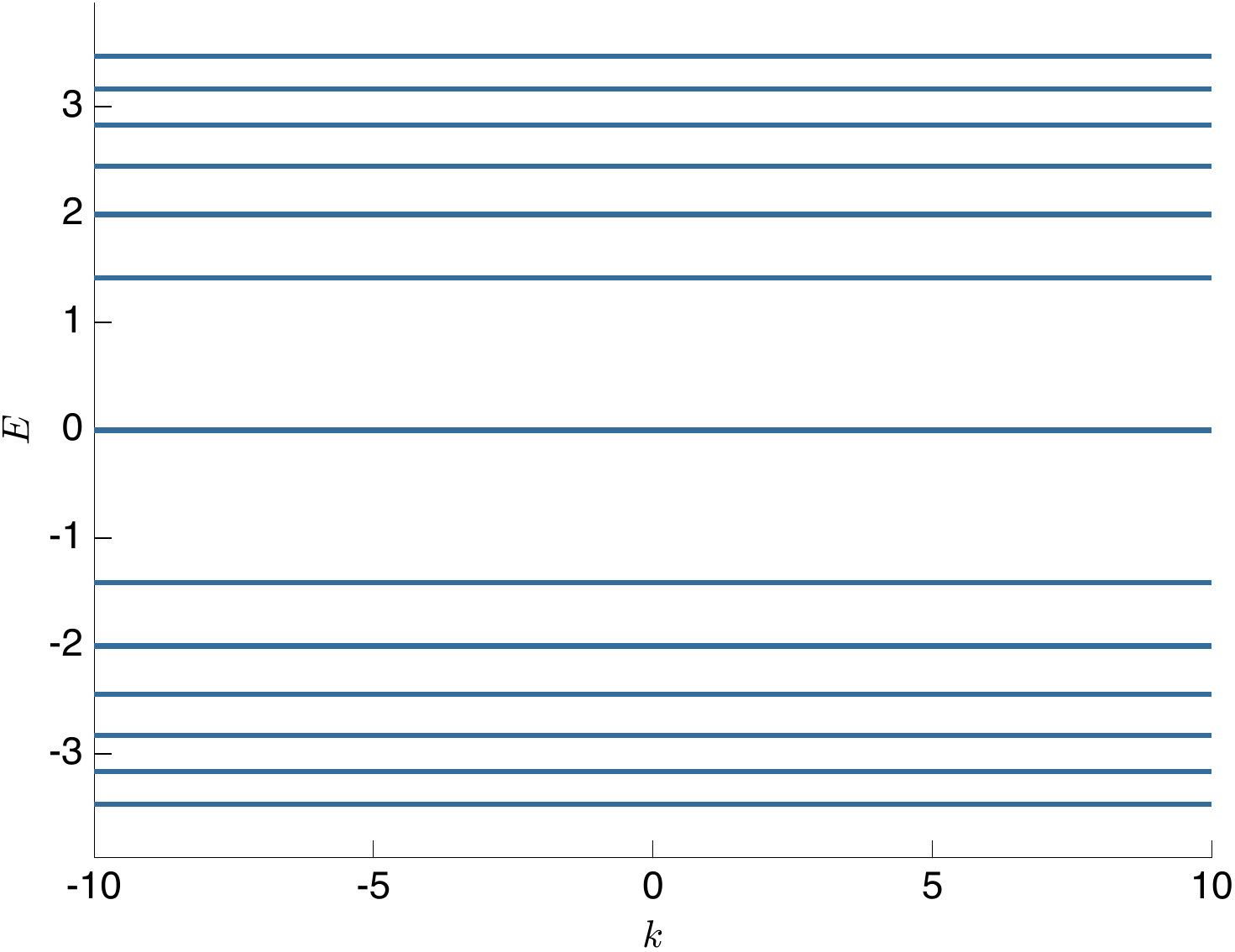}
 	\caption{A schematic of Landau levels $E_{n,\pm}=\pm v\sqrt{2n |B_0|}$ for $v=B_0=1$ and $|n|\le 6$. The solid curves are the energy-momentum curves $E(k)$. Draw a horizontal line for a fixed $E$. If $E=E_{n,\pm}$, then it meets uncountably many points on the solid curves. So the density of state is highly concentrated on these Landau levels}
 	\label{fig:LandauLevels}
\end{figure}
\begin{figure}[t]
	\centering
	\includegraphics[width=\linewidth]{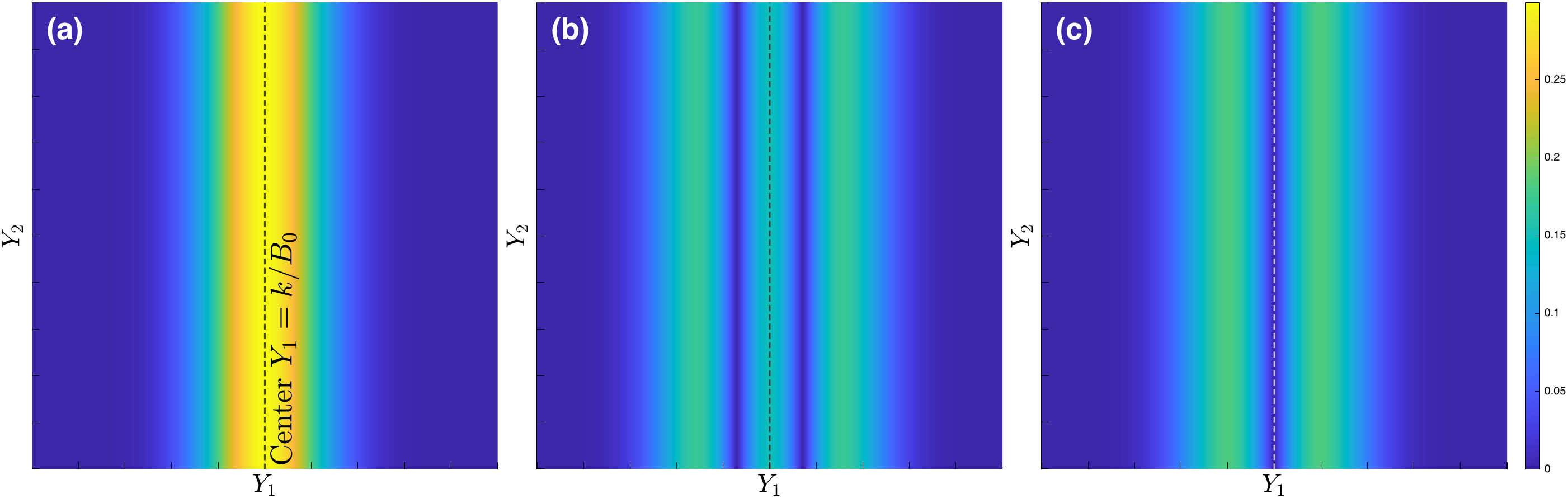}
	\caption{Plots of absolute values of eigenmodes corresponding to Landau levels when $B_0=1$. \textbf{(a)} is for the upper component of $\Psi_{0,k}$, which is $\psi_{0,k}$, while its lower one is $0$. The momentum $k$ only affects the center where it concentrates, and which is labeled in dashed line. \textbf{(b)} and \textbf{(c)} are the upper and lower components of $\Psi_{2,k}$, respectively}
	\label{fig:LandauEigenmodes}
\end{figure}

However, this choice of $\bu$ violates our assumption that $D_\bY\bu \in C_b^\infty$. Despite this, one can still approximate a constant magnetic field on any prescribed finite region by choosing a smooth, compactly supported deformation that agrees with the usual linear gauge inside that region and smoothly vanishes outside. This local approximation is physically natural since experimentally realizable magnetic fields are always of finite spatial extent. The resulting degeneracies are, of course, no longer infinite, but can still be made large.

\subsection{Simulations of Dynamics}
Finally, we simulate the evolutionary equation
\begin{equation}
\label{StandardDiracDynamics}
    \mi \partial_T \boldsymbol{\alpha} = \mathcal{D}\boldsymbol{\alpha}
\end{equation}
for Dirac operators \eqref{StandardHamiltonian}. This equation is unitarily equivalent to \eqref{SimplifiedDirac}.

\paragraph{Landau eigenfunction.} First, put $\bu(\bY) = (0,\beta Y_1^2\chi(Y_1))$ for some huge cut-off $\chi$, so we still have $A_1 = 0, A_2 = B_0Y_1$ near $\mathbf{0}$. Assume $B_0=1$ and take the initial value
\begin{equation}
\label{LandauInitialData}
    \boldsymbol{\alpha}(0) = \begin{pmatrix}
        \displaystyle c\int_\rN G(k)\psi_{0,k}\md k \\
        0
    \end{pmatrix}, \text{ where } G(k) = \me^{-\frac{(k-k_0)^2}{w^2}},
\end{equation}
for some constants $k_0$ and $w$, and $c$ is a normalization coefficient. The results are displayed in \Cref{fig:landau_dynamics}.
\begin{figure}[h]
    \centering
    \includegraphics[width=\linewidth]{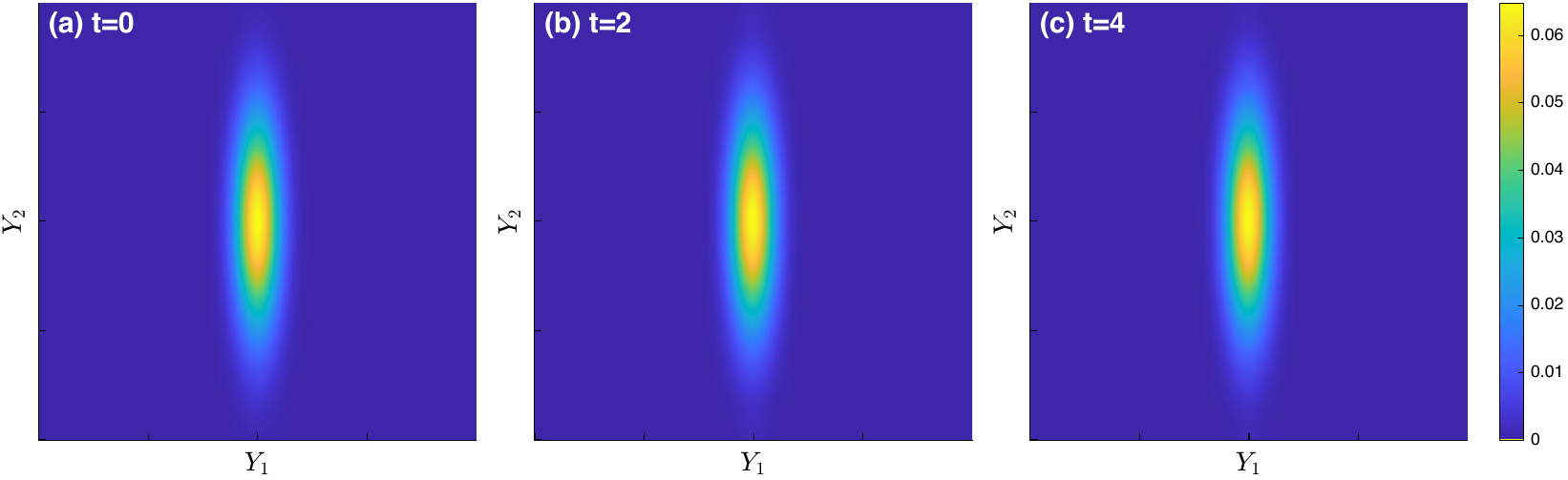}
    \caption{Density distribution of the solution $\boldsymbol{\alpha}$ to \eqref{StandardDiracDynamics} with prescribed $A_1,A_2$, subjected to the initial value \eqref{LandauInitialData} with $k_0 = 0, w= 0.3$. We plot the density of $\boldsymbol{\alpha}(t)$ at $t=0,2,4$ in \textbf{(a)}, \textbf{(b)} and \textbf{(c)}, respectively. The wave-packet remains nearly stationary, in agreement with the fact that the chosen initial value is approximately an eigenfunction of the operator $\mathcal{D}$}
    \label{fig:landau_dynamics}
\end{figure}

\paragraph{Error function}
This time we choose $\bu$ so that 
\begin{equation}
\label{ErfGaugeField}
    A_1=0,\quad A_2 = \mathrm{erf}(Y_1) = \frac{2}{\sqrt{\pi}} \int_0^{Y_1} \me^{-y^2} \md y.
\end{equation}
Then the Dirac operator is given by 
\begin{equation*}
\begin{aligned}
	\mathcal{D}(k) = v\left[ -\mi \partial_{Y_1}\sigma_1 + (-\mi\partial_{Y_2} - \mathrm{erf}(Y_1))\sigma_2\right].
\end{aligned}
\end{equation*}
Fourier transforming $Y_2$ to $k$, we find that
\begin{equation*}
    \mathcal{D} = v\left[ -\mi \partial_{Y_1}\sigma_1 + (k - \mathrm{erf}(Y_1))\sigma_2\right].
\end{equation*}
When $k \in (-1,1)$, this operator has $0$ eigenvalue, whose eigenvector is given by 
\begin{equation*}
\psi_k(Y_1) = 
\begin{pmatrix}
   \displaystyle \exp\left[\int_0^{Y_1} (k-\mathrm{erf}(y))\md y\right]\\
    0
\end{pmatrix}.
\end{equation*}
From this we see that for
\begin{equation}
\label{ErfInitialData}
    \boldsymbol{\alpha}(0) = \int_\rN G(k)\psi_k\md k,
\end{equation}
the wave-packet is also almost stationary, see \Cref{fig:erf_dynamics}.
\begin{figure}[h]
    \centering
    \includegraphics[width=\linewidth]{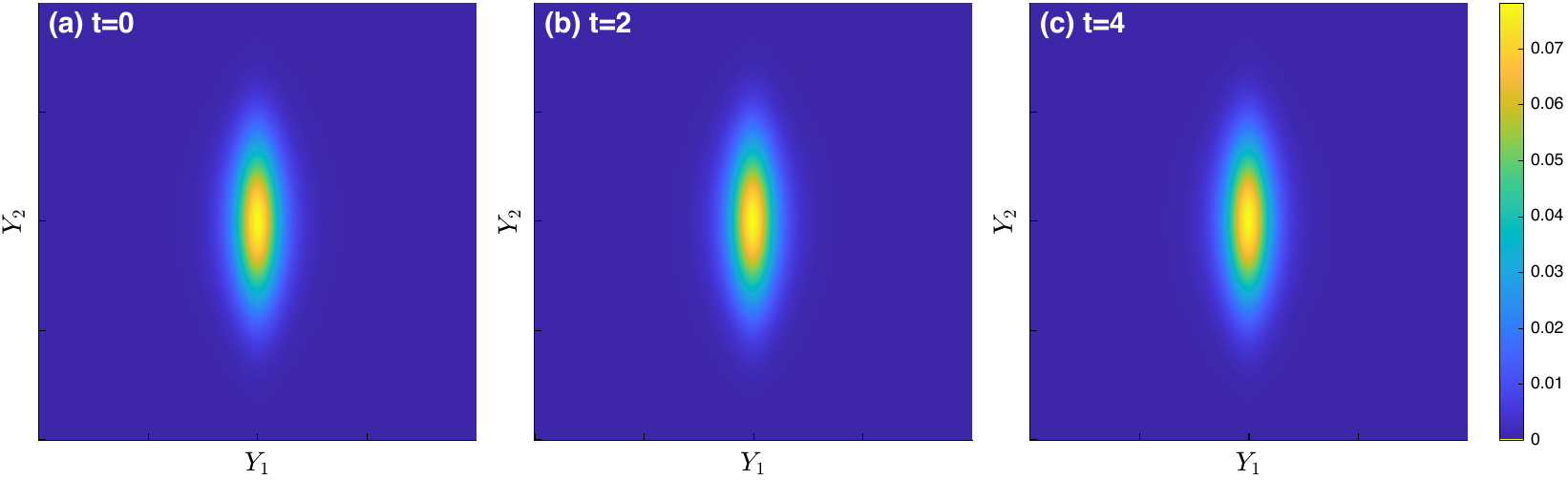}
    \caption{Density distribution of the solution $\boldsymbol{\alpha}$ to \eqref{StandardDiracDynamics} with \eqref{ErfGaugeField}, subjected to the initial value \eqref{ErfInitialData} with $k_0 = 0, w= 0.42$. We plot the density of $\boldsymbol{\alpha}(t)$ at $t=0,2,4$ in \textbf{(a)}, \textbf{(b)} and \textbf{(c)}, respectively. We can see that the wave-packet hardly moves, which agrees with the fact that the initial value we choose is approximately an eigenfunction of the operator $\mathcal{D}$}
    \label{fig:erf_dynamics}
\end{figure}

More generally, we can choose $A_2(Y_1)$ to be any increasing function tending to a negative number at $-\infty$ and a positive number at $+\infty$.

\begin{appendices}

\section{Asymptotic Expansions of the Operator}
\label{AppAEO}
First, we expand $M_\varepsilon(\by) = \left( \Id + \varepsilon U(\varepsilon\by) \right)^{-1}$. Recall the assumption that $U$ is globally bounded. Hence, for sufficiently small $\varepsilon$ we have the Neumann series expansion
\begin{equation}
	M_\varepsilon(\by) = \Id - \varepsilon U(\varepsilon \by) + \sum_{i=2}^\infty \left(-\varepsilon U(\varepsilon \by)\right)^i.
\end{equation}
The summation on the right hand side is a smooth matrix-valued function whose derivatives are bounded by $C\varepsilon^2$. This constant $C$ is independent of $\varepsilon$ and $\by$, but depends on the order of the derivatives.

Using the determinant formula
\begin{equation*}
	\det(\Id + \varepsilon A + \varepsilon^2 B) = 1 + \varepsilon \, \operatorname{Tr}(A) + \varepsilon^2 \left( \operatorname{Tr}(B) + \frac{1}{2} \left[ (\operatorname{Tr}(A))^2 - \operatorname{Tr}(A^2) \right] \right) + O(\varepsilon^3)
\end{equation*}
we obtain an expansion of $J_\varepsilon(\by)= \det(M_\varepsilon(\by))$:
\begin{equation}
\label{ExpansionOfDetJ}
	J_\varepsilon(\by) = 1 - \varepsilon \nabla_\bY \cdot \bu(Y)|_{\bY = \varepsilon \by} + \varepsilon^2 h(\by),
\end{equation}
for some $h$ lying in $C_b^\infty(\rN^2)$ (with bounds independent of $\varepsilon$). We also see that for sufficiently small $\varepsilon$, this determinant is bounded away from $0$, so both $J_\varepsilon(\by)$ and $J_\varepsilon(\by)^{-1}$ lie in $C_b^\infty(\rN^2)$ too.

It follows that, for sufficiently small $\varepsilon>0$, $J_\varepsilon^{-1}M_\varepsilon A M_\varepsilon^T$ is still Hermitian and positively definite. It satisfies the uniformly elliptic condition.
\begin{enumerate}
	\item Given any $\xi \in \cN^2$, we have
	\begin{equation*}
		\left| J_\varepsilon^{-1}\bar{\xi}^TM_\varepsilon A M_\varepsilon^T\xi \right|=\left| J_\varepsilon^{-1}\overline{M_\varepsilon^T\xi}^TA(M_\varepsilon^T\xi) \right| \lesssim \left| M_\varepsilon^T\xi \right|^2 \lesssim |\xi|^2,
	\end{equation*}
	with a constant independent of $\xi,\by$ and $\varepsilon$.
	\item Given any $\xi \in \cN^2\backslash \{0\}$,
	\begin{equation*}
		\left| J_\varepsilon^{-1}\bar{\xi}^TM_\varepsilon A M_\varepsilon^T\xi \right| \gtrsim \left| M_\varepsilon^T\xi \right|^2 \gtrsim |\xi|^2,
	\end{equation*}
	with a constant independent of $\xi,\by$ and $\varepsilon$.
\end{enumerate}
This makes $\mL^{(\varepsilon)}$ a second-order elliptic operator comparable to $-\Delta$.

Next, one obtains the expansion
\begin{equation}
\begin{aligned}
\label{ExpansionOfMepsAMepsT}
	M_\varepsilon(\by)A(\by)M_\varepsilon^T(\by) &= \left( \Id - \varepsilon U(\varepsilon \by) \right)A(\by)\left( \Id - \varepsilon U^T(\varepsilon \by) \right) + O(\varepsilon^2) \\
						 &= A(\by) - \varepsilon \left( U(\varepsilon \by)A(\by) + A(\by)U^T(\varepsilon \by)\right) +\varepsilon^2 r(\by),
\end{aligned}
\end{equation}
where $r \in C_b^\infty(\rN^2;\cN^{2 \times 2})$, with bounds independent of $\varepsilon$.

Now move on to the operator $-J_\varepsilon(\by) \nabla \cdot \frac{M_\varepsilon(\by)A(\by)M_\varepsilon^{T}(\by)}{J_\varepsilon(\by)}\nabla + V(\by)$. Given any $f \in H^2(\rN^2)$, we can expand that
\begin{equation}
\begin{aligned}
\label{FirstDivergenceOperatorExpansion}
		-J_\varepsilon(\by) \nabla \cdot \frac{M_\varepsilon(\by)A(\by)M_\varepsilon^{T}(\by)}{J_\varepsilon(\by)}\nabla f &= -\nabla \cdot (M_\varepsilon A M_\varepsilon^T) \nabla f - J_\varepsilon(\nabla J_\varepsilon^{-1}) \cdot (M_\varepsilon A M_\varepsilon^T \nabla f) \\
					  &=- \nabla \cdot (M_\varepsilon A M_\varepsilon^T) \nabla f + \varepsilon^2 L_r f,
\end{aligned}
\end{equation}
for some first order differential operator $L_r$ whose coefficients are all in $C_b^\infty$, with bounds independent of $\varepsilon$. Here, we used \eqref{ExpansionOfDetJ} to obtain a Neumann expansion of $J_\varepsilon(\by)^{-1}$ and then it is obvious that $\nabla J_\varepsilon^{-1} = O(\varepsilon^2)$ (in the $C_b^\infty$ sense).

Insert \eqref{ExpansionOfMepsAMepsT} into \eqref{FirstDivergenceOperatorExpansion} to get that for all $f \in H^2$:
\begin{equation}
\label{SecondDivergenceOperatorExpansion}
\begin{aligned}
	\bigg[ -J_\varepsilon(\by) &\nabla \cdot \frac{M_\varepsilon(\by)A(\by)M_\varepsilon^{T}(\by)}{J_\varepsilon(\by)}\nabla + V(\by) \bigg] f \\
		  &=- \nabla \cdot (M_\varepsilon A M_\varepsilon^T) \nabla f +V(\by)f+ \varepsilon^2 L_r f \\
		  &= -\nabla \cdot A(\by)\nabla f + V(\by)f \\
		  &\quad + \varepsilon \nabla \cdot \left(U(\varepsilon \by)A(\by) +A(\by)U^T(\varepsilon \by)\right)\nabla f + \varepsilon^2L_r f \\
		  &= \mL^{(0)}f + \varepsilon \nabla \cdot \left(U(\varepsilon \by)A(\by) +A(\by)U^T(\varepsilon \by)\right)\nabla f + \varepsilon^2 L_rf,
\end{aligned}
\end{equation}
where we absorbed $\nabla \cdot r(\by) \nabla$ into $L_r$ and hence now $L_r$ is a second order differential operator whose coefficients are all in $C_b^\infty$,with bounds independent of $\varepsilon$.

In the above expansions, the scales in the $O(\varepsilon)$ term were not completely separated. We need to expand further $\nabla \cdot (AU^T+UA)\nabla f$. We compute that (recall that $\bu = (u_1,u_2)$, and denote $\partial_i f = \partial_{y_i}f(\by), \partial_l a_{ij} = \partial_{y_{_l}}a_{ij}(\by)$ but $\partial_i u_j(\varepsilon \by) = \partial_{Y_i}u_j(\varepsilon \by)$)
\begin{equation}
\label{ExpandAUTUA}
	\begin{aligned}
		\nabla \cdot (A(\by)&U^T(\varepsilon\by) +U(\varepsilon \by)A(\by))\nabla f \\
									   &= \partial_{y_i} \big( a_{ij}\partial_j u_l(\varepsilon\by) \partial_l f + \partial_j u_i(\varepsilon\by) a_{jl}\partial_l f\big) \\
									   &= \partial_{y_i}(a_{ij}\partial_lf)\partial_j u_l(\varepsilon\by) + \varepsilon a_{ij}\partial_lf\partial_{ij}u_l(\varepsilon\by) \\
									   &\quad +\partial_{y_i}(a_{jl}\partial_l f)\partial_j u_i(\varepsilon\by) + \varepsilon a_{jl}\partial_lf\partial_{ij}u_i(\varepsilon\by) \\
									   &= \partial_i a_{ij}\partial_l f \partial_j u_l(\varepsilon\by)  + a_{ij}\partial_{il}f\partial_j u_l(\varepsilon\by) +  \partial_{y_i}(a_{jl}\partial_l f)\partial_j u_i(\varepsilon\by) \\
									   &\quad + \varepsilon a_{ij}\partial_lf\partial_{ij}u_l(\varepsilon\by) + \varepsilon a_{jl}\partial_lf\partial_{ij}u_i(\varepsilon\by) \\
									   &= U^T(\varepsilon \by):(\nabla\cdot A(\by) \otimes \nabla f) + U^T(\varepsilon \by):(A^T\text{Hess }f) + U^T(\varepsilon \by):D(A\nabla f) \\
									   &\quad + \varepsilon a_{ij}\partial_lf\partial_{ij}u_l(\varepsilon\by) + \varepsilon a_{jl}\partial_lf\partial_{ij}u_i(\varepsilon\by).
	\end{aligned}
\end{equation}
Above and below, $a_{ij}$'s are the elements of the matrix $A$ evaluated at $\by$, and $u_i$'s are the components of the deformation $\bu$ evaluated at $\varepsilon \by$, just as $U$ is. We will keep this convention. The $O(1)$ part is exactly $\mathfrak{A}f$ defined by \eqref{CoordinateFreeExpressionFrakA}. To obtain \eqref{SymmetricExpressionFrakA}, we use the expression after the second equality in \eqref{ExpandAUTUA}, followed by a permutation of dummy indices. With this notation, we have
\begin{equation}
	\varepsilon \nabla \cdot (AU^T+UA)\nabla f = \varepsilon \mathrm{Tr}(U\mathfrak{A}f) + \varepsilon ^2\big(a_{ij}\partial_lf\partial_{ij}u_l(\varepsilon\by) +a_{jl}\partial_lf\partial_{ij}u_i(\varepsilon\by) \big).
\end{equation}
Again, absorbing the $O(\varepsilon^2)$ terms into the remainder operator $L_r$ of \eqref{SecondDivergenceOperatorExpansion}, we arrive at
\begin{equation*}
	\widetilde{\mL^{(\varepsilon)}}f (\by) = \mL^{(0)}f + \varepsilon \operatorname{Tr}\left(U(\varepsilon \by)\mathfrak{A}f(\by)\right) + \varepsilon^2 L_r f(\by).
\end{equation*}

Finally, let us move to $\mL^\bracketeps$. By \eqref{ExpansionOfDetJ}, we have
\begin{equation*}
	J_\varepsilon^{1/2}=1-\frac{\varepsilon}{2}\operatorname{Tr}(U(\varepsilon\by))+\varepsilon^2h_1(\by),
	\qquad
	J_\varepsilon^{-1/2}=1+\frac{\varepsilon}{2}\operatorname{Tr}(U(\varepsilon\by))+\varepsilon^2h_2(\by),
\end{equation*}
where $h_1,h_2\in C_b^\infty(\rN^2)$, with bounds independent of $\varepsilon$. Moreover, for $|\alpha|\geq 1$,
\begin{equation*}
	\partial^\alpha_\by J_\varepsilon^{1/2}=O(\varepsilon^2),
	\qquad
	\partial^\alpha_\by J_\varepsilon^{-1/2}=O(\varepsilon^2)
\end{equation*}
in the $C_b^\infty$ sense.

Recall that
\begin{equation*}
	\mL^\bracketeps
		=
			\mathcal J_\varepsilon \widetilde{\mL^\bracketeps}\mathcal J_\varepsilon^{-1},
	\qquad
		\mathcal J_\varepsilon f=J_\varepsilon^{-1/2}f .
\end{equation*}
Hence, for $f\in H^2(\rN^2)$,
\begin{equation*}
	\mL^\bracketeps f
	=
	J_\varepsilon^{-1/2}\widetilde{\mL^\bracketeps}\left(J_\varepsilon^{1/2}f\right).
\end{equation*}
Applying the expansion of $\widetilde{\mL^\bracketeps}$ obtained above, we get
\begin{equation*}
\begin{aligned}
	\mL^\bracketeps f
	&=
	J_\varepsilon^{-1/2}\mL^{(0)}\left(J_\varepsilon^{1/2}f\right)  \\
	&\quad
	+\varepsilon J_\varepsilon^{-1/2}
	\operatorname{Tr}\left(U(\varepsilon\by)\mathfrak{A}\left(J_\varepsilon^{1/2}f\right)\right)
	+\varepsilon^2 J_\varepsilon^{-1/2}L_r\left(J_\varepsilon^{1/2}f\right).
\end{aligned}
\end{equation*}
Since all derivatives of $J_\varepsilon^{1/2}$ and $J_\varepsilon^{-1/2}$ are of order $O(\varepsilon^2)$, commuting these factors with the second-order differential operators only produces $O(\varepsilon^2)$ remainder terms. Therefore,
\begin{equation*}
	J_\varepsilon^{-1/2}\mL^{(0)}\left(J_\varepsilon^{1/2}f\right)
	=
	\mL^{(0)}f+\varepsilon^2L_rf,
\end{equation*}
and
\begin{equation*}
	J_\varepsilon^{-1/2}\mathfrak{A}\left(J_\varepsilon^{1/2}f\right)
	=
	\mathfrak{A}f+\varepsilon^2L_rf.
\end{equation*}
Substituting these two identities into the previous expression and absorbing all higher-order terms into $L_r$, and denoting the resulting operator now by $\mL_r$, we arrive at
\begin{equation*}
	\mL^\bracketeps f(\by)
	=
	\mL^{(0)}f
	+\varepsilon\operatorname{Tr}\left(U(\varepsilon\by)\mathfrak{A}f(\by)\right)
	+\varepsilon^2 \mL_r f(\by),
\end{equation*}
where $\mL_r$ is a second-order differential operator whose coefficients are all in $C_b^\infty$, with bounds independent of $\varepsilon$. This is \eqref{AsymptotcExpansionOfLeps}

\section{Asymptotic Expansion of the Error Equation}
\label{AppAEEE}
To find an equation for $\eta$, we directly plug $ \varphi = \varphi_{\mathrm{eff}}+\eta$ into \eqref{DeformedGoverningEquation} to find that
\begin{equation*}
	\frac{\partial^2 \eta}{\partial t^2} + \mL^\bracketeps \eta =  -\frac{\partial^2 \varphi_{\mathrm{eff}}}{\partial t^2} - \mL^\bracketeps \varphi_{\mathrm{eff}}.
\end{equation*}
We will expand the equation term by terms in the following. First, list the time-derivatives of $\varphi_{\mathrm{eff}}$:
\begin{equation*}
\begin{aligned}
		\frac{\partial \varphi_{\mathrm{eff}}}{\partial t}
		&= \varepsilon \mi \sqrt{E_D}\me^{\mi \sqrt{E_D} t} J_\varepsilon(\by)^{-1/2}\alpha_k \Phi_k  \\
		&\quad + \varepsilon^2\me^{\mi \sqrt{E_D}t} J_\varepsilon(\by)^{-1/2}\partial_T \alpha_k \Phi_k.
\end{aligned}
\end{equation*}
\begin{equation*}
\begin{aligned}
		\frac{\partial^2 \varphi_{\mathrm{eff}}}{\partial t^2}
		&= -\varepsilon E_D \me^{\mi \sqrt{E_D}t} J_\varepsilon(\by)^{-1/2}\alpha_k\Phi_k \\
		&\quad +2\varepsilon^2\mi \sqrt{E_D}\me^{\mi \sqrt{E_D}t} J_\varepsilon(\by)^{-1/2}\partial_T \alpha_k \Phi_k \\
		&\quad +\varepsilon^3 \me^{\mi \sqrt{E_D}t} J_\varepsilon(\by)^{-1/2}\partial^2_T \alpha_k \Phi_k,
\end{aligned} 
\end{equation*}
where $\alpha_k$'s and all its derivatives are evaluated at $(\varepsilon \by,\varepsilon t)$, and $\Phi_k$'s are evaluated at $\by$.

A shorter way to obtain the spatial expansion is to use \eqref{AsymptotcExpansionOfLeps} directly. Put
\begin{equation*}
	\widetilde{\varphi}_{\mathrm{eff}}(\by,t):=J_\varepsilon(\by)^{1/2}\varphi_{\mathrm{eff}}(\by,t)
	=\varepsilon \me^{\mi \sqrt{E_D}t}\alpha_k(\varepsilon\by,\varepsilon t)\Phi_k(\by).
\end{equation*}
Then $\varphi_{\mathrm{eff}}=J_\varepsilon(\by)^{-1/2}\widetilde{\varphi}_{\mathrm{eff}}$. By \eqref{ExpansionOfDetJ}, $J_\varepsilon(\by)^{-1/2}=1+O(\varepsilon)$ and $\partial^\beta_\by J_\varepsilon(\by)^{-1/2}=O(\varepsilon^2)$ for $|\beta|\geq1$. Applying \eqref{AsymptotcExpansionOfLeps} to $J_\varepsilon(\by)^{-1/2}\widetilde{\varphi}_{\mathrm{eff}}$, and absorbing the commutators generated by derivatives of $J_\varepsilon(\by)^{-1/2}$ into the remainder, gives
\begin{equation*}
\begin{aligned}
	\mL^{(\varepsilon)}\varphi_{\mathrm{eff}}
	&=J_\varepsilon(\by)^{-1/2} \mL^{(0)}\widetilde{\varphi}_{\mathrm{eff}}
	+\varepsilon J_\varepsilon(\by)^{-1/2} \operatorname{Tr}\left(U(\varepsilon\by)\mathfrak{A}\widetilde{\varphi}_{\mathrm{eff}}\right)
	+\varepsilon^2 \mL_r\widetilde{\varphi}_{\mathrm{eff}} .
\end{aligned}
\end{equation*}
Now
\begin{equation*}
\begin{aligned}
	\mL^{(0)}\widetilde{\varphi}_{\mathrm{eff}}
	&=\me^{\mi \sqrt{E_D}t}
	\left[\varepsilon\alpha_k\mL^{(0)}\Phi_k
	-\varepsilon^2\nabla_\bY\alpha_k\cdot \mi\mathcal{A}\Phi_k
	+\varepsilon^3 R_1\right],\\
	\varepsilon \operatorname{Tr}\left(U(\varepsilon\by)\mathfrak{A}\widetilde{\varphi}_{\mathrm{eff}}\right)
	&=\me^{\mi \sqrt{E_D}t}
	\left[\varepsilon^2\alpha_k\operatorname{Tr}(U\mathfrak{A}\Phi_k)
	+\varepsilon^3 R_2\right].
\end{aligned}
\end{equation*}
Since $J_\varepsilon(\by)^{-1/2}-1=O(\varepsilon)$, replacing $J_\varepsilon(\by)^{-1/2}$ by $1$ in the $\varepsilon^2$ terms only changes the $\varepsilon^3$ remainder. Therefore
\begin{equation*}
\begin{aligned}
	-\mL^{(\varepsilon)}\varphi_{\mathrm{eff}}
	=\me^{\mi \sqrt{E_D}t}\bigg[
	&-\varepsilon J_\varepsilon(\by)^{-1/2}\alpha_k\mL^{(0)}\Phi_k
	+\varepsilon^2\nabla_\bY\alpha_k\cdot \mi\mathcal{A}\Phi_k \\
	&-\varepsilon^2\alpha_k\operatorname{Tr}(U\mathfrak{A}\Phi_k)
	-\varepsilon^3 R_3\bigg],
\end{aligned}
\end{equation*}
where $R_3$ absorbs $R_1,R_2,\mL_r(\alpha_k\Phi_k)$ and the harmless factors generated by $J_\varepsilon(\by)^{-1/2}-1$. Moreover,
\begin{equation*}
	\|R_3(\varepsilon t)\|_{H^s}
	\lesssim \varepsilon^{-1}\|\alpha_k(\varepsilon t)\|_{H^{s+2}},
\end{equation*}
where the constant is independent of $\varepsilon$. The extra $\varepsilon^{-1}$ comes from the slow-variation of the profiles in spatial variables. With the above calculation, we obtain
\begin{equation*}
\begin{aligned}
		\frac{\partial^2 \eta}{\partial t^2} + \mL^{(\varepsilon)}\eta = \me^{\mi \sqrt{E_D}t} \bigg[ & \varepsilon J_\varepsilon(\by)^{-1/2}\alpha_k \left( E_D - \mL^{(0)}\right)\Phi_k + F_2(\by,\varepsilon t) \\
			  &+\varepsilon^2 \left(-2\mi \sqrt{E_D} \partial_T \alpha_k \Phi_k + \nabla_\bY \alpha_k \cdot \mi \mathcal{A}\Phi_k - \alpha_k \operatorname{Tr}(U\mathfrak{A}\Phi_k) \right) \bigg],
\end{aligned}
\end{equation*}
with initial data 
\begin{equation*}
\left\{
\begin{aligned}
	\eta(\by,0) &= 0, \\
		\frac{\partial \eta}{\partial t}(\by,0) &= -\varepsilon^2 J_\varepsilon(\by)^{-1/2}\partial_T\alpha_k(\varepsilon\by,0)\Phi_k(\by),
\end{aligned}\right.
\end{equation*}
where 
\begin{equation*}
\begin{aligned}
	F_2&:=-\varepsilon^3
	\left( R_3 + J_\varepsilon(\by)^{-1/2}\partial^2_T \alpha_k \Phi_k + R_4\right),\\
	R_4&:=2\mi \sqrt{E_D}\frac{J_\varepsilon(\by)^{-1/2}-1}{\varepsilon}\partial_T\alpha_k\Phi_k,\\
	 \| F_2(\varepsilon t)\|_{H^s} &\lesssim \varepsilon^2\left(\|\boldsymbol{\alpha}(\varepsilon t)\|_{H^{s+2}} + \| \partial^2_T\boldsymbol{\alpha}(\varepsilon t)\|_{H^s}\right).
\end{aligned}
\end{equation*}
Here, the bound for $R_4$ follows from $J_\varepsilon(\by)^{-1/2}-1=O(\varepsilon)$ and the effective Dirac equation. The $O(\varepsilon)$ term is clearly $0$. We further denote that
\begin{equation*}
	 F_1(\by,\varepsilon t) = \varepsilon^2 \left(-2\mi \sqrt{E_D} \partial_T \alpha_k \Phi_k + \nabla_\bY \alpha_k \cdot \mi \mathcal{A}\Phi_k - \alpha_k \operatorname{Tr}(U\mathfrak{A}\Phi_k) \right).
\end{equation*}
\begin{equation*}
	 F_0(\by) = -\varepsilon^2 J_\varepsilon(\by)^{-1/2}\partial_T\alpha_k(\varepsilon\by,0)\Phi_k(\by).
\end{equation*}
Then the error equation becomes
\begin{equation*}
\left\{
\begin{aligned}
	\frac{\partial^2 \eta}{\partial t^2} + \mL^{(\varepsilon)}\eta &= \me^{\mi \sqrt{E_D}t}\left[ F_1+F_2\right], \\
	\eta(\by,0) & = 0, \\
	\frac{\partial \eta}{\partial t}(\by,0) &= F_0.
\end{aligned}\right.
\end{equation*}

\section{Schr\"odinger Case}
\label{AppSE}
In this appendix, we carry the whole procedure to the Schr\"odinger equation. To be more precise, this time we consider the time-dependent Schr\"odinger equation
\begin{equation}
\label{SchrodingerEquation}
\left\{ \begin{aligned}
	\mi \frac{\partial \phi}{\partial t} &= \mL^\bracketeps \phi, \\
						\phi(\by,0)	&= \varepsilon J_\varepsilon(\by)^{-1/2}\left[ \beta_{10}(\varepsilon\by)\Phi_1(\by)+\beta_{20}(\varepsilon\by)\Phi_2(\by) \right].
\end{aligned}\right.
\end{equation}
The ansatz is then
\begin{equation}
	\label{TheSAnsatz}
	\phi_{\mathrm{eff}}(\by,t) = \varepsilon \me^{-\mi E_D t} J_\varepsilon(\by)^{-1/2} \left[ \beta_1(\varepsilon \by,\varepsilon t)\Phi_1(\by) + \beta_2(\varepsilon \by,\varepsilon t)\Phi_2(\by)\right],
\end{equation}
where $\boldsymbol{\beta}=(\beta_1,\beta_2)$ satisfies the following effective equation 
\begin{equation}
\label{SDiracEquation}
	\left\{\begin{aligned}
		\mi \partial_T\beta_1 &= -\nu_{_F}(\mi \partial_{Y_1}\beta_2 - \partial_{Y_2}\beta_2) \\&\quad + \mu\left(\operatorname{Tr}(U\sigma_3)-\mi\operatorname{Tr}(U\sigma_1) \right)\beta_2 + \left( \xi \operatorname{Tr}(U\sigma_0) + \xi^\# \operatorname{Tr}(U\sigma_2) \right)\beta_1, \\
		\mi \partial_T\beta_2 &= -\nu_{_F}(\mi \partial_{Y_1}\beta_1 + \partial_{Y_2}\beta_1) \\&\quad + \bar{\mu}\left(\operatorname{Tr}(U\sigma_3)+\mi\operatorname{Tr}(U\sigma_1) \right)\beta_1 + \left( \xi \operatorname{Tr}(U\sigma_0) + \xi^\# \operatorname{Tr}(U\sigma_2) \right)\beta_2, \\
		\beta_1(\bY,0)&=\beta_{10}(\bY), \quad\beta_2(\bY,0)=\beta_{20}(\bY).
	\end{aligned}\right.
\end{equation}
The constants are exactly the same as in \eqref{DiracEquation}.

We then similarly put $\eta = \phi-\phi_{\mathrm{eff}}$ and look for the equation of $\eta$, which is $\mi \partial_t \eta - \mL^\bracketeps \eta = \mL^\bracketeps \phi_{\mathrm{eff}} - \mi \partial_t \phi_{\mathrm{eff}}$. The right hand side is computed just as in Appendix \ref{AppAEEE}. Put
\begin{equation*}
	\widetilde{\phi}_{\mathrm{eff}}(\by,t):=J_\varepsilon(\by)^{1/2}\phi_{\mathrm{eff}}(\by,t)
	=\varepsilon \me^{-\mi E_D t}\beta_k(\varepsilon\by,\varepsilon t)\Phi_k(\by).
\end{equation*}
Then
\begin{equation*}
\begin{aligned}
	-\mi \partial_t \phi_{\mathrm{eff}}
	&= \me^{-\mi E_D t}\left[
	-\varepsilon E_D J_\varepsilon(\by)^{-1/2}\beta_k \Phi_k
	-\varepsilon^2 \mi J_\varepsilon(\by)^{-1/2}\partial_T \beta_k \Phi_k
	\right], \\
	\mL^{(\varepsilon)}\phi_{\mathrm{eff}}
	&= \me^{-\mi E_D t}\left[
	\varepsilon J_\varepsilon(\by)^{-1/2}\beta_k \mL^{(0)}\Phi_k
	-\varepsilon^2 \nabla_\bY \beta_k \cdot \mi \mathcal{A}\Phi_k
	+\varepsilon^2 \beta_k \operatorname{Tr}(U\mathfrak{A}\Phi_k)
	+\varepsilon^3 R_1
	\right].
\end{aligned}
\end{equation*}
Hence, we obtain that $\eta$ satisfies
\begin{equation*}
\begin{aligned}
	\mi \frac{\partial \eta}{\partial t} - \mL^\bracketeps \eta &= \me^{-\mi E_D t} \Big[ \varepsilon J_\varepsilon(\by)^{-1/2}\beta_k(\mL^{(0)}-E_D)\Phi_k \\
		&\qquad + \varepsilon^2 \left(-\mi \partial_T \beta_k \Phi_k - \nabla_\bY \beta_k \cdot \mi \mathcal{A}\Phi_k + \beta_k \operatorname{Tr}(U\mathfrak{A}\Phi_k) \right) \\
		&\qquad +\varepsilon^3R_2 \Big],
\end{aligned}
\end{equation*}
subjected to a $0$ initial datum $\eta(\by,0)=0$. Here, $R_2$ absorbs $R_1$ and the harmless factors generated by $J_\varepsilon(\by)^{-1/2}-1$, with the same type of Sobolev bounds as the remainders in Appendix \ref{AppAEEE}.

The error estimate is completely analogous to \Cref{SecDFEE} and \ref{SecCR}. The only change is that we now have a better kernel $k(\lambda,t) = \me^{-\mi \lambda t}$, which does not exhibit any singularity on the spectrum. This allows us to remove the assumption $V >0$ or $V \equiv 0$, and avoid the discussion in \Cref{SecTS0}. Consequently, we arrive at
\begin{theorem}
\label{MainTheoremS}
 Suppose $(A,V)$ is a pair of honeycomb media, $(\bK,E_D)$ is a Dirac point, $V$ is bounded by below, and $\bu:\rN^2 \to \rN^2$ is a deformation such that $U=D_\bY\bu \in C_b^\infty$. Assume that $\boldsymbol{\beta}(0)=(\beta_{10},\beta_{20}) \in \mathcal{S}$. Define $\phi_{\mathrm{eff}}$ through \eqref{TheSAnsatz} with $\boldsymbol{\beta} = (\beta_1,\beta_2)$ satisfying the Dirac equation \eqref{SDiracEquation}. Then for any $\rho>0$, the Schr\"odinger equation \eqref{SchrodingerEquation} admits a unique solution $\phi \in C^\infty([0,\rho\varepsilon^{-1}] \times \rN^2)$ satisfying
	\begin{equation}
		\sup_{t \in [0,\rho \varepsilon^{-1}]}\| \phi(t) - \phi_{\mathrm{eff}}(t) \|_{H^s} \le C_{s,\rho} \varepsilon,
	\end{equation} 
	for some constant $C_{s,\rho}$ independent of sufficiently small $\varepsilon>0$.
\end{theorem}

Finally, let us look at the pseudo-magnetic field in this setting. For $\mu \in \rN$, the effective equation \eqref{SDiracEquation} simplifies to
\begin{equation}
	\mi \partial_T \boldsymbol{\beta} = v\left[ (-\mi \partial_{Y_1} - A_1)\sigma_1 + (\mi \partial_{Y_2} + A_2) \sigma_2 \right] \boldsymbol{\beta} + W_{\mathrm{eff}}\sigma_0 \boldsymbol{\beta},
\end{equation}
where
\begin{equation}
\begin{aligned}
	v &= \nu_{_F}, \\
	W_{\mathrm{eff}} &= \xi \operatorname{Tr}(U\sigma_0) + \xi^\#\operatorname{Tr}(U\sigma_2), \\
	A_1 &= - \frac{\mu}{\nu_{_F}}\mathrm{Tr}(U\sigma_3), \\
	A_2 &=  \frac{\mu}{\nu_{_F}}\mathrm{Tr}(U\sigma_1).	
\end{aligned}
\end{equation}
Conjugating by $\sigma_1$, the Hamiltonian is also unitarily equivalent to the standard form
\begin{equation*}
	\mathcal{D}_S = v\left[ (\frac{1}{\mi}\nabla_\bY - \mathbf{A}_{\mathrm{eff}})\cdot \boldsymbol{\sigma}\right] + W_{\mathrm{eff}}\sigma_0.
\end{equation*}
Note that the expression for $\mathbf{A}_{\mathrm{eff}}$ is exactly the same as in the wave equation case, while those for $v$ and $W_{\mathrm{eff}}$ differ only by a factor $(2\sqrt{E_D})^{-1}$, and a sign.

\end{appendices}

\backmatter

\bmhead{Acknowledgements}

This work was supported by the National Key R\&D Program of China (Grant No. 2021YFA0719200).

\section*{Declarations}

\bmhead{Funding}

This work was supported by the National Key R\&D Program of China (Grant No. 2021YFA0719200).

\bmhead{Data availability}

The data that support the findings of this study are available from the authors upon reasonable request.

\bmhead{Competing interests}

The authors declare that they have no competing interests.

\bibliography{main}

\end{document}